\documentclass{amsart}

\usepackage{amsmath}
\usepackage{amssymb}
\usepackage{hyperref}

\newtheorem{theorem}{Theorem}[section]
\newtheorem{lemma}[theorem]{Lemma}
\newtheorem{corollary}[theorem]{Corollary}
\newtheorem{proposition}[theorem]{Proposition}

\theoremstyle{definition}
\newtheorem{definition}[theorem]{Definition}
\newtheorem{example}[theorem]{Example}

\theoremstyle{remark}
\newtheorem{remark}[theorem]{Remark}

\newcommand\R{\mathbb{R}}
\newcommand\Z{\mathbb{Z}}

\newcommand\C{\mathbb{C}}
\newcommand\N{\mathbb{N}}
\newcommand{\DNLS}{\mathrm{DNLS}}
\newcommand{\diff}{\mathrm{diff}}

\newcommand\op{\mathrm{op}}
\newcommand\hs{\mathfrak I_2}

\newcommand{\qtq}[1]{\quad\text{#1}\quad}
\newcommand\Schw{\mathcal{S}}
\newcommand\eps{\varepsilon}
\newcommand{\vk}{\varkappa}

\newcommand{\sbrack}[1]{^{[#1]}}

\let\Re=\undefined\DeclareMathOperator{\Re}{Re}
\let\Im=\undefined\DeclareMathOperator{\Im}{Im}
\DeclareMathOperator{\Id}{Id}
\DeclareMathOperator{\cosec}{cosec}
\DeclareMathOperator{\sech}{sech}

\DeclareMathOperator{\tr}{tr}
\DeclareMathOperator{\sgn}{sgn}

\newcommand{\p}{\partial}
\newcommand{\LHS}[1]{\mathrm{LHS}\eqref{#1}}
\newcommand{\RHS}[1]{\mathrm{RHS}\eqref{#1}}

\newcommand{\norm}[1]{{\left\vert\kern-0.25ex\left\vert\kern-0.25ex\left\vert #1\right\vert\kern-0.25ex\right\vert\kern-0.25ex\right\vert}}
\newcommand{\I}{\mathfrak I}
\usepackage{bbm}
\newcommand{\bbo}{\mathbbm{1}}
\newcommand{\loc}{\mathrm{loc}}

\numberwithin{equation}{section}

\allowdisplaybreaks

\begin{document}

\title[DNLS is well-posed in $L^2(\R)$]{Global well-posedness for the derivative nonlinear Schr\"odinger equation in $L^2(\R)$}

\author[B.~Harrop-Griffiths]{Benjamin Harrop-Griffiths}
\address{Department of Mathematics, University of California, Los Angeles, CA 90095, USA}
\email{harropgriffiths@math.ucla.edu }

\author[R.~Killip]{Rowan Killip}
\address{Department of Mathematics, University of California, Los Angeles, CA 90095, USA}
\email{killip@math.ucla.edu}

\author[M.~Ntekoume]{Maria Ntekoume}
\address{Department of Mathematics, Rice University, Houston, TX 77005-1892, USA}
\email{maria.ntekoume@rice.edu}

\author[M.~Vi\c{s}an]{Monica Vi\c{s}an}
\address{Department of Mathematics, University of California, Los Angeles, CA 90095, USA}
\email{visan@math.ucla.edu}

\begin{abstract}
We prove that the derivative nonlinear Schr\"odinger equation in one space dimension is globally well-posed on the line in $L^2(\R)$, which is the scaling-critical space for this equation.  
\end{abstract}

\maketitle

\tableofcontents

\section{Introduction} \label{sec;introduction}

The derivative nonlinear Schr\"odinger equation
\begin{equation}\tag{DNLS}\label{DNLS} 
    i\partial_t q + q'' +i \bigl(|q|^2 q\bigr)'=0  
\end{equation}
describes the evolution of a complex-valued field $q$ defined on the line $\R$.  Here and below, primes indicate spatial derivatives.

Physical applications of \eqref{DNLS} are reviewed briefly in subsection~\ref{SS:gauge} below.  There, we also discuss certain well-documented changes of variables that convert \eqref{DNLS} to other evolutions of interest in the physical sciences. 

One of the most basic questions we should ask of any model is whether it is well-posed: Do solutions exist? Are they unique?  Do they depend continuously on the initial data?  Without such properties, it is unclear whether the model is capable of making experimentally falsifiable predictions.  The well-posedness question also forms an important benchmark in our understanding of an equation.  Gaps between well- and ill-posedness results leave open the possibility that there are basic physical processes --- instabilities and/or stabilizing mechanisms --- that remain undiscovered.

The principal goal of this paper is to show that \eqref{DNLS} is globally well-posed in $L^2(\R)$:

\begin{theorem}\label{t:main} The \eqref{DNLS} evolution is globally well-posed in \(L^2(\R)\). 
Precisely, there is a jointly continuous map \(\Phi:\R\times L^2(\R)\to L^2(\R)\) that agrees with the data-to-solution map when restricted to Schwartz-class initial data.
\end{theorem}

It is not wanton abstraction to define the data-to-solution map as an extension from Schwartz-class initial data; indeed, this is the textbook approach to defining the Fourier transform on $L^2(\R)$ and is widespread in nonlinear PDE.  The heart of the matter is to prove key metric properties that allow one to extend the mapping to general elements of $L^2$ and then to ensure that the extension retains the many good properties of its Schwartz-class restriction.

A relatively small fraction of this paper would suffice to show that $L^2$-precompact sets of Schwartz initial data are mapped under the flow to $C_t([-T,T];L^2(\R))$-precompact sets of orbits. (Here $T>0$ must be finite, but is otherwise arbitrary.)  This would be a new result and it trivially yields the existence of solutions; however, it goes no way to justifying uniqueness, nor continuous dependence on the initial data.  This thinking helps us appreciate the uniqueness statement embedded in Theorem~\ref{t:main}: no matter how we approximate an $L^2$ initial data by a sequence of Schwartz initial data, the corresponding trajectories will converge and they will converge to the same limit!  In particular, our solutions have the group property.

Theorem~\ref{t:main} implicitly asserts that Schwartz initial data lead to global unique solutions.  This is true.  While uniqueness of smooth solutions is easily verified via the Gr\"onwall inequality, the existence of global solutions for large Schwartz-class initial data is, in fact, a very recent result! (See the discussion in subsection \ref{SS:prior}.)  Although the existence of such solutions is not a prerequisite for our methods, building on this result leads to a much clearer exposition.  Moreover, without the triumphs of these authors, which we celebrate in subsection~\ref{SS:prior}, we would not have had the courage to pursue the results of this paper.

Next, we wish to discuss why Theorem~\ref{t:main} considers initial data in $L^2(\R)$.  There are several reasons that make $L^2(\R)$ a natural space in which to study \eqref{DNLS}.  First, the $L^2$-norm is conserved by the flow; indeed, we have the microscopic conservation law
\begin{align}\label{micro M}
    \partial_t  |q|^2+ \partial_x \bigl[2\Im (q' \bar q) +\tfrac{3}{2} |q|^4 \bigr]=0 . 
\end{align}
Second, it is a scale-invariant space: if $q(t,x)$ is a smooth solution to \eqref{DNLS}, then so too is 
\begin{align}\label{scaling}
q_\lambda(t,x) :=  \sqrt{\lambda} \,q(\lambda^2 t,\lambda x)  
\end{align}
for every $\lambda>0$.  Notice that this transformation does not affect the $L^2$ norm of the initial data (nor indeed at any later time).

Long-standing physical intuition dictates that dispersive equations will be ill-posed below the scaling-critical regularity.  For the case of \eqref{DNLS}, this is justified by the self-similar solutions constructed in \cite{MR4079021,MR818186}.  Beyond proving that well-posedness fails in $H^s(\R)$ with $s<0$, these solutions even show that it fails in weak-$L^2$, which is a scale-invariant space!

It is a matter of some pride for us that we are able to treat \eqref{DNLS} in the most natural scale-invariant space.  It is only quite recently that global-in-time solutions could be constructed for large data in scaling-critical spaces for any kind of dispersive PDE.   Moreover, we are dealing with a focusing nonlinearity (e.g. soliton solutions abound).  Many focusing dispersive equations do not admit large-data global solutions; it is typical for wave collapse to occur above a certain threshold size (as measured in scaling-critical spaces).

There is an obvious scapegoat here: \eqref{DNLS} is completely integrable, \cite{MR464963}.  However scaling-critical well-posedness does not seem to be the norm for such models: it fails for KdV, NLS, and mKdV!  The phenomenology of \eqref{DNLS} becomes even more curious when we endeavor to find a quantitative expression of the continuous dependence of the solution on its initial data.  As discussed below, we know that when $s<\frac12$, the data-to-solution map cannot be uniformly continuous on \emph{any} neighborhood of the origin in $H^s(\R)$.  This is quite different from the behavior of the mass- or energy-critical NLS, for example, where the data-to-solution map is real analytic (cf. \cite{MR3098643}).

It is perhaps better to compare \eqref{DNLS} to other models with derivative nonlinearity.  For the notoriously difficult two-dimensional wave maps equation, for example, Tataru \cite{MR2130618} proved that the data-to-solution map (defined on scaling-critical balls) is Lipschitz in lower regularity norms.  We will show in Proposition~\ref{P:noGron} that this fails for \eqref{DNLS} --- again \eqref{DNLS} appears \emph{less} continuous! Complete integrability, it seems, is not a stern parent that keeps its flows safe and orderly; rather it is permissive and allows its solutions to become quite wild before issuing the rebuke of ill-posedness. 

Earlier we singled out the $H^s(\R)$ family of spaces in our discussion of well- and ill-posedness.  Already the natural prerequisite that the \emph{linear} Schr\"odinger equation be well-posed is quite restrictive; this precludes the consideration of $L^p$-based Sobolev spaces with $p\neq 2$.  As we will see below, $H^s(\R)$ spaces are both the most natural and most studied classes of initial data; indeed, they arise from the consideration of conservation laws for \eqref{DNLS}.  Building on Theorem~\ref{t:main}, we will prove 

\begin{corollary}\label{C:main} \eqref{DNLS} is globally well-posed in \(H^s(\R)\) for every $s\geq 0$. 
\end{corollary}

Prior work in this direction is discussed at length in the next subsection.  It is evident that Theorem~\ref{t:main} guarantees the existence and uniqueness of solutions for data in $H^s$ with $s\geq 0$.  That such solutions remain bounded in $H^s$ is known as \emph{persistence of regularity} and may be deduced as a consequence of conservation laws.  To complete the proof of well-posedness, one must upgrade continuous dependence from the $L^2$ metric to the $H^s$ metric.  As demonstrated in several prior works of the authors \cite{MR4304314, harropgriffiths2020sharp, KNV, MR3990604}, this is easily done if one can verify that $H^s$-equicontinuous sets of initial data lead to $H^s$-equicontinuous ensembles of orbits.

\begin{definition}
A subset $Q$ of $H^s(\R)$ is \emph{$H^s(\R)$-equicontinuous} if for every $\eps>0$ there is a $\delta>0$ so that
$$
\sup_{q\in Q}\, \sup_{|y|<\delta} \, \bigl\| q(x+y) - q(x) \bigr\|_{H^s_x} < \eps.
$$
\end{definition}

This definition extends naturally to any translation-invariant Banach space of functions on any group.  For bounded continuous functions on $\R^d$, one recovers the notion of equicontinuity familiar from the Arzel\`a--Ascoli Theorem.  Indeed, this more general notion of equicontinuity was introduced precisely to formulate the analogous compactness theorem in $L^p(\R^d)$ spaces; see \cite{MRiesz}.

We will also need the second key requirement for compactness, albeit only in the $L^2$ setting:

\begin{definition}\label{D:tight}
We say that $Q\subseteq L^2(\R)$ is \emph{tight} if $$
\limsup_{R\to\infty} \ \sup_{q\in Q}\, \int_{|x|\geq R} |q(x)|^2\,dx =0.
$$
\end{definition}

The transportation of $L^2$ norm is expressed by \eqref{micro M}.  As is characteristic of dispersive equations, we see that the flux of the conserved quantity involves more derivatives than the conserved quantity itself.  While this is an obstacle in our path to proving tightness, it is also the key property of microscopic conservation laws that provides for local smoothing estimates.

To formulate local smoothing estimates, we must first agree on how to localize the solution in space.  We will do this through the Schwartz-class function
\begin{equation}\label{psi}
\psi(x) := \sqrt{\sech(\tfrac x{99})} \qtq{and its translates} \psi_\mu(x) := \psi(x - \mu).
\end{equation}
There is nothing terribly special about this choice.  The fact that it has slow exponential decay (relative to unit scale) is quite convenient; beyond this, it is merely the case that this choice has served us well in the prior work \cite{harropgriffiths2020sharp}.

\begin{theorem}[Local smoothing]\label{t: local smoothing}
Let $Q\subset \Schw(\R)$ be both $L^2$-bounded and equicontinuous.  For each $T>0$, solutions $q(t)$ to \eqref{DNLS} with initial data $q(0)\in Q$ satisfy
\begin{align}\label{E:naiveLS}
\smash[t]{   \sup_{\mu\in \R}\ \int_{-T}^T\|\psi_\mu^{12} q(t)\|_{H^\frac12}^2\,dt\lesssim_{T,Q} \|q(0)\|_{L^2}^2.   }
\end{align}
\end{theorem}

\begin{corollary}\label{C:distrib}
The solutions constructed in Theorem~\ref{t:main} are distributional solutions; indeed, the data-to-solution map is continuous as a mapping of $L^2(\R)$ into \(L^3_{\mathrm{loc}}(\R\times\R)\). 
\end{corollary}

The first striking thing about Theorem~\ref{t: local smoothing} is the fact that the estimate is only claimed for equicontinuous sets, not balls.  This is of necessity, as we will show in Proposition~\ref{P:no smoothing}, and reiterates the non-perturbative nature of our analysis.

Other than proving Corollary~\ref{C:distrib}, Theorem~\ref{t: local smoothing} will play no role in the analysis.  It is not strong enough!  For example, it is not sufficient to prove tightness.  For that purpose, we will need the stronger estimate \eqref{refined local smoothing} expressed in terms of our local smoothing spaces $X^{1/2}_\kappa$ introduced in subsection~\ref{S:ls}.  In essence, these spaces capture the local smoothing norm living at frequencies $|\xi|\geq \kappa$.  In this way, \eqref{refined local smoothing} expresses that there is little local smoothing norm at high frequencies and consequently, little transportation of the $L^2$ norm by the high frequencies.


\subsection{Prior work}\label{SS:prior} Local well-posedness of \eqref{DNLS} was first proved in $H^s(\R)$ for $s>\tfrac32$ via energy methods in \cite{MR621533,MR634894}.  Subsequently, this was improved to $s\geq \frac12$ by Takaoka in \cite{MR1693278} via contraction mapping in $X^{s,b}$ spaces.  The solution so constructed is a real-analytic function of the initial data.

As explained in \cite[\S7]{MR1836810}, the results of \cite{MR1693278} show that the data-to-solution map cannot be real-analytic (or even $C^3$) on $H^s(\R)$ for any $s<\frac12$.  Indeed, by analyzing the family of solitons reviewed in subsection~\ref{SS:solitons}, it was shown in \cite{MR1837253} that the data-to-solution map cannot even be uniformly continuous (on bounded sets) in $H^s(\R)$ for $s<\frac12$.

Contraction mapping arguments have also been applied in other function spaces.  Local well-posedness of \eqref{DNLS} in certain Fourier--Lebesgue and modulation spaces was shown in \cite{MR2181058} and \cite{MR4256461}, respectively.  In both cases, the spaces are based on $s\geq \frac12$ number of derivatives.  It is noted in \cite{MR4256461} that if fewer derivatives are used, the data-to-solution map cannot be smooth.

We discussed earlier how this irregularity of the data-to-solution map challenges the naivest notions of complete integrability.  It also has profound implications in terms of methods.  For a generation now, work on well-posedness problems for dispersive PDE has been dominated by contraction mapping arguments in increasingly sophisticated spaces, employing ever subtler harmonic analysis tools.  By their very nature, solutions built by contraction mapping will be real-analytic functions of their initial data.  The poor regularity of the data-to-solution map in the setting of Theorem~\ref{t:main} is a strong signal that very different methods will be needed.

Let us turn now to the question of \emph{global} well-posedness. The standing paradigm here is to extend local-in-time results by employing exact (or approximate) conservation laws.  As a completely integrable system, \eqref{DNLS} has a multitude of exact conservation laws.  The most basic three are
\begin{align}
M(q) &:= \int |q(x)|^2\,dx,  \\
H(q)&:=-\tfrac{1}{2}\int i (q \bar q'-\bar q q')+|q|^4 \, dx, \label{H}\\
H_2(q)&:=\int |q'|^2 +\tfrac{3}{4} i |q|^2 (q \bar q'-\bar q q') + \tfrac 12 |q|^6\, dx.
\end{align}
The functional $H(q)$ will be called the Hamiltonian since it generates the \eqref{DNLS} dynamics in concert with the Poisson structure
\begin{align}\label{Poisson}
\{ F , G \} := \int \tfrac{\delta F}{\delta q} \bigl(\tfrac{\delta G}{\delta\bar q}\bigr)'
		+ \tfrac{\delta F}{\delta \bar q} \bigl(\tfrac{\delta G}{\delta q}\bigr)'\,dx .
\end{align}

The problem with these conservation laws is that they are \emph{not} coercive for large initial data, specifically, when $M(q)$ is large.  It was observed in \cite{MR1152001} that coercivity does hold if $M(q)<2\pi$, which was used to obtain global well-posedness in $H^1(\R)$ under this restriction.  The subsequent works \cite{MR1871414, MR1950826,MR2823664,MR1836810} ultimately led to $H^s$-well-posedness for $s\geq \frac12$ under the $M(q)<2\pi$ restriction.

Later, Wu showed that the $2\pi$ barrier was illusory and that a priori bounds could be obtained under the weaker restriction $M(q)<4\pi$; see \cite{MR3198590,MR3393674} and \cite{MR3668587}.  Global $H^s$-well-posedness for $s\geq\frac12$ and $M(q)<4\pi$ was then shown in \cite{MR3583477}.

The $4\pi$ barrier is certainly not illusory: Algebraic solitons (see \eqref{alg s}) are explicit solutions of \eqref{DNLS} with $M(q)=4\pi$, but for which all other polynomial conservation laws vanish.  Applying the symmetry \eqref{scaling} to these algebraic solitons, we see that the polynomial conservation laws alone cannot provide the kind of control that is needed; see the discussion surrounding \eqref{coercive fails}.  Because of such obstructions, the behavior of large-data solutions to \eqref{DNLS} was for a long time a terra incognita.

The first definitive evidence that large data do not blow up was provided via the inverse scattering approach; see \cite{MR3913998,MR4042219,MR4149070,MR3706093,MR3563476,MR3702542,MR3862117,saalmann2017global}.  Among these works, we wish to single out \cite{MR4149070} as not only constructing solutions (without any spectral hypotheses), but also for proving continuous dependence on the initial data.  Concretely, they proved that \eqref{DNLS} is globally well-posed in $H^{2,2}(\R)=\{f\in H^2 : x^2 f\in L^2\}$.  Combined with the local-in-time arguments in \cite{MR1152001}, this result shows that \eqref{DNLS} is globally well-posed in Schwartz space.

Strong spatial decay requirements are a prerequisite for the inverse scattering approach as we understand it today.  Currently, there is no satisfactory theory of forward nor inverse scattering in any $H^s(\R)$ space (not only for \eqref{DNLS}, but also for KdV, NLS, and mKdV).  On the other hand, one of the major strengths of the inverse scattering method is its ability to describe the long-time behavior of solutions.  Indeed, a soliton resolution result for generic data in $H^{2,2}(\R)$ was proved in \cite{MR3858827}; see also \cite{MR4042219,MR3739932}.

The large-data impasse in Sobolev spaces was dramatically broken by Bahouri and Perelman in the recent paper \cite{BP}.  By synthesizing the existing well-posedness theory with an in-depth analysis of the transmission coefficient, they proved that \eqref{DNLS} is globally well-posed in $H^{1/2}(\R)$. 

For what follows, it is more convenient to discuss the reciprocal of the transmission coefficient and to define this quantity, $a(k;q)$, via a Fredholm determinant. For $\kappa>0$, we first define
\begin{align}\label{Lambda}
    \Lambda(\kappa;q) :=(\kappa-\partial)^{-\frac 12} q (\kappa+\partial)^{-\frac 12} \qtq{and}
    	\Gamma(\kappa;q):=(\kappa+\partial)^{-\frac 12} \bar q (\kappa-\partial)^{-\frac 12},
\end{align}
which extend to $\kappa<0$ via $\Lambda(-\kappa;q) = - \Gamma(\kappa;\bar q)$.  By Lemma~\ref{L:HS} below (reproduced from \cite{MR3820439}), both $\Lambda$ and $\Gamma$ are Hilbert--Schmidt operators; thus we may define
\begin{align}\label{a defn}
a(i\kappa ;q) := \det\bigl[ 1 - i\kappa \Lambda\Gamma\bigr] .
\end{align}

This expression originates from a perturbation determinant based on the Lax pair for \eqref{DNLS} discovered in \cite{MR464963}.  In particular, it is conserved under the \eqref{DNLS} flow; see \cite{klaus2020priori} for a proof of this.

It follows from \eqref{a defn} that $k\mapsto a(k;q)$ extends to a holomorphic function in both the upper and lower half-planes.  While this extension was essential for \cite{BP} and the paper \cite{HGKV:equi} that we will discuss shortly, we will not need this here and so restrict our attention to the case where $k=i\kappa$ is purely imaginary.

The central problem overcome by \cite{BP} was the ineffectual nature of the conservation laws attendant to \eqref{DNLS}; however, this solution did not provide new conservation laws with which to fill the void.  In particular, \cite{BP} does not provide a priori control on lower regularity norms, nor the means to address the question of equicontinuity in such spaces.

Using the ideas of \cite{BP} as a jumping-off point, the paper \cite{HGKV:equi} shows that \eqref{DNLS} does preserve $L^2$-equicontinuity.  Note that this assertion takes the form of an a priori bound on Schwartz-class solutions since solutions were not known to exist for merely $L^2$ initial data.  In fact (and this will be important for us), the paper \cite{HGKV:equi} shows that this equicontinuity property is enjoyed by any flow preserving the perturbation determinant \eqref{a defn} and so by the \emph{entire} \eqref{DNLS} hierarchy:

\begin{theorem}[\cite{HGKV:equi}]\label{T:HGKV:equi}
Let $Q\subseteq \Schw(\R)$ be \(L^2\)-bounded and equicontinuous.  Then
\begin{align}\label{QsI}
Q_* = \bigl\{q\in \Schw(\R): a(k;q)\equiv a(k;\tilde q) \text{ for some $\tilde q\in Q$}\bigr\}
\end{align}
is also \(L^2\)-bounded and equicontinuous.
\end{theorem}

A key motivation for addressing the equicontinuity question in \cite{HGKV:equi} is that it unlocks a large number of tools in the study of \eqref{DNLS}; it was this realization that lead \cite{KNV} to the explicit formulation of this equicontinuity problem.   

The first tools unlocked by the equicontinuity property are low-regularity conservation laws, specifically conservation laws at the level of $H^s$ for $0<s<\frac12$.  Such laws were first derived in \cite{klaus2020priori} following the approach of \cite{MR3820439}; however, they were only applicable to small solutions.  The realm of applicability was first raised to $M(q)<4\pi$ in \cite{KNV} by proving equicontinuity in that regime and then to arbitrarily large solutions in \cite{HGKV:equi}.

Equicontinuity also unlocks higher regularity conservation laws for large data.  In \cite{KNV}, $L^2$-equicontinuity and the conservation of $H_2(q)$ are shown to provide global $H^1(\R)$ bounds.  In \cite{BLP}, the result of \cite{HGKV:equi} is used as the base step of an inductive argument to cover $H^s(\R)$ spaces for all $s\geq \frac12$.  This brings closure to the question of coercive conservation laws: we now know that $H^s$-bounded sets of Schwartz-class initial data lead to $H^s$-bounded solutions for all $s\geq 0$.

To prove local smoothing, we need \emph{microscopic} conservation laws such as \eqref{micro M}, rather than mere conserved quantities.  Note that \eqref{micro M} itself is useless for this purpose because the current is not coercive.  Already for the proof of \eqref{E:naiveLS}, we need scaling-critical coercive microscopic conservation laws and the full proof of well-posedness will require even more subtle estimates.


Just such microscopic conservation laws were worked out in \cite{tang2020microscopic} and will be recapitulated in Proposition~\ref{P:micro laws}.  The structure of these laws closely resembles those of the NLS/mKdV hierarchy presented in \cite{harropgriffiths2020sharp}.   There is a good reason for this: the Kaup--Newell Lax operator for \eqref{DNLS} can be written as
\begin{align}\label{L defn}
   L(\kappa;q) := \begin{bmatrix}  1 &0\\ 0 &-1\end{bmatrix}
   \begin{bmatrix}\kappa -\partial & \sqrt\kappa q\\
   i\sqrt\kappa \bar q & \kappa +\partial\end{bmatrix},
\end{align}
which closely resembles the AKNS--ZS Lax operator of the NLS/mKdV hierarchy.  We will only discuss this operator for $\kappa\in\R$.  Throughout this paper,
$$
\sqrt\kappa=i\sqrt{|\kappa|} \qtq{when} \kappa<0.
$$

There is one more prior result that we wish to discuss, namely,  global well-posedness in $H^{1/6}(\R)$.  This was first shown in \cite{KNV} for initial data satisfying $M(q)<4\pi$, a restriction that was removed in \cite{HGKV:equi}.  This result was shown using the first-generation method of commuting flows introduced in \cite{MR3990604} and reviewed below.

\subsection{Description of the method}\label{SS:method} 
The principal problem we must address in order to prove Theorem~\ref{t:main} is this: given $T>0$ and an $L^2$-convergent sequence $q_n(0)$ of Schwartz initial data, show that the corresponding (Schwartz-class) solutions $q_n(t)$ converge in $L^2(\R)$ uniformly for $|t|\leq T$. 

Given the breakdown of uniform continuity of the data-to-solution map (on bounded sets) and the further instabilities highlighted in Proposition~\ref{P:noGron}, it is difficult to conceive of a method of controlling differences of solutions in terms of their initial data.  It was to address this specific challenge that the method of commuting flows was introduced in \cite{MR3990604}.

To explain the method of commuting flows, let us imagine that we wish to prove well-posedness of the flow generated by a Hamiltonian $H$ in $L^2$; in our case, $H$ is given in \eqref{H}. Central to the method is the construction of a one-parameter family of Hamiltonians $H_\kappa$ whose flows satisfy the following three properties:

(1) they commute with the $H$ flow, 

(2) they are well-posed in the target well-posedness space $L^2$, and 

(3) they converge to the $H$ flow as $\kappa\to \infty$. 

Let us temporarily take for granted the existence of the family of Hamiltonians $H_\kappa$ satisfying these properties.  Their construction for \eqref{DNLS} is quite involved and will be discussed shortly.  Demonstrating that they satisfy the three properties requires almost the entire bulk of this paper.

Property (1), namely commutativity of the flows, can be expressed as 
$$
e^{ t J \nabla H} \circ e^{ s J \nabla H_\kappa} = e^{ s J \nabla H_\kappa} \circ e^{ t J \nabla H}
	= e^{ J\nabla [t  H + sH_\kappa]}  \quad \text{for any $s,t\in \R$,}
$$  
where we adopt the exponential notation for the flow of a vector field and write $J\nabla$ for the symplectic gradient. The relevance of this relation to demonstrating that a sequence of (Schwartz) solutions $q_n(t)$ is Cauchy in $C([-T,T];L^2)$ may be best understood via the following identity: 
\begin{align}
q_n(t) - q_m(t) & =  \bigl[e^{ t J \nabla H_\kappa} q_n(0) - e^{ t J \nabla H_\kappa} q_m(0)  \bigr]
	 + \bigl[ e^{ t J \nabla (H-H_\kappa)} -\Id \bigr]\circ e^{ t J \nabla H_\kappa}  q_n(0)\notag \\
&\qquad - \bigl[ e^{ t J \nabla (H-H_\kappa)} -\Id \bigr]\circ e^{ t J \nabla H_\kappa}   q_m(0) .\label{E:comm and diff}
\end{align}

Property (2) is well-posedness of the $H_\kappa$ flows. This implies that the first term on the right-hand side of \eqref{E:comm and diff} converges to zero as $n, m\to \infty$ for each fixed $\kappa$.  In order to prove that the sequence $q_n(t)$ is Cauchy in $C([-T,T];L^2)$, it remains to show that 
\begin{align}\label{diff to id}
\limsup_{\kappa\to \infty}\ \sup_n\ \sup_{|t|\leq T}  \bigl\|\bigl[ e^{ t J \nabla (H-H_\kappa)} -\Id \bigr]\circ e^{ t J \nabla H_\kappa}   q_n(0)\bigr\|_{L^2} =0.
\end{align}
We will refer to the flow generated by the Hamiltonian $H-H_\kappa$ as the \emph{difference flow}.  Relation \eqref{diff to id} embodies the statement that as  $\kappa\to\infty$, the difference flow converges to the identity.  This is a quantitative interpretation of property (3).  

In implementing the method of commuting flows, we have come to regard properties (1) and (2) as selection criteria for the $H_\kappa$ Hamiltonians, leaving property (3) as the key analytical difficulty that must be faced. 

In our experience, \eqref{diff to id} has always proved to be a very difficult problem.  First, we must acknowledge that the difference flow inherits all the strong instabilities of the original flow.  As the $H_\kappa$ flow is typically a diffeomorphism, it cannot undo these problems.  The one big advance, however, is that we no longer need to control \emph{differences} of solutions: $q_n$ and $q_m$ are now completely decoupled.  This is partially off-set by the fact that the initial data for the difference flow is not $q_n(0)$, but rather $e^{ t J \nabla H_\kappa}  q_n(0)$ where $t$ varies over $[-T,T]$.  As a result, we will need to show that the difference flow converges to the identity uniformly across sets of initial data about which we know very little.  

As in previous works, we will exploit that $\{e^{ t J \nabla H_\kappa}   q_n(0): n\in \N, \, |t|\leq T\}$ inherits equicontinuity from the precompact set $\{q_n(0)\}$. This follows from Theorem~\ref{T:HGKV:equi} and the fact that the $H_\kappa$ flows constructed below conserve $a(i\kappa;q)$.  Refracted through this perspective, property (3)  becomes the following assertion:
\begin{align}\label{diff to id2}
\limsup_{\kappa\to \infty}\sup_{q\in Q}\sup_{|t|\leq T}  \bigl\|\bigl[ e^{ t J \nabla (H-H_\kappa)} -\Id \bigr](q)\bigr\|_{L^2} =0
\end{align}
for any $L^2$-bounded and equicontinuous set $Q\subseteq \Schw(\R)$.

It is natural to seek to prove \eqref{diff to id2} by estimating the difference-flow vector field, that is, the time derivative under this flow.  As a prerequisite, one needs to be able to make sense of the nonlinearity appearing therein --- this includes the nonlinearity under the $H$ flow, which for \eqref{DNLS} is $(|q|^2q)'$.  

The method of commuting flows was applied to \eqref{DNLS} in \cite{KNV}, treating initial data in $H^{1/6}$.  Conservation laws were used to bound the resulting solution in $L_t^\infty H^{1/6}$, which allowed the authors to prove \eqref{diff to id2} with $L^2$ replaced by $H^{-4}$.  The lost derivatives were then recovered using equicontinuity, specifically, the general statement that if a sequence $q_n$ converges in $H^\sigma$ and is equicontinuous in $H^s$ for $s>\sigma$, then it converges in $H^s$. The relevance of $H^{1/6}$ is that it embeds into $L^3$ and this allows us to interpret the nonlinearity $(|q|^2q)'$ as an element of $L^\infty_t H^\sigma$ for any $\sigma<-\frac32$.

Already in the first application of the method of commuting flows in \cite{MR3990604}, which was for KdV in $H^{-1}$, it was not possible to estimate the difference-flow vector field 
directly.   To address this problem, the authors introduced a gauge transformation (a diffeomorphic change of unknown), whose difference-flow dynamics they could estimate pointwise in time (albeit with a sizable loss of derivatives, which were then recovered using equicontinuity).

One advantage of this \emph{first-generation} method of commuting flows, where the difference flow (with or without a gauge) is estimated pointwise in time, is that it works equally well for problems posed both on the line and on the circle.  However, there are models (such as NLS and mKdV) where the threshold regularities for well-posedness are different in the two geometries.  The treatment of these equations in \cite{harropgriffiths2020sharp} necessitated the introduction of a \emph{second-generation} method of commuting flows, based on new \emph{local smoothing} and \emph{tightness} estimates.  

The job of local smoothing estimates is to make sense of the vector field as a spacetime distribution in instances where this cannot be done pointwise in time. As the central problem is to control the \emph{difference flow} by estimating the size of the corresponding vector field, one must develop local smoothing estimates for \emph{this} flow.   This is almost paradoxical: local smoothing is an expression of high-frequency transport; however, our ultimate goal is to demonstrate that the difference flow converges to the identity.  The demonstration of sufficiently strong smoothing estimates for the \eqref{DNLS} difference flow (see Proposition~\ref{p: diff local smoothing}) requires a vast amount of work;  we will return to this topic after the Hamiltonians $H_\kappa$ have been introduced.

Using local smoothing, we will only be able to prove convergence of the difference flow to the identity locally in space (cf. Theorem~\ref{T: diff flow convergence}).  The role of the second new ingredient, tightness, is to overcome this limitation.  As both radiation and solitons move under the \eqref{DNLS} flow, establishing tightness is challenging.  For \eqref{DNLS} this is accomplished in Proposition~\ref{P:tight} below and relies on the subtle control of the high-frequency transportation provided by \eqref{refined local smoothing}.  

\subsection*{The $H_\kappa$ flows and their properties}

To introduce the Hamiltonians $H_\kappa$ for \eqref{DNLS}, we return to the perturbation determinant \eqref{a defn}, or rather, to the closely related quantity
\begin{align}\label{A defn}
A(\kappa;q) := - \sgn(\kappa) \log[a(i\kappa;q)] = - \sgn(\kappa) \log \det\bigl[ 1 - i\kappa \Lambda\Gamma\bigr].
\end{align}
As discussed above, $a(i\kappa;q)$ is conserved by the \eqref{DNLS} flow.  This guarantees that both the real and imaginary parts of the (complex) functional $A(\kappa;q)$ Poisson commute with the \eqref{DNLS} Hamiltonian $H(q)$.

The inclusion of $\sgn(\kappa)$ in \eqref{A defn} ensures that $A(\kappa;q)$ has the same asymptotic expansion as $\kappa\to\pm\infty$.  This expansion shows that $A(\kappa;q)$ encodes all the polynomial conservation laws of \eqref{DNLS}; it begins
\begin{align}\label{Asymp of A}
   A(\kappa;q) = \tfrac{i}{2}M(q) + \tfrac{1}{4\kappa} H(q) - \tfrac{i}{8\kappa^2} H_2(q) + O\bigl(\tfrac1{\kappa^3}\bigr)
\end{align}
for $q\in\Schw(\R)$.  Rearranging this formula leads one to believe that  
\begin{align}\label{Hk defn}
    H_\kappa(q):= 4\kappa\Re A(\kappa;q)
\end{align}
is a good approximation for the \eqref{DNLS} Hamiltonian $H(q)$, at least as $\kappa\to\infty$.  Moreover, the Poisson commutativity of $\Re A(\kappa;q)$ and $H(q)$ noted above guarantees that $H_\kappa$ and $H$ also commute.  This is the sought-after property (1) from our overview of the method of commuting flows.

The preceding discussion has been predicated on the non-vanishing of $a(i\kappa;q)$, so that one may safely take the logarithm in \eqref{A defn}.  This issue is discussed in \cite{KNV}, where it is shown that $A(\kappa;q)$ is well-defined provided $\kappa$ is sufficiently large; however, (by necessity) the restriction on $\kappa$ is not dictated solely by the size of $q$, but also by its frequency distribution.  

Informed by the many computations ahead of us, in this paper we adopt the expedient of using \eqref{scaling} to rescale solutions $q$ so that we may impose a single restriction on $\kappa$, namely, $|\kappa|\geq1$.  The goal of the rescaling is to make $\widehat q(\xi)$ small at frequencies $|\xi|\geq 1$; such smallness is conveniently expressed through the following notion:

\begin{definition}\label{D:good}
Fix $0<\sigma<\frac12$.  Given $\delta>0$, we say that $Q\subseteq \Schw(\R)$ is \emph{$\delta$-good} if it is $L^2$-bounded, $L^2$-equicontinuous, and it satisfies
\begin{align}\label{goodness}
\sup_{q\in Q} \int \frac{|\xi|^{2\sigma}|\hat q(\xi)|^2}{(4+\xi^2)^{\sigma}}\, d\xi \leq \delta^2 .
\end{align}
\end{definition}

Although the parameter $\sigma$ could be frozen once and for all, say $\sigma=\frac14$, we believe that retaining the symbol $\sigma$ makes it easier to check our computations.

As reviewed in Section~\ref{S:Green}, the series \eqref{A defn} converges uniformly on all $\delta$-good sets (once $\delta$ is small enough).  Local well-posedness of the $H_\kappa$ flow for $\delta$-good sets of initial data follows from Picard's Theorem because the corresponding vector field is Lipschitz.  Moreover, these solutions remain Schwartz-class and conserve $A(\vk;q)$ for all $|\vk|\geq 1$.  These assertions were shown in \cite[\S5]{KNV}.

In order to  construct a global-in-time $H_\kappa$ flow, we must ensure that orbits remain $\delta$-good as time progresses.  This is accomplished by combining the fact that the $H_\kappa$ flow preserves $a(i\vk;q)$ together with the following consequence of Theorem~\ref{T:HGKV:equi}:

\begin{corollary} \label{C:descendants}
Let \(Q\subseteq\Schw(\R)\) be an \(L^2\)-bounded and equicontinuous set.  Given $\delta>0$, there exists $\lambda=\lambda(Q, \delta)$ so that the set 
\begin{equation*}
Q_*^\lambda = \bigl\{ \sqrt{\lambda} \,q(\lambda x)  \in \Schw(\R) : a(i\vk;q)\equiv a(i\vk;\tilde q) \text{ for some $\tilde q\in Q$}\bigr\}
	\quad\text{is $\delta$-good. }
\end{equation*}
\end{corollary}

Before turning to the difficult topic of analyzing the difference flow, let us pause to summarize the preceding discussion as a theorem.  In particular, this theorem encapsulates properties (1) and (2) of the $H_\kappa$ flows.

\begin{theorem}\label{T:Hk WP}
There exists $\delta_0>0$ sufficiently small so that for any $0<\delta\leq \delta_0$ and \(L^2\)-bounded and equicontinuous set \(Q\subseteq\Schw(\R)\), the $H_\kappa$ flow is globally well-posed on the set $Q^\lambda_*$ in the $L^2$ topology, where $\lambda(Q,\delta)$ is chosen according to Corollary~\ref{C:descendants}.  In particular, solutions remain of Schwartz-class and remain in the set $Q^\lambda_*$. Moreover, the \eqref{DNLS} flow also preserves the set $Q^\lambda_*$ and commutes with the $H_\kappa$ flow.
\end{theorem}

All the assertions made here about the $H_\kappa$ flow were proved in \cite{KNV} contingent on the question of equicontinuity that was subsequently resolved in \cite{HGKV:equi}.  As discussed earlier, the existence of global Schwartz-class solutions to \eqref{DNLS} follows from  \cite{MR1152001,MR4149070}. That such \eqref{DNLS} solutions conserve the transmission coefficients is a classical result.

\subsection*{Analysis of the difference flow} This occupies the bulk of this paper and requires many new insights.

To understand the difference flow, we must first give the explicit form of this evolution.  This relies on the functional derivatives of $H_\kappa$, which are easily deduced from those of $A(\kappa;q)$ given in \eqref{FderivA}.

The functions $g_{12}(x)$ and $g_{21}(x)$ appearing in \eqref{FderivA} are the two off-diagonal entries of the Green's function corresponding to the Lax operator \eqref{L defn} evaluated on the diagonal $x=y$.  Together with a third component $\gamma(x)$, these functions will be recurrent characters in our story and Section~\ref{S:Green} is devoted to a detailed elaboration of their algebraic and analytic properties.  

Combining the functional derivatives with the Poisson structure \eqref{Poisson}, we find an explicit formula for the difference flow evolution:
\begin{equation}\label{404}
i\tfrac{d}{dt}  q =  -q'' - i (|q|^2q)' + 2\kappa \bigl[\sqrt\kappa g_{12}'(\kappa)- \sqrt{-\kappa} g_{12}'(-\kappa)\bigr].
\end{equation}
The goal of this section is to explain how to prove \eqref{diff to id2} for the difference flow given by \eqref{404}. 

Unlike all other terms, the nonlinearity $(|q|^2q)'$ appearing on the right-hand side of \eqref{404} does not make sense pointwise in time for $q\in C_tL^2$ and so we are immediately tasked with finding a remedy.  Despite strenuous efforts, we were unable to find a gauge transformation that would allow us to estimate the resulting difference-flow vector field pointwise in time.   Based on previous successes with the diagonal Green's function \cite{harropgriffiths2020sharp,MR3990604}, it is natural to imagine that this might be a satisfactory  gauge. This idea is refuted by \eqref{dt g12}, as explained there.  Therefore, we are forced to adopt the second-generation method of commuting flows.  

As discussed earlier, the characteristic features of the second-generation method are the use of local smoothing estimates to control the difference flow and the resulting necessity of showing that compact sets of initial data lead to tight ensembles of orbits under the \eqref{DNLS} flow, over bounded time intervals.  The proof  of the tightness statement relies on refined local smoothing estimates for the \eqref{DNLS} flow.

Local smoothing estimates are a direct expression of the dispersive nature of an equation: high frequencies travel rapidly and so spend little time in any fixed spacetime region.  As such, they originate from the linear/dispersive part of the equation.

As discussed in \cite{MR2233925}, there are two standard ways for proving local smoothing estimates for a linear equation: via spacetime Fourier transforms and via monotonicity identities. When the nonlinearity may be treated perturbatively, local smoothing estimates can be transferred directly from the underlying linear flow to the full nonlinear equation.  Correspondingly, it matters little what method one uses for establishing the linear estimates.  In the non-perturbative regime considered in this paper, we have no choice but to pursue an approach based on monotonicity identities for the full nonlinear flow. 

All monotonicity-type identities we know originate from microscopic conservation laws of the form
\begin{align}\label{12:29}
\partial_t \rho+ \nabla_x\cdot \vec \jmath=0.
\end{align}
In one spatial dimension for example, this implies
\begin{align}\label{12:30}
\partial_t \int \tanh(x)\rho(t,x)\, dx  = \int \sech^2(x) j(t,x)\, dx.
\end{align}

In the rare event that one can find such a law with $j\geq 0$, \eqref{12:29} constitutes true monotonicity.  It is more reflective of actual practice however to find a coercive term $j_1$ in the current and then integrate \eqref{12:30} to obtain
\begin{align}\label{12:31}
\int_{-T}^T \int \sech^2(x) j_1(t,x)\, dx \, dt &\leq 2\sup_{|t|\leq T}\Bigl|\int \tanh(x)\rho(t,x)\, dx\Bigr|\notag \\
&\quad+ \int_{-T}^T \int \sech^2(x) [j_1-j](t,x)\, dx \, dt.
\end{align}
The utility of this inequality rests on finding a suitable microscopic conservation law.  First, one must find a density $\rho$ whose integral can be controlled uniformly in time.  Second, one must be able to identify a coercive part $j_1$ of the current that controls the sought-after local smoothing norm. Third, one must be able to control the contribution of $j_1-j$.  

For our analysis we need two one-parameter families of microscopic conservation laws, one for \eqref{DNLS} and one for the difference flow \eqref{404}.  These can be found in Proposition~\ref{P:micro laws}, with the density given in \eqref{rho}.  Local smoothing for \eqref{DNLS} is proved in Proposition~\ref{p:LS} and for the difference flow in Proposition~\ref{p:  diff local smoothing}.

The first task is to estimate the integral of the density $\rho$ in \eqref{rho} uniformly in time.  This is achieved in Lemma~\ref{L:rho}.  The complicated structure of this density makes this a nontrivial task.  Moreover, to prove tightness of orbits under the \eqref{DNLS} flow we need the refined local smoothing estimate \eqref{refined local smoothing}, which requires us to prove that the contribution of $\rho$ converges to zero in the high-frequency regime.
The analysis of $\rho$ in Lemma~\ref{L:rho} relies on the detailed study of the diagonal Green's functions carried out in Section~\ref{S:Green}.  

Our second task is to identify a coercive part in the currents appearing in Proposition~\ref{P:micro laws}.  In our analysis, the quadratic terms $j_\DNLS^{[2]}$ and $j_\diff^{[2]}$ of the currents will play the role of $j_1$ in \eqref{12:31}.  Although these are not sign definite, we are able to demonstrate the requisite coercivity up to acceptable errors.  For the treatment of $j_\DNLS^{[2]}$, see the discussion surrounding \eqref{j2}.  Extracting coercivity from $j_\diff^{[2]}$ requires considerable regrouping and the estimation of many error terms and commutators; see the treatment of \eqref{j2diff}.

The third and most difficult part of obtaining local smoothing estimates is controlling the remainder of the current $j_1-j$.  In defocusing problems, the most dangerous parts of $j_1-j$ typically have a favorable sign.  When the problem treated is subcritical, the second term on RHS\eqref{12:31} can be controlled by interpolating between the LHS\eqref{12:31} and a priori conservation laws.  A small data hypothesis can also provide the smallness needed to bound the contribution of $j_1-j$ by a small fraction of LHS\eqref{12:31}.  The problem studied here has none of these favorable features.  In fact, any of these features would yield local smoothing estimates that depend only on the norm of the initial data; this is ruled out by Proposition~\ref{P:no smoothing}.

In our case, the remainder $j_1-j$ comprises the quartic and higher order terms $j_\DNLS^{[\geq 4]}$ and $j_{\diff}^{[\geq 4]}$.  There is an enormous number of contributions that need to be controlled.  Moreover, these cannot be estimated directly using the $L^2$ norm of $q$ since they involve both derivatives and higher powers of $q$.  Instead, we endeavor to control these contributions using local smoothing and a bootstrap argument.

As we are dealing with a large-data scaling-critical problem, there is no easy source of smallness for closing the bootstrap.  This is one of the key analytical challenges we must overcome in this article.  The subcriticality of the models treated in \cite{harropgriffiths2020sharp} expressed itself through the appearance of negative powers of the large parameter $\kappa$, which provided the requisite smallness.  For \eqref{DNLS}, we are forced  to simultaneously exhibit two copies of the local smoothing norm (to be bootstrapped) and a third factor encoding equicontinuity (the source of smallness) for every single error term.  To achieve this, we must identify and exploit many subtle hidden cancellations in the flow --- see, for example, the carefully curated decompositions of $j_\DNLS^{[\geq 4]}$ and $j_{\diff}^{[\geq 4]}$ appearing in \eqref{jDNLS rewriting} and in the proof of Lemma~\ref{L:rewriting}, respectively.

The analysis of these error terms relies on a large body of work built up in the preceding sections of the paper.  In Section~\ref{S:Pre} we introduce the norms used to quantify both equicontinuity and local smoothing.  We also need to introduce and analyze a Banach algebra $B$ of bounded multiplication operators on our equicontinuity spaces.  This section also contains a suite of basic nonlinear estimates used later in the paper.

Much of Section~\ref{S:Green} is devoted to proving estimates on the diagonal Green's functions.  These arise in several places in our analysis: not only are they an integral part of the microscopic conservation laws, but they also appear in \eqref{404} because they encode the functional derivatives of $A(\kappa;q)$.  The culmination of Section~\ref{S:Green} is the estimation of the diagonal Green's functions and key nonlinear combinations thereof in the equicontinuity and the local smoothing spaces.  We also elucidate the structure of these functions in terms of a new class of paraproducts introduced in this section; see, for example, Lemma~\ref{L3:11} and Proposition~\ref{p:Asymptotics}.  

A second class of paraproducts which incorporates the localizing weights intrinsic to local smoothing estimates is introduced in Section~\ref{S:Local}.  A key feature of our analysis is demonstrating that one may distribute these localizing weights to all entries in these paraproducts. This is important since any one of the input functions in a paraproduct may be the highest frequency term and so will need to be estimated in the local smoothing norm.  The culmination of Section~\ref{S:Local} is the proof of the local smoothing estimates for the \eqref{DNLS} flow stated in Proposition~\ref{p:LS}.

Tightness of orbits for solutions to \eqref{DNLS} is proved in Section~\ref{S:Tight}.  This argument is quite short because of the strength of the estimate \eqref{refined local smoothing} proved in Section~\ref{S:Local}.

Section~\ref{sect: diff local smoothing} contains a proof of local smoothing for the difference flow \eqref{404}.  It is the most demanding part of the paper and relies on all the analysis that precedes it.  Indeed, the needs of this section dictated much of the prior development.

Section~\ref{S:Convergence} combines all that precedes it to prove convergence of the difference flow to the identity, locally in space.  Finally, a short Section~\ref{S:end} deduces all the main theorems from these prerequisites.

\subsection{The soliton menace}\label{SS:solitons}
%
%
In this subsection we present the family of soliton solutions to \eqref{DNLS} and use them to exhibit some of the instabilities of this equation.

For each value of $\theta\in(0,\frac\pi2)$, the function
\begin{align*}
q_0(x;\theta):= \sqrt{ 2\sin(2\theta) }
	\frac{[\cos(\theta)\cosh(x) - i\sin(\theta)\sinh(x) ]^3}{[\cos^2(\theta)\cosh^2(x) + \sin^2(\theta)\sinh^2(x) ]^2}
	e^{ -i x \cot(2\theta) }
\end{align*}
provides initial data for a soliton solution to \eqref{DNLS}.  In understanding the shape of this function, it is useful to note that the central factor can be written as $Z^3/|Z|^4$ with $Z=\cos(\theta)\cosh(x) - i\sin(\theta)\sinh(x)=\cosh(x-i\theta)$.
The soliton with this initial data takes the form
\begin{align}\label{soliton}
q(t,x;\theta) = q_0\bigl( x+2\cot(2\theta) t;\theta\bigr)  e^{i t \cosec^2(2\theta)}.
\end{align}
Further solitons can be obtained by translation, phase rotation, and scaling.

The $\theta\to 0$ limit of this solution exists and is identically zero.  Indeed
\begin{equation}\label{norm of soliton}
 \| q_0 \|_{L^2}^2 = 8 \theta.
\end{equation}
In the form we have presented, the $\theta\to\tfrac\pi2$ limit does not exist.  However by rescaling in accordance with \eqref{scaling}, a limit can be recovered, namely, the algebraic soliton:
\begin{align}\label{alg s}
q(t,x) = q_a(x-t)  e^{it/4}  \qtq{with initial data} q_a(x) = \frac{2(1-ix)}{(1+ix)^2} e^{ix/2}   . 
\end{align}

This solution embodies a key obstruction to coercivity of the polynomial conservation laws.  Indeed, while $M(q_a) = 4\pi$ all other polynomial conserved quantities vanish.  These properties also hold for all rescalings \eqref{scaling} of $q_a$.  However,
\begin{align}\label{coercive fails}
\|q_{a,\lambda}\|_{H^s} \to \infty \qtq{as} \lambda\to \infty
\end{align}
for any $s>0$. 

To see that $q_a$ also witnesses an obstruction to using the perturbation determinant \eqref{a defn} to prove equicontinuity we note that $a(i\kappa;q_{a,\lambda})\equiv 1$ for all $\lambda>0$.  However, $\{q_{a,\lambda}:\, \lambda>0\}$ is \emph{not} an equicontinuous family. 

Let us now turn our attention to the soliton with $\theta=\frac\pi4$, which simplifies to
\begin{align}\label{stat soliton}
q_s(t,x)  = \frac{ 2 e^{it} {[\cosh(x)-i\sinh(x)]}^3}{ {[ \cosh^2(x) + \sinh^2(x) ]}^2} .
\end{align}
The subscript $s$ appearing here emphasizes that this is a \emph{stationary} soliton.  (While it does oscillate in time, it does not propagate through space.)  This property makes it the archenemy of local smoothing and Strichartz estimate.  In particular, our next proposition shows that the inequality \eqref{E:naiveLS} does not hold for sets $Q$ that are merely $L^2$-bounded.  This further emphasizes the non-perturbative nature of the \eqref{DNLS} flow in the $L^2$ topology.

\begin{proposition}\label{P:no smoothing}
Local smoothing and Strichartz norms cannot be controlled solely by the $L^2$ norm of the initial data.  Concretely, there is a sequence of solutions $q_n$ to \eqref{DNLS} satisfying $M(q_n) \equiv 2\pi$ but
\begin{align*}
\int_{-1}^1 \| \sech^{12}(x) q_n(t,x) \|_{H^{1/2}_x}^2 \,dt   \longrightarrow \infty \qtq{and} \int_{-1}^1 \int_{\R}|q_n(t,x)|^6\, dx\,dt  \longrightarrow \infty
\end{align*}
as $n\to \infty$.
\end{proposition}

\begin{proof}
We choose the $q_n$ to be rescalings of $q_s$ according to \eqref{scaling} and observe that $M(q_{s,\lambda}) \equiv 2\pi$ but
\begin{equation*}
\int_{-1}^1 \| \sech^{12}(x) q_{s,\lambda}(t,x) \|_{H^{1/2}_x}^2 \,dt \approx \lambda \qtq{and} \int_{-1}^1\int_{\R}|q_{s,\lambda}(t,x)|^6\,dx \,dt  \approx \lambda^2. \qedhere
\end{equation*}
\end{proof}

As the last topic of this section, we demonstrate another instability inherent to \eqref{DNLS}.  Concretely, we will show that it is not possible to prove uniqueness of solutions via Gr\"onwall's inequality in lower regularity spaces. This is a widely successful uniqueness technique and a key ingredient in constructing solutions via compactness/uniqueness arguments; the Gr\"onwall inequality yields Lipschitz dependence in $H^s$ of the data-to-solution map.  However, our next proposition shows failure of Lipschitz dependence, no matter how negative one chooses $s$.

\begin{proposition}\label{P:noGron}
Fix $s\leq 0$. There are times $t_n\to0$ and pairs of solutions $q_n$ and $\tilde q_n$ to \eqref{DNLS} so that 
\begin{align}
\| q_n\|_{L^2} + \| \tilde q_n\|_{L^2} \to 0 \qtq{but}
	\frac{\| q_n(t_n) - \tilde q_n(t_n) \|_{H^s} }{ \| q_n(0) - \tilde q_n(0) \|_{H^s} }  \to \infty.
\end{align} 
\end{proposition}

\begin{proof}
We choose $q_n$ and $\tilde q_n$ to be distinct rescalings of the soliton solution \eqref{soliton} with parameter $\theta_n$.  We first choose $\theta_n\to 0$ to ensure that their $L^2$ norms converge to zero; see \eqref{norm of soliton}.

The key idea to exploit is the fact that $q_n$ and $\tilde q_n$ travel at different speeds.  To ensure that their separation at time $t_n$ diverges, we require that
\begin{align}\label{t_n choosing}
|\lambda_n-\tilde\lambda_n| \cot(2\theta_n) t_n \to \infty \qtq{yet} t_n\to0 \qtq{as} n\to\infty. 
\end{align}

In order to compute the overall size of the norms at the times $0$ and $t_n$, it is convenient to compute the Fourier transform of a soliton exactly.  The key identity is this:
\begin{align}
\int \frac{\cosh(x-i\theta)}{\cosh^2(x+i\theta)}\,e^{-i\xi x}  \,dx
	= \frac{\pi e^{-\theta\xi}}{\cosh(\frac\pi2 \xi)}\bigl[ \cos(2\theta) - \xi\sin(2\theta) \bigr], 
\end{align}
which follows by a simple residue computation.  It follows from this that $\hat q_0(\xi)$ has a simple zero at the origin, yielding three cases: $s<-\frac32$, $s=-\frac32$, and $s>-\frac32$. 

In the regime where \eqref{t_n choosing} and $|\lambda_n-\tilde\lambda_n| \ll \lambda_n$ both hold, elementary (but lengthy) computations show
$$
\tfrac{\lambda_n}{|\lambda_n-\tilde\lambda_n|} \| q_n(0) - \tilde q_n(0) \|_{H^s(\R)}
	\lesssim  \| q_n(0) \|_{H^s(\R)} \approx \| q_n(t_n) - \tilde q_n(t_n) \|_{H^s}.
$$
With this information it is not difficult to choose the necessary parameters.
\end{proof}

\subsection{Equivalent models and their physical origins}\label{SS:gauge} Let us begin by noting that \eqref{DNLS} does not admit a focusing/defocusing dichotomy: the sign of the nonlinearity can be reversed by simply replacing $x\mapsto -x$.  Likewise, the relative coupling of the three terms in \eqref{DNLS} can be freely adjusted by rescaling the space and time variables.

To the best of our knowledge, \eqref{DNLS} first appears in the literature as a model for the propagation of large-wavelength Alfv\'en waves in plasma.  For a further discussion of this scenario, including how this effective model informs our understanding of the stability of such Alfv\'en waves, see \cite{KAWAHARA,Mjolhus,doi:10.1063/1.1693399,MR496220}.

It is easily seen that \eqref{DNLS} does not inherit the Galilean symmetry of the linear Schr\"odinger equation.  Indeed, if $q$ solves \eqref{DNLS}, then 
\begin{align}\label{boost}
v(t,x) = e^{ i k x - i k^2 t} q(t,x-2kt)
\end{align}
solves
\begin{align}\label{Galilei}
	i \partial_t v + v'' + i(|v|^2 v)'  + k |v|^2 v =0.
\end{align}
Here $k\in\R$ is fixed but arbitrary.

This computation indicates that the traditional cubic nonlinear Schr\"odinger equations (both focusing and defocusing) are `embedded' inside \eqref{DNLS} in the limit of large modulation.  While we know of no mathematical work on this embedding, we will describe two physical systems which speak to this phenomenology.

The combined nonlinearities of \eqref{Galilei} arise naturally in nonlinear optics.  While negligible in many experimental scenarios, the derivative nonlinearity becomes physically important in the propagation of short pulses (cf. \cite{PhysRevA.27.1393,PhysRevA.23.1266}).

One early application of the cubic nonlinear Schr\"odinger equation was to modeling amplitude modulations of Alfv\'en waves, with the unknown function describing deviations from a plane wave.  (In our earlier discussion of \eqref{DNLS} as model of Alfv\'en waves, $q$ describes the entire amplitude of the wave, not fluctations.)   One of the key assumptions in deriving this model is that the characteristic length of the modulations far exceeds the carrier wavelength.  As argued in \cite{JPSJ41.265}, the combined nonlinearities of \eqref{Galilei} allow one to extend the realm of applicability of this effective model to include cases where these two length scales are almost comparable.

As part of a search for completely integrable PDE, a different form of derivative nonlinear Schr\"odinger equation was uncovered in \cite{MR544493}, namely,
\begin{equation}\label{CLLeqn}
    i\partial_t q + q'' + i |q|^2 q'=0 .
\end{equation}
It was subsequently discovered (see, e.g. \cite{MR700302}) that this model can be obtained from \eqref{DNLS} via a change of variables.  The change of variables in question takes the following form:
\begin{align}\label{guage change}
w(t,x) = q(t,x) e^{i\nu \Phi} \qtq{with}  \Phi(t,x) = \int_{-\infty}^x |q(t,y)|^2\,dy
\end{align}
and $\nu\in \R$ fixed.  With a little work, we find that $w$ satisfies
\begin{align}\label{guaged DNLS}
i w_t + w'' = 2i(\nu-1)|w|^2 w' + i(2\nu-1)w^2\bar w' - \tfrac12\nu(2\nu-1)|w|^4w .
\end{align}
When $\nu=\frac12$, we recover \eqref{CLLeqn}.  When $\nu=1$, we obtain the Gerdjikov--Ivanov form of derivative NLS; see \cite{GIart}.

The more general nonlinearity presented in \eqref{guaged DNLS}, as well as that arising by further incorporating the Galilei transformation \eqref{boost}, appears naturally in the study of the Benjamin--Feir instability in the theory of water waves; see \cite{doi:10.1146/annurev.fluid.40.111406.102203,MR479013,MR887472}. 

It is easy to see that both changes of variables, \eqref{Galilei} and \eqref{guage change}, are real-analytic diffeomorphisms on $L^2(\R)$; indeed, they are diffeomorphisms on $H^s(\R)$ for every $s\geq 0$.  Thus Theorem~\ref{t:main} guarantees the following:

\begin{corollary}
The evolutions \eqref{Galilei} and \eqref{guaged DNLS} are globally well-posed in $L^2(\R)$.
\end{corollary}

When we look at \eqref{guaged DNLS}, it seems all the more surprising that large data GWP holds, since it fails for the focusing quintic nonlinear Schr\"odinger equation!

\subsection{Acknowledgements} R.K. was supported by NSF grant DMS--1856755 and M.V. by NSF grant DMS--2054194.

\section{Preliminaries}\label{S:Pre} 

Throughout, we will use scaling-homogeneous Littlewood--Paley decompositions with frequency parameters $N\in 2^\Z$.  Concretely, choosing a smooth, non-negative function \(\varphi\) supported on $|\xi|\leq 2$ with $\varphi(\xi)=1$ for $|\xi|\leq 1$, we define $P_{\leq N}$ as the Fourier multiplier with symbol $\varphi(\xi/N)$ and then $P_N = P_{\leq N} - P_{\leq N/2}$.  Observe that 
\[
1 = \sum_{N\in 2^\Z} P_N.
\]
Such decompositions will be ubiquitous and we often adopt the more compact notations $f_N=P_N f$ , $f_{\leq N}=P_{\leq N} f$, and $f_{>N} = [1 - P_{\leq N}]f$.

As a similar expedient, we often write Fourier multipliers under their arguments.  For example, for $\kappa>0$,
$$
\tfrac{q}{\sqrt{4\kappa^2-\partial^2}} := (4\kappa^2-\partial^2)^{-1/2} q \qtq{and} \tfrac{q}{2\kappa\pm\partial} := (2\kappa\pm\partial)^{-1} q .
$$

For \(s\in \R\) and \(|\kappa|\geq 1\) we define the Sobolev space $H^s_\kappa$ as the completion of $\Schw(\R)$ with respect to the norm
\[
\|q\|_{H^s_\kappa}^2 := \int (4\kappa^2+\xi^2)^s|\hat q(\xi)|^2\, d\xi,
\]
and write \(H^s = H^s_1\).

Associated to the localizing function $\psi_\mu$ defined in \eqref{psi}, we have
\begin{equation}\label{psi int}
\int_\R \psi(x-\mu)^{24} \,d\mu = \tfrac{512}7 \qtq{and so} f(x) = \tfrac 7{512}\int_\R f(x) \psi_\mu^{24}(x)\,d\mu.
\end{equation}

\subsection{Equicontinuity spaces}
To quantify the equicontinuity, for \(\sigma,s\in \R\) and \(|\kappa|\geq 1\) we define
\begin{align}\label{D:E}
\|q\|_{E_{s, \kappa}^\sigma}^2 := |\kappa|^{2(s-\sigma)}\| |\partial|^\sigma q\|_{H^{-s}_\kappa}^2= \int \tfrac{|\kappa|^{2(s-\sigma)}|\xi|^{2\sigma}}{(4\kappa^2+\xi^2)^{s}}|\hat q(\xi)|^2\, d\xi,
\end{align}
and take \(E_s^\sigma = E_{s,1}^\sigma\).

We write $B$ for the space of bounded functions that belong to the homogeneous Besov space $\dot B^{\frac12}_{2,\infty}$.  We equip this space with the norm
$$
\|f\|_{B} := \|f\|_{L^\infty} +\sup_{N\in 2^\Z}N^{\frac12}\|f_N\|_{L^2}.
$$
This space is an algebra; see Lemma~\ref{L:alg}.  Moreover, by Lemma~\ref{L:product}, multiplication by functions in $B$ defines a bounded operator on our equicontinuity spaces.

Our next lemma shows how the spaces $E^\sigma_{\sigma,\kappa}$ allow us to track the equicontinuity properties of orbits. 

\begin{lemma}\label{L:basic prop}
Let $Q\subset L^2$ be bounded and equicontinuous. For $\sigma>0$ we have
\begin{align}\label{E equi prop}
\lim_{\kappa\to \infty} \sup_{q\in Q} \, \|q\|_{E_{\sigma, \kappa}^\sigma}=0.
\end{align}
\end{lemma}

\begin{proof}
The claim follows from the following estimate
\begin{align*}
\|q\|_{E_{\sigma, \kappa}^\sigma}\lesssim\|q_{\leq N}\|_{E_{\sigma, \kappa}^\sigma}+ \|q_{>N}\|_{E_{\sigma, \kappa}^\sigma}\lesssim \bigl(\tfrac{N}{\kappa}\bigr)^\sigma \|q\|_{L^2}+ \|q_{>N}\|_{L^2},
\end{align*}
by choosing the frequency $N\in 2^\Z$ appropriately.
\end{proof}

\begin{lemma}\label{l:E int}
For $s>\sigma> 0$, $\beta\in \R$, and \(\kappa\geq 1\),
\begin{align}\label{E int}
\int_\kappa^\infty \|q\|_{E_{s, \varkappa}^{\sigma+\beta}}^2 \, \varkappa^{2\beta}\,\tfrac{d\varkappa}{\varkappa}\sim \kappa^{2\beta}\|q\|_{E_{\sigma, \kappa}^{\sigma+\beta}}^2.
\end{align}

Further, if \(\kappa\geq 2\) and \(I_\kappa = [1,\frac\kappa 2]\cup[2\kappa,\infty)\) then
\begin{align}\label{E int 2}
\int_{I_\kappa}\|q\|_{E_{s, \varkappa}^{\sigma+\beta}}^2\,  \varkappa^{2\beta}\,\tfrac{d\varkappa}{\varkappa}\sim \|q\|_{E_{\sigma}^{\sigma+\beta}}^2.
\end{align}
\end{lemma}

\begin{proof}
Decomposing into Littlewood--Paley pieces, 
\begin{align*}
\int_\kappa^\infty \|q\|_{E_{s, \varkappa}^{\sigma+\beta}}^2 \, \varkappa^{2\beta}\,\tfrac{d\varkappa}{\varkappa}
&\sim \sum_N\int_\kappa^\infty \tfrac{N^{2(\sigma+\beta)}\varkappa^{2(s-\sigma)}}{(\varkappa+N)^{2s}}\,\tfrac{d\varkappa}{\varkappa}\|q_N\|_{L^2}^2\\
&\sim  \sum_N \tfrac{N^{2(\sigma+\beta)}}{(\kappa+N)^{2\sigma}}\|q_N\|_{L^2}^2 \sim \kappa^{2\beta}\|q\|_{E_{\sigma, \kappa}^{\sigma+\beta}}^2.
\end{align*}
In order to integrate in $\varkappa$ one considers separately the cases $N\leq \kappa$ and $N>\kappa$, breaking the integral into the regions $[\kappa, N]$ and $[N, \infty)$ in the latter case. The estimate \eqref{E int 2} is proved similarly.
\end{proof}

\begin{lemma}\label{L:alg}
For \(0\leq \sigma<\frac12\) and \(\vk\geq 1\) we have the estimates
\begin{align}
\|fg\|_{B}&\lesssim \|f\|_B \|g\|_B,\label{E:alg}\\
\|f\|_{B}&\lesssim \vk^{-\frac12}\bigl[\vk\|f\|_{E_{2\sigma, \vk}^\sigma} + \|f'\|_{E_{2\sigma, \vk}^{\sigma}}\bigr],\label{Sob}\\
\|f\|_{B}&\lesssim \|f\|_{L^\infty} + \|f'\|_{L^1}.\label{tanh}
\end{align}
\end{lemma}

\begin{proof}
To verify \eqref{E:alg}, we employ the basic decomposition
\begin{align}\label{Bony}
P_N(fg)= P_N\Bigl[f_{\sim N} g_{\lesssim N} + f_{\lesssim N} g_{\sim N} + \sum_{M>2N} f_M g_M\Bigr].
\end{align}
From Bernstein's and H\"older's  inequalities, we see that
\begin{align*}
\| (fg)_N \|_{L^2} &\lesssim  \| f_{\sim N}\|_{L^2} \| g_{\lesssim N}\|_{L^\infty}  + \|f_{\lesssim N}\|_{L^\infty} \|g_{\sim N}\|_{L^2} + \smash{\sum_{M>2N}} \! N^\frac12 \| f_M \|_{L^2} \| g_M \|_{L^2} \\
&\lesssim   N^{-\frac12} \| f \|_{B} \|g\|_B, 
\end{align*}
from which \eqref{E:alg} then follows easily.

The estimate \eqref{Sob} follows from Bernstein's inequality:
\begin{align*}
\|f\|_{B}\lesssim \vk^{\sigma}\sum_{N\leq |\vk|}N^{\frac12-\sigma}\|f_N\|_{E_{2\sigma,\vk}^\sigma} + \vk^{-\sigma}\sum_{N>|\vk|}N^{\sigma-\frac12}\|f_N'\|_{E_{2\sigma,\vk}^{\sigma}}\lesssim \RHS{Sob}.
\end{align*}

Turning now to \eqref{tanh}, we recall that the convolution kernel associated to $P_N$ is a Schwartz function which integrates to zero.  Indeed, it can be written in the form $N K'(Nx)$ for some Schwartz function $K$.  Integrating by parts, we find
\begin{equation*}
N^{\frac12}\| f_N \|_{L^2} \leq  N^{\frac12}\biggl\| \int  K\bigl(N(x-y)\bigr) f'(y)\,dy \biggr\|_{L^2_x} \leq  \| K \|_{L^2} \|f'\|_{L^1} \lesssim \|f'\|_{L^1} . \qedhere 
\end{equation*}
\end{proof}

\begin{lemma}\label{L:product} Let \(0\leq \sigma<\frac12\), $0\leq s<\sigma + \frac12$, and $|\vk|,\kappa\geq 1$. Then we have the estimate
\begin{equation}\label{E transfer}
|\vk|\|f\|_{E_{s, \vk}^\sigma} + \|f'\|_{E_{s, \vk}^{\sigma}}\sim \|(2\vk - \p)f\|_{E_{s, \vk}^\sigma}
\end{equation}
and the product estimates
\begin{align}
\|fg\|_{E_{s, \kappa}^\sigma} &\lesssim \|f\|_{E_{s, \kappa}^\sigma}\|g\|_{B},\label{E product I}\\
\|fg\|_{E_{s, \kappa}^\sigma} &\lesssim |\vk|^{-\frac12}\|f\|_{E_{s, \kappa}^\sigma}\Bigl[|\vk|\|g\|_{E_{2\sigma, \vk}^\sigma} + \|g'\|_{E_{2\sigma, \vk}^{\sigma}}\Bigr].\label{E product II}
\end{align}
In particular, if \(\phi'\in \Schw\) we have the localization estimate
\begin{align}
\|\phi f\|_{E_{s, \kappa}^\sigma} &\lesssim_\phi \|f\|_{E_{s, \kappa}^\sigma}.\label{E loc}
\end{align}
\end{lemma}
\begin{proof}
The estimate \eqref{E transfer} follows from
\[
\|(2\vk - \p)f\|_{E_{s, \vk}^\sigma}^2 = 4|\vk|^2\|f\|_{E_{s, \vk}^\sigma}^2 + \|f'\|_{E_{s, \vk}^\sigma}^2.
\]

For the product estimate \eqref{E product I}, we have
\begin{align*}
\|fg\|_{E_{s, \kappa}^\sigma} ^2 \sim \sum_N \tfrac{\kappa^{2(s-\sigma)}N^{2\sigma}}{(\kappa+N)^{2s}} \|P_N(fg)\|_{L^2}^2.
\end{align*}
Decomposing as in \eqref{Bony} and using the H\"older and Bernstein inequalities, we estimate
\begin{align*}
\|fg\|_{E_{s, \kappa}^\sigma} ^2 &\lesssim \sum_N \tfrac{\kappa^{2(s-\sigma)}N^{2\sigma}}{(\kappa+N)^{2s}} \|f_N\|_{L^2}^2 \|g_{\ll N}\|_{L^\infty}^2\\
&\quad+\sum_N \tfrac{\kappa^{2(s-\sigma)}N^{2\sigma}}{(\kappa+N)^{2s}} \Bigl[ \sum_{M\ll N}M^{\frac12} \|f_M\|_{L^2}\Bigr]^2\|g_N\|_{L^2}^2\\
&\quad + \sum_N \tfrac{\kappa^{2(s-\sigma)}N^{2\sigma}}{(\kappa+N)^{2s}} \Bigl[ \sum_{M>2N}N^{\frac12} \|f_M\|_{L^2}\|g_{\sim M}\|_{L^2}\Bigr]^2.
\end{align*}

The first summand is easily seen to be acceptable. To estimate the second summand, we first sum in $N$ and then apply Schur's test:
\begin{align*}
\sum_N &\tfrac{\kappa^{2(s-\sigma)}N^{2\sigma}}{(\kappa+N)^{2s}} \Bigl[ \sum_{M\ll N}M^{\frac12} \|f_M\|_{L^2}\Bigr]^2\|g_N\|_{L^2}^2\\
&\lesssim \|g\|_{B}^2 \sum_{M_1\leq M_2\ll N} \tfrac{\kappa^{2(s-\sigma)}N^{2\sigma}M_1^{\frac12}M_2^{\frac12}}{N(\kappa+N)^{2s}} \|f_{M_1}\|_{L^2} \|f_{M_2}\|_{L^2}\\
&\lesssim \|g\|_{B}^2 \sum_{M_1\leq M_2} \tfrac{M_1^{\frac12-\sigma}(\kappa+M_1)^s}{M_2^{\frac12-\sigma}(\kappa+M_2)^s} \|f_{M_1}\|_{E_{s, \kappa}^\sigma}  \|f_{M_2}\|_{E_{s, \kappa}^\sigma} \\
&\lesssim \|f\|_{E_{s, \kappa}^\sigma} ^2\|g\|_B^2.
\end{align*}
Arguing similarly, we estimate the remaining summand by
\begin{align*}
&\sum_N \tfrac{\kappa^{2(s-\sigma)}N^{2\sigma}}{(\kappa+N)^{2s}}\Bigl[ \sum_{M>2N}N^{\frac12} \|f_M\|_{L^2}\|g_{\sim M}\|_{L^2}\Bigr]^2\\
&\qquad\lesssim \|g\|_{B}^2 \sum_{N< M_1\leq M_2} \tfrac{\kappa^{2(s-\sigma)}N^{2\sigma+1}M_1^{-\frac12}M_2^{-\frac12}}{(\kappa+N)^{2s}} \|f_{M_1}\|_{L^2} \|f_{M_2}\|_{L^2}\\
&\qquad\lesssim \|g\|_B^2 \sum_{M_1\leq M_2} \tfrac{M_1^{\sigma+\frac12}(\kappa+M_2)^s}{M_2^{\sigma+\frac12}(\kappa+M_1)^s} \|f_{M_1}\|_{E_{s, \kappa}^\sigma} \|f_{M_2}\|_{E_{s, \kappa}^\sigma} \\
&\qquad\lesssim \|f\|_{E_{s, \kappa}^\sigma} ^2\|g\|_B^2.
\end{align*}

The estimate \eqref{E product II} follows as a corollary of \eqref{E product I} and \eqref{Sob}.
The localization estimate \eqref{E loc} follows from \eqref{E product I} and \eqref{tanh}.
\end{proof}

\subsection{Operator estimates} We begin with the basic Hilbert-Schmidt bound for the operators $\Lambda(q)$ and $\Gamma(q)$ introduced in \eqref{Lambda}.  This estimate appeared already in \cite[Lemma~4.1]{MR3820439}:

\begin{lemma}[\cite{MR3820439}]\label{L:HS}
For $q\in L^2$ and $\kappa>0$ we have 
\begin{align}
    \|\Lambda\|_{\hs}^2=\|\Gamma\|_{\hs}^2&\approx  \int_{\mathbb R} \log(4+\tfrac{\xi^2}{\kappa^2}) \frac{|\hat q(\xi)|^2}{\sqrt{4\kappa^2 +\xi^2}}\,d\xi\lesssim \kappa^{-1}\|q\|_{L^2}^2,\label{HS real}\end{align}
\end{lemma}

Using this lemma as our basic tool, we obtain the following basic estimates when $q\in L^2$ is frequency localized.

\begin{lemma}[Operator estimates]\label{L:op est}
For $|\kappa|\geq 1$, \(0\leq \sigma<\frac12\), and \(0\leq s<\sigma+\frac12\) we have
\begin{align}
&\|\Lambda(q_N)\|_{\hs}= \|\Gamma(q_N)\|_{\hs} \approx \sqrt{\tfrac{1}{|\kappa| +N} \log\left(4+\tfrac{N^2}{\kappa^2}\right)} \|q_N\|_{L^2}, \label{hilbert_schmidt}\\
&\|\Lambda(q_N)\|_{\op}= \|\Gamma(q_N)\|_{\op} \lesssim \tfrac{\sqrt{N}}{|\kappa|+N}\sqrt{\log\left(4+\tfrac{N^2}{\kappa^2}\right)}  \|q_N\|_{L^2}\label{operator}\\
&\sum_{M\leq N}\|\Lambda(f_M) \|_{\op} \lesssim \tfrac{N}{|\kappa| +N}\log^{\frac32}\bigl(4+\tfrac{N^2}{\kappa^2}\bigr)\sup_{M\in 2^\Z} M^{-\frac12}\|f_M\|_{L^2},\label{op by L1}\\
&\sum_{M\leq N}\|\Lambda(f_M) \|_{\op} \lesssim \tfrac{|\kappa|^{-\frac12}N^{\frac12-\sigma}}{(|\kappa|+N)^{\frac12-\sigma}} \|f\|_{E_{s, \kappa}^\sigma},
\label{op by E}\\
&\sum_{M\leq N}\|\Lambda(f_M) \|_{\op} \lesssim \tfrac{N^{\frac12-\sigma}(1 + N)^s}{|\kappa|^{\frac12+\sigma-s}(|\kappa|+N)^{\frac12+s-\sigma}} \|f\|_{E_{s}^\sigma}\label{op by E 2}.
\end{align}
\end{lemma}

\begin{proof}
The estimate \eqref{hilbert_schmidt} follows immediately from Lemma~\ref{L:HS}.  The estimate \eqref{operator} follows from \eqref{hilbert_schmidt} for \(N>\kappa\) and the Bernstein inequality:
\begin{align}\label{1215}
\|\Lambda(q_N)\|_{\op}\lesssim \tfrac1{|\kappa|} \|q_N\|_{L^\infty}\lesssim \tfrac{\sqrt N}{|\kappa|}\|q_N\|_{L^2},
\end{align}
for \(N\leq \kappa\).


Claim \eqref{op by L1} follows from \eqref{operator}:
\begin{align*}
\sum_{M\leq N}\|\Lambda(f_M) \|_{\op}&\lesssim \sum_{M\leq N}\tfrac{M}{|\kappa|+M} \log^{\frac12}\left(4+\tfrac{M^2}{\kappa^2}\right)M^{-\frac12}\|f_M\|_{L^2} \\
&\lesssim \tfrac{N}{|\kappa| +N} \log^{\frac32}\left(4+\tfrac{N^2}{\kappa^2}\right)\sup_{M\in 2^\Z} M^{-\frac12}\|f_M\|_{L^2},
\end{align*}
as does \eqref{op by E}:
\begin{align*}
\sum_{M\leq N}\|\Lambda(f_M) \|_{\op}\lesssim \sum_{M\leq N}\tfrac{M^{\frac12}}{|\kappa|+M}\log^{\frac12}\left(4+\tfrac{M^2}{\kappa^2}\right)\|f_M\|_{L^2} \lesssim \tfrac{|\kappa|^{-\frac12}N^{\frac12-\sigma}}{(|\kappa|+N)^{\frac12-\sigma}} \|f\|_{E_{s, \kappa}^\sigma}
\end{align*}
and \eqref{op by E 2}:
\begin{equation*}
\sum_{M\leq N}\|\Lambda(f_M) \|_{\op}\lesssim \sum_{M\leq N}\tfrac{M^{\frac12-\sigma}(1+M)^s}{|\kappa|+M}\log^{\frac12}\left(4+\tfrac{M^2}{\kappa^2}\right)\|f_M\|_{E^\sigma_s} \lesssim \text{RHS\eqref{op by E 2}}.\qedhere
\end{equation*}
\end{proof}

\begin{corollary} \label{C:large kappa}
For $q\in L^2$ and \(0\leq \sigma<\frac12\),
\begin{align}\label{op norm bound}
\sqrt{\kappa} \bigl\|\Lambda(q)\bigr\|_{\op}\lesssim  \|q\|_{E^\sigma_{\sigma,\kappa}}\quad\text{uniformly for $\kappa\geq 1$}.
\end{align}
In particular, if $Q$ is a bounded and equicontinuous subset of $L^2$,
\begin{align}\label{equi'}
\lim_{\kappa\to \infty} \sup_{q\in Q} \sqrt{\kappa} \|\Lambda(q)\|_{\op}=0.
\end{align}
\end{corollary}

\begin{proof}
Using \eqref{op by E} and Lemma~\ref{L:HS}, we may bound
\begin{align*}
\sqrt{\kappa} \bigl\|\Lambda(q)\bigr\|_{\op}& \lesssim   \sqrt{\kappa} \sum_{M\leq \kappa} \|\Lambda(q_M)\|_{\op}+\sqrt{\kappa} \, \bigl\|\Lambda(q_{>\kappa} )\bigr\|_{\op}\\
&\lesssim   \|q\|_{E^\sigma_{\sigma,\kappa}}+\|q_{>\kappa} \|_{L^2} \lesssim \|q\|_{E^\sigma_{\sigma,\kappa}},
\end{align*}
which settles \eqref{op norm bound}.  Lemma \ref{L:basic prop} then yields \eqref{equi'}.
\end{proof}

\begin{lemma}
For \(0\leq \sigma<\frac12\) and \(\kappa\geq 1\) we have
\begin{align}
\|\Lambda(\psi_\mu^3 q;\kappa)\|_{\I_8}^8 &\lesssim \kappa^{-5}\Bigl\|\tfrac{\psi_\mu^3 q}{\sqrt{4\kappa^2 - \p^2}}\Bigr\|_{H^{\frac32}}^2\|q\|_{E_\sigma^\sigma}^6.\label{I8}
\end{align}
\end{lemma}
\begin{proof}
Decomposing into Littlewood--Paley pieces, applying \eqref{op by E 2} and \eqref{E loc} at low frequency, and \eqref{hilbert_schmidt} at high frequency, we have
\begin{align*}
&\LHS{I8} = \Bigl|\tr\Bigl\{\Lambda(\psi_\mu^3 q;\kappa)^8  \Bigr\}\Bigr|\\
&\qquad\lesssim \sum_{N_1\sim N_2\geq \dots \geq N_8}\bigl\|\Lambda\bigl(P_{N_1}(\psi_\mu^3 q)\bigr)\bigr\|_{\I_2}\bigl\|\Lambda\bigl(P_{N_2}(\psi_\mu^3 q)\bigr)\bigr\|_{\I_2}\prod_{j=3}^8\bigl\|\Lambda\bigl(P_{N_j}(\psi_\mu^3 q)\bigr)\bigr\|_{\op}\\
&\qquad\lesssim \sum_{N_1\sim N_2} \tfrac{N_1^{3-6\sigma}(1+N_1)^{6\sigma}}{\kappa^3(1+N_1)^{3}(\kappa + N_1)^2}\log\Bigl(4 + \tfrac{N_1^2}{\kappa^2}\Bigr)\Bigl\|\tfrac{P_{N_1}(\psi_\mu^3 q)}{\sqrt{4\kappa^2 - \p^2}}\Bigr\|_{H^{\frac32}}\Bigl\|\tfrac{P_{N_2}(\psi_\mu^3 q)}{\sqrt{4\kappa^2 - \p^2}}\Bigr\|_{H^{\frac32}}\|q\|_{E_\sigma^\sigma}^6\\
&\qquad\lesssim \RHS{I8}.
\end{align*}
The fact that $N_1\sim N_2$ must hold is most evident by computing the trace (which is unitarily invariant) in Fourier variables.
\end{proof}

\subsection{Local smoothing spaces}\label{S:ls}
To control the local smoothing property, for \(s,h\in \R\) and \(|\kappa|\geq 1\) we define
\begin{align}
\|q\|_{F_\kappa^s(h)}^2 &:=\int \bigl\|\tfrac{\psi_{\mu}^{12}q}{\sqrt{4\kappa^2 - \p^2}}\bigr\|_{H^{s+1}}^2\,e^{-\frac1{200}|h - \mu|}\,d\mu,\label{F def}\\
\|q\|_{X_\kappa^s}^2 &:= \sup_{h\in \R}\int_{-1}^1\|q(t)\|_{F_\kappa^s(h)}^2\,dt.\label{X def}
\end{align}
As before, when \(\kappa = 1\) we denote \(F^s(h) = F_1^s(h)\) and \(X^s = X_1^s\).

We note that as a consequence of Lemma~\ref{L:F product} below, we obtain an alternative characterization of \(X_\kappa^s\) that is closer to that used in \cite{harropgriffiths2020sharp}; see Remark~\ref{R:X alt}.  The additional complexity apparent in \eqref{F def} is necessitated by the scaling criticality of the problem.

Multiplicative commutators are an essential tool for repositioning the spatial localization factors within the paraproducts appearing in our analysis of local smoothing estimates. The next lemma, which extends \cite[Lemma 2.8]{harropgriffiths2020sharp}, is our basic workhorse in this task:

\begin{lemma}\label{l:mult comm}
For \(|\vk|,|\kappa|\geq 1\),  \(s,\sigma\in\R\), \(1<p<\infty\), $r\in\Z$, and integer \(|\ell|\leq 24\), all fixed, we have the following uniformly for $\mu\in\R$ and $q\in\Schw(\R):$
\begin{align}
\|(2 \pm \p)^s(2\kappa + \p)^\sigma \psi_\mu^\ell(\vk - \p)^{r} q \|_{L^p}\label{mult comm}
&\sim \| (2\pm\p)^s(2\kappa + \p)^\sigma (\vk - \p)^{r} \psi_\mu^\ell  q \|_{L^p}.
\end{align}
If both \(s,\sigma\in \Z\) then \eqref{mult comm} also holds for \(p\in \{1,\infty\}\).
\end{lemma}

\begin{proof}The Mikhlin multiplier theorem shows that the choice of $\pm$ signs is immaterial and we shall restrict attention to the $+$ case.

Both inequalities can be treated simultaneously through a slightly larger family of estimates involving two parameters $\theta,\nu\in\Z$.  Specifically, adopting the notation
$$
W= (2 + \p)^s (2\kappa + \p)^\sigma (\vk -\partial)^\nu
$$
it suffices to show that for each pair $\theta,\nu\in\Z$,
$$
\| W \psi^\ell (\vk -\partial)^\theta \psi^{-\ell} (\vk -\partial)^{-\theta} W^{-1} q \|_{L^p} \lesssim \| q \|_{L^p} .
$$
(In fact, just the two cases $(\theta,\nu)=(r,0)$ and $(\theta,\nu)=(-r,r)$ are truly needed.)

When $1<p<\infty$, complex interpolation allows us to restrict attention to the case where \(s,\sigma\in \Z\), which we do in what follows.

The next step is to perform additive commutations, moving each positive power of a differential operator toward its inverse, one factor at a time.  Proceeding in this fashion until all positive powers of said differential operators are exhausted leaves a very concrete (but combinatorially very messy) finite linear combination of products of operators from the following list:
$$
\psi^{-\ell}(\partial^m\psi^\ell), \ \psi^{\ell}(\partial^m\psi^{-\ell}),\ (2 + \p)^{-1},\ (2\kappa + \p)^{-1},\ (\vk - \p)^{-1},\ %
\text{and}\ \psi^\ell (\vk -\partial)^{-1} \psi^{-\ell},
$$ 
where $m$ is any integer satisfying $0\leq m \leq |\sigma|+|s|+|\nu|+|\theta|$.  In this way, we see that the proof will be complete if we can show that any operator on the list is $L^p$ bounded for every $1\leq p\leq \infty$.

Boundedness of the first two operators in the list is trivial given our choice of $\psi$.  Boundedness of the remaining operators can be deduced from their explicit kernels.  Indeed, $\psi^\ell (\vk - \p)^{-1}\psi^{-\ell}$ has kernel
\[
K(x,y) = \psi(x)^\ell\psi(y)^{-\ell} e^{\vk(x-y)}\bbo_{x<y}, \qtq{which satisfies} |K(x,y)|\lesssim_\ell e^{-\frac12\vk|x-y|}.
\]
Thus $L^p$ boundedness follows from Schur's test.\end{proof}

As we are dealing with a nonlinear equation, one needs to understand how to estimate products in our local smoothing spaces.   Due to the low regularity of the objects we are treating in this paper, each term in the product must itself satisfy local smoothing estimates in order for the product to be bounded.  This dictates the structure of our basic product estimates below.

\begin{lemma}\label{L:F product}
If \(0\leq \sigma<\frac12\), \(|\vk|,\kappa\geq 1\), \(h\in \R\), \(1\leq \ell\leq 12\) is an integer, and \(\phi'\in \Schw\) then
\begin{align}
|\vk|\|q\|_{F_\kappa^{\frac12}(h)} +  \|q\|_{F_\kappa^{\frac32}(h)} \sim \|(2\vk - \p)q\|_{F_\kappa^{\frac12}(h)}\label{LS transfer}
\end{align}
and we have the product estimates
\begin{align}
\|fg\|_{F_\kappa^{\frac12}(h)} &\lesssim |\vk|^{-\frac12}\Bigl[ \|f\|_{F_\kappa^{\frac12}(h)}\|(2\vk-\p)g\|_{E_{2\sigma,\vk}^\sigma} + \|f\|_{E_{2\sigma,\vk}^\sigma}\|(2\vk-\p)g\|_{F_\kappa^{\frac12}(h)}\Bigr],\label{LS product I}\\
\|fg\|_{F_\kappa^{\frac32}(h)} &\lesssim |\vk|^{-\frac12}\Bigl[\|f\|_{F_\kappa^{\frac32}(h)}\|(2\vk-\p)g\|_{E_{2\sigma,\vk}^\sigma}+
\|(2\vk-\p)f\|_{E_{2\sigma,\vk}^\sigma}\|g\|_{F_\kappa^{\frac32}(h)}\Bigr].\label{LS product II}
\end{align}
We also have the localization estimates
\begin{align}
\int \bigl\|\tfrac{\psi_{\mu}^\ell q}{\sqrt{4\kappa^2 - \p^2}}\bigr\|_{H^{\frac32}}^2\,e^{-\frac1{200}|h-\mu|}\,d\mu &\lesssim  \|q\|_{F_\kappa^{\frac12}(h)}^2,\label{LS loc}\\
\|\phi q\|_{F_\kappa^{\frac12}(h)}&\lesssim_\phi \|q\|_{F_\kappa^{\frac12}(h)}.\label{LS lock}
\end{align}
\end{lemma}

\begin{proof}
For \eqref{LS transfer}, we first use Lemma~\ref{l:mult comm} to obtain
\begin{align*}
&\|(2\vk - \p)q\|_{F_\kappa^{\frac12}(h)}^2\sim \int \bigl\|(2\vk - \p)\tfrac{\psi_{\mu}^{12}q}{\sqrt{4\kappa^2 - \p^2}}\bigr\|_{H^{\frac32}}^2\,e^{-\frac1{200}|h - \mu|}\,d\mu\\
&\sim 4|\vk|^2\int \bigl\|\tfrac{\psi_{\mu}^{12}q}{\sqrt{4\kappa^2 - \p^2}}\bigr\|_{H^{\frac32}}^2\,e^{-\frac1{200}|h - \mu|}\,d\mu + \int \bigl\|\bigl(\tfrac{\psi_{\mu}^{12}q}{\sqrt{4\kappa^2 - \p^2}}\bigr)'\bigr\|_{H^{\frac32}}^2\,e^{-\frac1{200}|h - \mu|}\,d\mu\\
&\sim |\vk|^2\|q\|_{F_\kappa^{\frac12}(h)}^2 + \|q\|_{F_\kappa^{\frac32}(h)}.
\end{align*}

Using space-translation invariance, it suffices to prove \eqref{LS loc} for \(h=0\). If \(\ell=12\), the claim follows from the definition. Otherwise, define \(T_{\mu,\nu}\colon L^2\to L^2\) to have integral kernel
\[
\tfrac{(4 + \xi^2)^{\frac34}}{(4\kappa^2 + \xi^2)^{\frac12}}\widehat{\psi_{\mu}^\ell \psi_{\nu}^{12}}(\xi - \eta)\tfrac{(4\kappa^2 + \eta^2)^{\frac12}}{(4 + \eta^2)^{\frac34}},
\]
and apply Schur's test to bound
\[
\|T_{\mu,\nu}\|_\op\lesssim\|\langle \xi \rangle^{\frac52} \widehat{\psi_{\mu}^\ell \psi_{\nu}^{12}}(\xi)\|_{L^1_\xi} \lesssim\|\psi_{\mu}^\ell \psi_{\nu}^{12}\|_{H^{\frac72}}
\lesssim e^{-\frac1{200}|\mu-\nu|}.
\]
We then apply \eqref{psi int} to obtain
\begin{align*}
\int \bigl\|&\tfrac{\psi_{\mu}^\ell q}{\sqrt{4\kappa^2 - \p^2}}\bigr\|_{H^{\frac32}}^2\,e^{-\frac1{200}|\mu|}\,d\mu\\
&\lesssim \iiint  \bigl\|\tfrac{\psi_{\mu}^\ell \psi_{\nu_1}^{24}q}{\sqrt{4\kappa^2 - \p^2}}\bigr\|_{H^{\frac32}}\bigl\|\tfrac{\psi_{\mu}^\ell \psi_{\nu_2}^{24}q}{\sqrt{4\kappa^2 - \p^2}}\bigr\|_{H^{\frac32}}\,e^{-\frac1{200}|\mu|}\, d\nu_1\, d\nu_2\,d\mu\\
&\lesssim \iiint \|T_{\mu,\nu_1}\|_{\op}\|T_{\mu,\nu_2}\|_{\op} \bigl\|\tfrac{\psi_{\nu_1}^{12} q}{\sqrt{4\kappa^2 - \p^2}}\bigr\|_{H^{\frac32}} \bigl\|\tfrac{\psi_{\nu_2}^{12} q}{\sqrt{4\kappa^2 - \p^2}}\bigr\|_{H^{\frac32}}\,e^{-\frac1{200}|\mu|}\, d\nu_1\, d\nu_2\,d\mu\\
&\lesssim \iiint e^{-\frac{|\nu_1|+ |\nu_2|}{400}}e^{-\frac{|\mu-\nu_1|+|\mu-\nu_2|}{400}} \bigl\|\tfrac{\psi_{\nu_1}^{12} q}{\sqrt{4\kappa^2 - \p^2}}\bigr\|_{H^{\frac32}} \bigl\|\tfrac{\psi_{\nu_2}^{12} q}{\sqrt{4\kappa^2 - \p^2}}\bigr\|_{H^{\frac32}}\, d\nu_1\, d\nu_2\,d\mu \\
&\lesssim \|q\|_{F_\kappa^{\frac12}(0)}^2.
\end{align*}
In the last step we first integrated in $\mu$ and then used $L^2$ boundedness of the resulting convolution operator.

Turning next to the product estimates \eqref{LS product I} and \eqref{LS product II}, we take \(s \in\{ \frac12,\frac32\}\), \(\mu\in\R\), and use \eqref{Bony} to estimate
\begin{align*}
&\bigl\|\tfrac 1{\sqrt{4\kappa^2 - \p^2}}(\psi_{\mu}^{12}fg)\bigr\|_{H^{s+1}}^2\\
&\qquad\lesssim \sum_N \tfrac{(1+N)^{2s+2}}{(\kappa+N)^2} \|P_{\sim N}(\psi_{\mu}^6f)\|_{L^2}^2 \|P_{\lesssim N}(\psi_{\mu}^6 g)\|_{L^\infty}^2\\
&\qquad\quad+\sum_N \tfrac{(1+N)^{2s+2}}{(\kappa+N)^2} \|P_{\lesssim N}(\psi_{\mu}^6 f)\|_{L^\infty}^2 \|P_{\sim N}(\psi_{\mu}^6 g)\|_{L^2}^2\\
&\qquad\quad + \sum_N \tfrac{(1+N)^{2s+2}}{(\kappa+N)^2} \Bigl[ \sum_{M>2N}N^{\frac12} \|P_M(\psi_{\mu}^6 f)\|_{L^2}\|P_{\sim M}(\psi_{\mu}^6 g)\|_{L^2}\Bigr]^2.
\end{align*}
Recalling \eqref{Sob}, \eqref{E transfer},  and \eqref{E loc}, we bound the first summand by
\begin{align*}
&\sum_N \tfrac{(1+N)^{2s+2}}{(\kappa+N)^2} \|P_N(\psi_{\mu}^6f)\|_{L^2}^2 \|P_{\lesssim N}(\psi_{\mu}^6 g)\|_{L^\infty}^2\\
&\qquad\lesssim |\vk|^{-1}\bigl\|\tfrac{\psi_{\mu}^6 f}{\sqrt{4\kappa^2 - \p^2}}\bigr\|_{H^{s+1}}^2\Bigl[|\vk|\|\psi_{\mu}^6g\|_{E_{2\sigma, \vk}^\sigma} + \|(\psi_{\mu}^6g)'\|_{E_{2\sigma, \vk}^{\sigma}}\Bigr]^2\\
&\qquad\lesssim |\vk|^{-1}\bigl\|\tfrac{\psi_{\mu}^6 f}{\sqrt{4\kappa^2 - \p^2}}\bigr\|_{H^{s+1}}^2\|(2\vk-\p)g\|_{E_{2\sigma, \vk}^\sigma}^2.
\end{align*}
For \eqref{LS product II}, we bound the second summand in a symmetric fashion, reversing the roles of \(f,g\). For \eqref{LS product I}, we instead apply Schur's test with \eqref{E transfer}, \eqref{E loc}, and \eqref{mult comm} to bound
\begin{align*}
&\sum_N \tfrac{(1+N)^3}{(\kappa+N)^2} \Bigl[ \sum_{M\lesssim N}M^{\frac12}\|P_M(\psi_{\mu}^6 f)\|_{L^2}\Bigr]^2 \|P_N(\psi_{\mu}^6g)\|_{L^2}^2\\
&\lesssim \bigl\|\tfrac{(2\vk - \p)(\psi_{\mu}^6g)}{\sqrt{4\kappa^2-\p^2}}\bigr\|_{H^{\frac32}}^2\!\!\sum_{M_1\leq M_2}\! \!\!\tfrac{(\vk+M_1)^{2\sigma} M_1^{\frac12-\sigma}M_2^{\frac12-\sigma}}{|\vk|^{2\sigma}(\vk+M_2)^{2-2\sigma}} \bigl\|P_{M_1}(\psi_\mu^6 f)\bigr\|_{E^\sigma_{2\sigma, \vk}} \bigl\|P_{M_2}(\psi_\mu^6 f)\bigr\|_{E^\sigma_{2\sigma, \vk}}\\
&\lesssim |\vk|^{-1}\bigl\|\tfrac{\psi_{\mu}^6(2\vk - \p)g}{\sqrt{4\kappa^2-\p^2}}\bigr\|_{H^{\frac32}}^2\|f\|^2_{E^\sigma_{2\sigma, \vk}}.
\end{align*}
The third summand is again bounded using Schur's test, \eqref{Sob}, \eqref{E transfer}, and \eqref{E loc}:
\begin{align*}
&\sum_N \tfrac{(1+N)^{2s+2}}{(\kappa+N)^2} \Bigl[ \sum_{M>2N}N^{\frac12} \|P_M(\psi_{\mu}^6 f)\|_{L^2}\bigl\|P_{\sim M}g\|_{L^2}\Bigr]^2\\
&\qquad\lesssim \|g\|_{B}^2\sum_{M_1\leq M_2} \tfrac{(1+M_1)^{s+1}(\kappa + M_2)M_1^{\frac12}}{(1 + M_2)^{s+1}(\kappa + M_1)M_2^{\frac12}}\bigl\|\tfrac{P_{M_1}(\psi_{\mu}^6 f)}{\sqrt{4\kappa^2 - \p^2}}\bigr\|_{H^{s+1}}\bigl\|\tfrac{P_{M_2}(\psi_{\mu}^6 f)}{\sqrt{4\kappa^2 - \p^2}}\bigr\|_{H^{s+1}}\\
&\qquad\lesssim |\vk|^{-1}\bigl\|\tfrac{\psi_{\mu}^6 f}{\sqrt{4\kappa^2 - \p^2}}\bigr\|_{H^{s+1}}^2\|(2\vk-\p)g\|_{E_{2\sigma, \vk}^\sigma}^2.
\end{align*}
The estimates \eqref{LS product I}, \eqref{LS product II} then follow from \eqref{LS loc}.

Applying Schur's test in Fourier variables, we find
$$
\sup_\mu \, \Bigl\| \tfrac{\langle\partial\rangle^{3/2}}{\sqrt{4\kappa^2-\partial^2}} \, \phi\psi_\mu^6 \, \tfrac{\sqrt{4\kappa^2-\partial^2}}{\langle\partial\rangle^{3/2}} \Bigr\|_{\op} \lesssim_\phi 1.
$$
In this way, we see that 
\[
\bigl\|\tfrac{\psi_\mu^{12}\phi q}{\sqrt{4\kappa^2 - \p^2}}\bigr\|_{H^{\frac32}}\lesssim_\phi \bigl\|\tfrac{\psi_\mu^6q}{\sqrt{4\kappa^2 - \p^2}}\bigr\|_{H^{\frac32}},
\]
and consequently, \eqref{LS lock} follows from \eqref{LS loc}.
\end{proof}

\begin{remark}\label{R:X alt} As a consequence of the proof of Lemma~\ref{L:F product}, we note that 
\begin{align}
\|q\|_{X_\kappa^{\frac12}}^2 \sim \sup_{\mu\in \R}\, \bigl\|\tfrac{\psi_\mu^{12}q}{\sqrt{4\kappa^2 - \p^2}}\bigr\|_{L^2_tH^{\frac32}}^2.\label{LS alt}
\end{align}
Indeed, the inequality
\[
\|q\|_{X_\kappa^{\frac12}}^2 \lesssim \sup_{\mu\in \R}\, \bigl\|\tfrac{\psi_\mu^{12}q}{\sqrt{4\kappa^2 - \p^2}}\bigr\|_{L^2_tH^{\frac32}}^2
\]
follows immediately from the definition. For the converse inequality, we argue as in \eqref{LS loc}: in view of \eqref{psi int},
\begin{align*}
\bigl\|\tfrac{\psi_\mu^{12}q}{\sqrt{4\kappa^2 - \p^2}}\bigr\|_{L^2_tH^{\frac32}}^2&\lesssim \iint \|T_{\mu,\nu_1}\|_\op\|T_{\mu,\nu_2}\|_\op\bigl\|\tfrac{\psi_{\nu_1}^{12} q}{\sqrt{4\kappa^2 - \p^2}}\bigr\|_{L^2_tH^{\frac32}}\bigl\|\tfrac{\psi_{\nu_2}^{12} q}{\sqrt{4\kappa^2 - \p^2}}\bigr\|_{L^2_tH^{\frac32}}\,d\nu_1\,d\nu_2\\
&\lesssim \|q\|_{L^2_tF_\kappa^{\frac12}(h)}^2.
\end{align*}
\end{remark}

\section{Green's functions and microscopic conservation laws}\label{S:Green}

By the Green's function, we mean the integral kernel associated to the inverse of the Lax operator presented in \eqref{L defn}.  It is not a given that this operator is invertible; we will rely on the subtle interplay between the spectral parameter $\kappa$ and the equicontinuity properties of $q$.  This same issue was discussed in the introduction in connection with making sense of $A(\kappa;q)$.  Indeed, it formed the central rationale for introducing the notion of a $\delta$-good subset of $L^2(\R)$; see Definition~\ref{D:good}.  Let us begin our discussion by revisiting the construction of $A(\kappa;q)$.

We subsequently take \(0<\sigma<\tfrac12\) to be fixed. If $Q$ is a $\delta$-good subset of $L^2(\R)$, then Corollary~\ref{C:large kappa} shows that
\begin{align}\label{goody yum yum}
|\kappa|^{\frac12} \bigl\| \Lambda(q;\kappa) \bigr\|_{\op} \lesssim \| q \|_{E_\sigma^\sigma} \leq \delta
	\quad\text{uniformly for $|\kappa|\geq 1$ and $q\in Q$.}
\end{align}
As shown in \cite[Lemma~5.1]{KNV}, it follows that if $\delta$ is sufficiently small then
\begin{align}
A(\kappa;q) &= -\sgn(\kappa) \log \det\bigl[ 1 - i\kappa (\kappa-\partial)^{-1} q (\kappa+\partial)^{-1} \bar q\bigr ] \notag\\
&= \sgn(\kappa) \sum_{\ell\geq 1} \tfrac{1}{\ell} \tr\bigl\{ (i \kappa \Lambda \Gamma )^{\ell}\bigr\} \label{A series}
\end{align}
defines a real-analytic function of $\kappa$ and $q$.  Moreover, the domain of this function includes all $|\kappa|\geq 1$ and an $L^2$-neigborhood of $Q$.

It is important to define $A(\kappa;q)$ in such a neighborhood of $Q$ (rather than just on $Q$) to ensure that the functional derivatives are well defined.  We find that
\begin{align}\label{A derivs}
&\int \tfrac{\delta A}{\delta q} f + \tfrac{\delta A}{\delta \bar q} \bar f \,dx 
= \sgn(\kappa) \sum_{m\geq 0} \tr\Bigl\{ \bigl(i\kappa\Lambda(q)\Gamma(q)\bigr)^{m}i\kappa[\Lambda(f)\Gamma(q)+\Lambda(q)\Gamma(f)] \Bigr\}.\!\!
\end{align}
As Lemma~\ref{L:HS} and \eqref{goody yum yum} show, this series defines a bounded linear functional on $f\in L^2$ and correspondingly the functional derivatives exist as $L^2$ functions.

Duality also gives an efficient way to introduce the functions $\gamma(\kappa;q)$, $g_{12}(\kappa;q)$, and $g_{21}(\kappa;q)$ that will be of central importance in what follows: for $a,b,c\in L^2$,
\begin{align}\label{geegee}
\int g_{21}\;\!  b + g_{12}\;\! c +  \gamma\;\!a \,dx = \sgn(\kappa) \tr\Bigl\{ \bigl[\begin{smallmatrix}a\, & \ \ b\\ c\,& \,-a\end{smallmatrix}\bigr] \bigl[ L^{-1} - L_0^{-1}\bigr] \Bigr\} .
\end{align}
Here $L_0$ denotes the Lax operator \eqref{L defn} with $q\equiv 0$.

Lemma~\ref{L:HS} and \eqref{goody yum yum} guarantee that a Neumann expansion of the right-hand side of \eqref{geegee} yields a convergent series for all $a,b,c\in L^2(\R)$.  On comparing these series with those of \eqref{A derivs}, we find
\begin{align}\label{FderivA}
\tfrac{\delta A}{\delta \bar q} = i\sqrt{\kappa}\, g_{12} \qtq{and} \tfrac{\delta A}{\delta q} = -\sqrt{\kappa}\, g_{21}.
\end{align}

It is evident from \eqref{geegee} that $g_{12}$, $g_{21}$, and $\gamma$ are closely connected with the matrix Green's function evaluated on the diagonal (i.e. at the coincidence of the two spatial points).  For the continuity needed to make sense of this directly, see see \cite[Prop.~3.1]{harropgriffiths2020sharp}.

For simplicity of exposition, the discussion above only constructed $g_{12}$, $g_{21}$, and $\gamma$ as $L^2$ functions.  By estimating more carefully (as was done in \cite{KNV}), one finds that the series defining \(g_{12}\), \(g_{21}\), and \(\gamma\) converge in \(H^1\) to real analytic functions of \(q\); moreover, these functions are Schwartz whenever \(q\in \Schw\).

Direct computations also reveal certain basic identities among these functions; see \cite{harropgriffiths2020sharp} or \cite{tang2020microscopic}.  Concretely, we have
\begin{align}\label{sym}
g_{12}(\kappa) = -\overline{g_{21} (-\kappa)}, \quad \gamma(\kappa) = \overline{\gamma (-\kappa)},  \qtq{and} A(\kappa) =  -\overline{A(-\kappa)}
\end{align}
as well as
\begin{align}
g_{12}'&=2\kappa g_{12} -\kappa^{\frac12}q(\gamma +1)\label{g12 deriv}\\
g_{21}'&=-2\kappa g_{21} -i\kappa^{\frac12}\bar q(\gamma +1) \label{g21 deriv}\\
\gamma'&= 2\kappa^{\frac12}(qg_{21}+i\bar q g_{12}). \label{gamma deriv}
\end{align}
Lastly, we have the quadratic identity
\begin{align}\label{quadratic id}
2g_{12}g_{21} + \tfrac12 \gamma^2 +\gamma=0,
\end{align}
which can be proved by differentiating the left-hand side and applying \eqref{g12 deriv}--\eqref{gamma deriv}.

Using these relations, we may write
\begin{equation}\label{useful IDs}
\begin{gathered}
g_{12} = \tfrac{\sqrt\kappa}{2\kappa-\partial} \bigl [ q(\gamma+1)\bigr], \qtq{}g_{21} = \tfrac{-i\sqrt\kappa}{2\kappa+\partial} \bigl [ \bar q(\gamma+1)\bigr], \qtq{} \gamma = - \tfrac{4g_{12}g_{21}}{2+\gamma},\\
\tfrac{g_{12}}{2+\gamma} = \tfrac{\sqrt\kappa}{2(2\kappa-\partial)} \bigl[ q + 4i\bar q\, \bigl(\tfrac{g_{12}}{2+\gamma}\bigr)^2 \bigr],
\qtq{and} \tfrac{g_{21}}{2+\gamma} = \tfrac{-i\sqrt\kappa}{2(2\kappa+\partial)} \bigl[ \bar q - 4iq\, \bigl(\tfrac{g_{21}}{2+\gamma}\bigr)^2 \bigr].
\end{gathered}
\end{equation}
While it is very elementary to check these last two identities, it is much less obvious that they are key to efficiently treating a lot of what follows.

To prove our theorems, we require bounds for these functions in the \(B\), \(E\), and \(F\) spaces introduced in Section~\ref{S:Pre}.  We start with the following:

\begin{lemma}\label{l:g H}
For $\delta$ sufficiently small and $Q\subset \Schw$ a $\delta$-good set, the following estimates hold uniformly for \(q\in Q\) and $|\kappa|\geq 1:$
\begin{align}
\|\gamma(\kappa;q)\|_B\lesssim\|q\|_{E_{2\sigma, \kappa}^\sigma}^2&\leq \delta^2,\label{gamma in sup norm}\\
\|g_{12}(\kappa;q)\|_B + \|g_{21}(\kappa;q)\|_B &\lesssim \|q\|_{E_{2\sigma, \kappa}^\sigma},\label{g12 in sup norm}\\
\bigl\|\tfrac{g_{12}(\kappa;q)}{2 + \gamma(\kappa;q)}\bigr\|_B + \bigl\|\tfrac{g_{21}(\kappa;q)}{2 + \gamma(\kappa;q)}\bigr\|_B &\lesssim\|q\|_{E_{2\sigma, \kappa}^\sigma}.\label{g12 frac in sup norm}
\end{align}
\end{lemma}
\begin{proof}
To prove \eqref{gamma in sup norm}, we argue by duality.  To this end, we test against functions $f$ satisfying $\sup_{M\in 2^\Z} M^{-\frac12}\|f_M\|_{L^2}\leq 1$ and employ 
\begin{align*}
    \langle f, \gamma(\kappa;q)\rangle =& \sgn(\kappa)\sum_{\ell\geq 1}(i\kappa)^{\ell} \tr\Bigl\{ \bigl[ (\kappa-\partial)^{-1}q (\kappa+\partial)^{-1}\bar q \bigr]^{\ell} (\kappa-\partial)^{-1} \bar f \Bigr\}\\
    &+ \sgn(\kappa)\sum_{\ell\geq 1}(i\kappa)^{\ell} \tr\Bigl\{ \bigl[ (\kappa+\partial)^{-1}\bar q(\kappa-\partial)^{-1}q \bigr]^{\ell} (\kappa+\partial)^{-1} \bar f \Bigr\}.
\end{align*}

Let us first observe that by \eqref{op by E} and \eqref{goodness}, there exists $C\geq 1$ so that
$$
\sup_{N\in 2^\Z} \,|\kappa|^{\frac12} \sum_{M\leq N} \|\Lambda (q_M)\|_\op\leq C \delta, 
$$
uniformly in $|\kappa|\geq 1$.  Note also that by Lemma~\ref{L:op est} and our assumptions on $f$,
$$
\sum_{M\leq N} \|\Lambda (f_M)\|_\op \lesssim  \tfrac{N}{|\kappa| +N}\log^{\frac32}\bigl(4+\tfrac{N^2}{\kappa^2}\bigr)
 \ \ \text{and} \ \ %
 \|\Lambda(f_N)\|_{\hs} \lesssim  \sqrt{\tfrac{N}{|\kappa| +N} \log\bigl(4+\tfrac{N^2}{\kappa^2}\bigr)}  .
$$

Decomposing into Littlewood--Paley pieces, using that the two highest frequencies must be comparable, together with Lemma~\ref{L:op est} and the preceding estimates,  we find that
\begin{align*}
&\left|(i\kappa)^{\ell}\tr\Bigl\{ \bigl[ (\kappa-\partial)^{-1}q (\kappa+\partial)^{-1}\bar q \bigr]^{\ell} (\kappa-\partial)^{-1} \bar f \Bigr\}\right|\\
&\quad  \lesssim (C \delta)^{2(\ell-1)} \!
	\sum_{N_1\sim N_2}  \tfrac{|\kappa| N_1}{|\kappa| +N_1}\log^{\frac32}\bigl(4+\tfrac{N_1^2}{\kappa^2}\bigr) \|\Lambda(q_{N_1})\|_{\hs}\|\Lambda(q_{N_2})\|_{\hs}  \\
&\qquad + (C \delta)^{2(\ell-1)}  \!
	\sum_{N_1\sim N_2 \geq N_3} |\kappa| \|\Lambda(q_{N_1})\|_{\hs}\|\Lambda(f_{N_2})\|_{\hs}  \|\Lambda(q_{N_3})\|_{\op} \\
&\quad \lesssim (C \delta)^{2(\ell-1)} \!
	\sum_{N_1\sim N_2} \tfrac{|\kappa| N_1}{(|\kappa| +N_1)^2} \log^{\frac52}\bigl(4+\tfrac{N_1^2}{|\kappa|^2}\bigr)\|q_{N_1}\|_{L^2}\|q_{N_2}\|_{L^2} \\
&\quad \quad + (C \delta)^{2(\ell-1)} \! \sum_{N_1\sim N_2} \tfrac{|\kappa|^{\frac12}N_1^{1-\sigma}}{(|\kappa| +N_1)^{\frac32-\sigma}}
	 \log\bigl(4+\tfrac{N_1^2}{|\kappa|^2}\bigr)\|q_{N_1}\|_{L^2} \|q\|_{E_{2\sigma, \kappa}^\sigma}\\
&\quad\lesssim (C \delta)^{2(\ell-1)} \|q\|_{E_{2\sigma, \kappa}^\sigma}^2 .
\end{align*}
Choosing $\delta$ sufficiently small, we may sum in $\ell\geq 1$ and so deduce \eqref{gamma in sup norm}.

Further reducing \(\delta\), if necessary, the estimates \eqref{gamma in sup norm} and \eqref{E:alg} allow us to sum the geometric series and so obtain
\begin{equation}\label{gamma frac in B}
\bigl\|\tfrac1{2+\gamma}\bigr\|_B\lesssim 1.
\end{equation}

Regarding \eqref{g12 in sup norm}, we use \eqref{Sob}, \eqref{E transfer}, \eqref{useful IDs}, and \eqref{E product I} to obtain
\[
\|g_{12} \|_{B} \lesssim |\kappa|^{-\frac12} \| (2\kappa-\p)g_{12} \|_{E_{2\sigma, \kappa}^\sigma}
	\lesssim \bigl[ 1 + \| \gamma\|_B \bigr] \| q \|_{E_{2\sigma, \kappa}^\sigma}.
\]
The bound  \eqref{g12 in sup norm} then follows from \eqref{gamma in sup norm} and \eqref{sym}.

The estimate \eqref{g12 frac in sup norm} follows directly from \eqref{g12 in sup norm}, \eqref{gamma frac in B}, and \eqref{E:alg}.
\end{proof}

\begin{lemma}\label{l:g E}
For $\delta$ sufficiently small and $Q\subset \Schw$ a $\delta$-good set, the following estimates hold uniformly for \(q\in Q\) and \(|\vk|,\kappa\geq 1:\)
\begin{align}
|\vk|\|g_{12}(\vk)\|_{E_{s, \kappa}^\sigma} + \|g_{12}'(\vk)\|_{E_{s,\kappa}^{\sigma}}&\lesssim |\vk|^{\frac12}\|q\|_{E_{s, \kappa}^\sigma},\label{g12 in E}\\
|\vk|\|\gamma(\vk)\|_{E_{s, \kappa}^\sigma}  + \|\gamma'(\vk)\|_{E_{s,\kappa}^{\sigma}} &\lesssim |\vk|^{\frac12}\|q\|_{E_{s, \kappa}^\sigma} \|q\|_{E_{2\sigma, \vk}^\sigma}, \label{gamma in E}\\
|\vk|\bigl\|\tfrac{g_{12}(\vk)}{2 + \gamma(\vk)}\bigr\|_{E_{s, \kappa}^\sigma}  + \bigl\|\bigl(\tfrac{g_{12}}{2 + \gamma}\bigr)'(\vk)\bigr\|_{E_{s,\kappa}^{\sigma}}&\lesssim |\vk|^{\frac12}\|q\|_{E_{s, \kappa}^\sigma}.\label{g12 frac in E}
\end{align}
Moreover, in view of \eqref{sym}, $g_{21}$ also satisfies the estimates \eqref{g12 in E} and \eqref{g12 frac in E}.
\end{lemma}

\begin{proof}
In view of \eqref{g12 deriv}, \eqref{E product I} and \eqref{gamma in sup norm},
\begin{align*}
\|(2\vk -\p)g_{12}\|_{E_{s,\kappa}^\sigma} &\lesssim |\vk|^{\frac12}\|q(1+\gamma)\|_{E_{s,\kappa}^\sigma}
\lesssim |\vk|^{\frac12}\|q\|_{E_{s,\kappa}^\sigma}\Bigl[1 + \|\gamma\|_{B}\Bigr]\lesssim |\vk|^{\frac12}\|q\|_{E_{s,\kappa}^\sigma}.
\end{align*}
The estimate \eqref{g12 in E} then follows from \eqref{E transfer}.

Recalling $-\gamma=\tfrac{4g_{12}g_{21}}{2+\gamma}$ from \eqref{useful IDs}, the estimates \eqref{g12 in E} and \eqref{g12 frac in sup norm} show that
\begin{align*}
 |\vk| \|\gamma(\vk)\|_{E_{s,\kappa}^\sigma} \lesssim  |\vk| \| g_{12}(\vk) \|_{E_{s,\kappa}^\sigma} \bigl\| \tfrac{g_{21}(\vk)}{2+\gamma(\vk)} \bigr\|_B
 	\lesssim |\vk|^{\frac12} \|q\|_{E_{s, \kappa}^\sigma}  \|q\|_{E_{2\sigma, \vk}^\sigma}.
\end{align*}
To complete the proof of \eqref{gamma in E}, we complement this with the estimate
\[
\|\gamma'(\vk)\|_{E_{s,\kappa}^{\sigma}}\lesssim |\vk|^{\frac12}\|q\|_{E_{2\sigma,\vk}^\sigma}\|q\|_{E_{s,\kappa}^\sigma},
\]
for which we employed \eqref{gamma deriv}, \eqref{E product II}, and \eqref{g12 in E}. 

Using \eqref{E product I}, the fact that $B$ is an algebra, and the estimates \eqref{g12 frac in sup norm}, \eqref{gamma frac in B}, \eqref{g12 in E}, \eqref{gamma in E}, and \eqref{goodness}, we obtain
\begin{align*}
\bigl\| (2\vk-\p)\tfrac{g_{12}}{2+\gamma} \bigr\|_{E^\sigma_{s,\kappa}}
&\lesssim \bigl\| (2\vk-\p) g_{12}\bigr\|_{E^\sigma_{s,\kappa}} \bigl\| \tfrac1{2+\gamma} \bigr\|_{B}
	+ \bigl\| \gamma' \bigr\|_{E^\sigma_{s,\kappa}} \bigl\| \tfrac{g_{12}}{(2+\gamma)^2} \bigr\|_{B}\\
&\lesssim |\vk|^{\frac12} \|q\|_{E^\sigma_{s,\kappa}} \bigl(1+ \|q\|_{E^{\sigma}_{2\sigma, \vk}}^2\bigr)\lesssim |\vk|^{\frac12} \|q\|_{E^\sigma_{s,\kappa}}.
\end{align*}
The estimate \eqref{g12 frac in E} now follows from \eqref{E transfer}.
\end{proof}

\begin{lemma}\label{l:g F}
For $\delta$ sufficiently small and $Q\subset \Schw$ a $\delta$-good set, the following estimates hold uniformly for \(q\in Q\), \(|\vk|,\kappa\geq 1\), and \(h\in \R:\)
\begin{align}
|\vk|\|g_{12}(\vk)\|_{F_\kappa^{\frac12}(h)} + \|g_{12}(\vk)\|_{F_\kappa^{\frac32}(h)}&\lesssim |\vk|^{\frac12}\|q\|_{F_\kappa^{\frac12}(h)},\label{g12 in LS}\\
|\vk|\|\gamma(\vk)\|_{F_\kappa^{\frac12}(h)} + \|\gamma(\vk)\|_{F_\kappa^{\frac32}(h)} &\lesssim |\vk|^{\frac12} \|q\|_{F_\kappa^{\frac12}(h)}\|q\|_{E_{2\sigma,\vk}^\sigma},\label{gamma in LS}\\
|\vk|\bigl\|\tfrac{g_{12}(\vk)}{2 + \gamma(\vk)}\bigr\|_{F_\kappa^{\frac12}(h)} + \bigl\|\tfrac{g_{12}(\vk)}{2 + \gamma(\vk)}\bigr\|_{F_\kappa^{\frac32}(h)}&\lesssim |\vk|^{\frac12}\|q\|_{F_\kappa^{\frac12}(h)}.\label{g12 frac in LS}
\end{align}

\end{lemma}
\begin{proof}
Using the quadratic identity \eqref{quadratic id} together with \eqref{LS transfer}, \eqref{LS product I}, and \eqref{LS product II}, followed by \eqref{g12 in E} and \eqref{gamma in E} we get
\begin{align*}
&|\vk|\|\gamma(\vk)\|_{F_\kappa^{\frac12}(h)} + \|\gamma(\vk)\|_{F_\kappa^{\frac32}(h)}\\
&\lesssim |\vk|^{-\frac12} \|(2\vk-\p)\gamma\|_{E^\sigma_{2\sigma, \vk}}\|(2\vk-\p)\gamma\|_{F_\kappa^{\frac12}(h)}\\
&\quad + |\vk|^{-\frac12} \max_{\pm \varkappa}\Bigl[|\vk|\|g_{12}\|_{E_{2\sigma, \kappa}^\sigma} + \|g_{12}'\|_{E_{2\sigma,\kappa}^{\sigma}}\Bigr]\max_{\pm \vk}\Bigl[|\vk|\|g_{12}\|_{F_\kappa^{\frac12}(h)} + \|g_{12}\|_{F_\kappa^{\frac32}(h)}\Bigr]\\
&\lesssim \|q\|_{E_{2\sigma,\vk}^\sigma}^2\|(2\vk-\p)\gamma\|_{F_\kappa^{\frac12}(h)} +\|q\|_{E_{2\sigma,\vk}^\sigma}\max_{\pm \vk}\Bigl[|\vk|\|g_{12}\|_{F_\kappa^{\frac12}(h)} + \|g_{12}\|_{F_\kappa^{\frac32}(h)}\Bigr].
\end{align*}
Using \eqref{goodness} we deduce
\begin{align}\label{gamma by g}
\|(2\vk-\p)\gamma\|_{F_\kappa^{\frac12}(h)}\lesssim \|q\|_{E_{2\sigma,\vk}^\sigma}\max_{\pm \vk}\Bigl[|\vk|\|g_{12}\|_{F_\kappa^{\frac12}(h)} + \|g_{12}\|_{F_\kappa^{\frac32}(h)}\Bigr]
\end{align}
and so \eqref{gamma in LS} will follow from \eqref{g12 in LS}.

From \eqref{gamma by g} and \eqref{g12 deriv} together with \eqref{LS transfer}, \eqref{LS product I}, and \eqref{gamma in E}, we obtain
\begin{align*}
&\max_{\pm \vk}\Bigl[|\vk|\|g_{12}\|_{F_\kappa^{\frac12}(h)} + \|g_{12}\|_{F_\kappa^{\frac32}(h)}\Bigr]\\
&\lesssim |\vk|^{\frac12}\|q(1+\gamma)\|_{F_\kappa^{\frac12}(h)}\\
&\lesssim |\vk|^{\frac12}\|q\|_{F_\kappa^{\frac12}(h)} +  \|q\|_{F_\kappa^{\frac12}(h)}  \|(2\vk-\p)\gamma\|_{E^\sigma_{2\sigma, \vk}} + \|q\|_{E^\sigma_{2\sigma, \vk}}\|(2\vk-\p)\gamma\|_{F_\kappa^{\frac12}(h)}\\
&\lesssim  |\vk|^{\frac12}\|q\|_{F_\kappa^{\frac12}(h)}\bigl(1+ \|q\|_{E^{\sigma}_{2\sigma, \vk}}^2\bigr) + \|q\|_{E_{2\sigma,\vk}^\sigma}^2 \max_{\pm \vk}\Bigl[|\vk|\|g_{12}\|_{F_\kappa^{\frac12}(h)} + \|g_{12}\|_{F_\kappa^{\frac32}(h)}\Bigr].
\end{align*}
The estimate \eqref{g12 in LS} then follows from \eqref{goodness}.

It remains to prove \eqref{g12 frac in LS}.   As a preliminary, let us pause to observe that
\begin{align*}
(2\vk-\p) \bigl(\tfrac{g_{12}}{2+\gamma}\bigr)^2 = \tfrac{2 g_{12}}{2+\gamma} \cdot (\vk-\p) \tfrac{g_{12}}{2+\gamma}.
\end{align*}
Using \eqref{E product I}, \eqref{g12 frac in sup norm}, and \eqref{g12 frac in E}, we get
\begin{align}\label{10:11}
\bigl\|(2\vk-\p)\bigl(\tfrac{g_{12}}{2+\gamma}\bigr)^2\bigr\|_{E^\sigma_{2\sigma,\vk}} 
\lesssim \bigl\|\tfrac{g_{12}}{2+\gamma} \bigr\|_B \bigl\| (\vk-\p) \tfrac{g_{12}}{2+\gamma}\bigr\|_{E^\sigma_{2\sigma,\vk}}
\lesssim |\vk|^{\frac12} \|q\|_{E_{2\sigma,\vk}^\sigma}^2,
\end{align}
while by \eqref{LS product I} and \eqref{g12 frac in E}, 
\begin{align}\label{10:12}
\bigl\|(2\vk-\p)\bigl(\tfrac{g_{12}}{2+\gamma}\bigr)^2\bigr\|_{F_\kappa^{\frac12}} 
&\lesssim |\vk|^{-\frac12}\bigl\| (2\vk-\p) \tfrac{g_{12}}{2+\gamma}\bigr\|_{E^\sigma_{2\sigma,\vk}}\bigl\|(2\vk-\p)\tfrac{g_{12}}{2+\gamma}\bigr\|_{F_\kappa^{\frac12}} \notag\\
&\lesssim  \|q\|_{E_{2\sigma,\vk}^\sigma}\bigl\|(2\vk-\p)\tfrac{g_{12}}{2+\gamma}\bigr\|_{F_\kappa^{\frac12}}.
\end{align}
Using \eqref{LS transfer} together with \eqref{useful IDs} followed by \eqref{LS product I}, \eqref{10:11}, and \eqref{10:12} we obtain
\begin{align*}
&\text{LHS\eqref{g12 frac in LS}}
	\lesssim \bigl\| (2\vk-\p)\tfrac{g_{12}}{2+\gamma} \bigr\|_{F_\kappa^{\frac12}}\\
&\quad\lesssim |\vk|^{\frac12}\|q\|_{F_\kappa^{\frac12}}+ \|q\|_{F_\kappa^{\frac12}}\bigl\|(2\vk-\p)\bigl(\tfrac{g_{12}}{2+\gamma}\bigr)^2\bigr\|_{E^\sigma_{2\sigma,\vk}} +   \|q\|_{E^\sigma_{2\sigma,\vk}}\bigl\|(2\vk-\p)\bigl(\tfrac{g_{12}}{2+\gamma}\bigr)^2\bigr\|_{F_\kappa^{\frac12}}\\
&\quad\lesssim |\vk|^{\frac12}\|q\|_{F_\kappa^{\frac12}} \bigl[ 1+ \|q\|_{E_{2\sigma,\vk}^\sigma}^2\bigr]   +\|q\|_{E_{2\sigma,\vk}^\sigma}^2\bigl\|(2\vk-\p)\tfrac{g_{12}}{2+\gamma}\bigr\|_{F_\kappa^{\frac12}}.
\end{align*}
Using \eqref{goodness}, we may absorb the second term on the right-hand side above into the left-hand side, thus settling \eqref{g12 frac in LS}.
\end{proof}

A consequence of Lemmas~\ref{l:g E} and ~\ref{l:g F} is the following:
\begin{corollary}\label{c:g}
Fix $\delta$ sufficiently small and let $Q\subset \Schw$ be a $\delta$-good set.  All functions $f$ from the following list $($and so finite linear combinations thereof $):$ 
\begin{equation}\label{f list}
\begin{gathered}
q,\quad\bar q,\\
\tfrac{2\vk\pm\p}{\sqrt\vk} g_{12}(\vk),\quad\tfrac{2\vk\pm\p}{\sqrt\vk} g_{21}(\vk),\quad \tfrac{2\vk\pm\p}{\sqrt\vk} \tfrac{g_{12}(\vk)}{2 + \gamma(\vk)},\quad\tfrac{2\vk\pm\p}{\sqrt\vk}\tfrac{g_{21}(\vk)}{2 + \gamma(\vk)},\\
\tfrac{2\kappa\pm\p}{\sqrt\kappa} g_{12}(\kappa),\quad\tfrac{2\kappa\pm\p}{\sqrt\kappa} g_{21}(\kappa),\quad \tfrac{2\kappa\pm\p}{\sqrt\kappa} \tfrac{g_{12}(\kappa)}{2 + \gamma(\kappa)},\quad\tfrac{2\kappa\pm\p}{\sqrt\kappa}\tfrac{g_{21}(\kappa)}{2 + \gamma(\kappa)},
\end{gathered}
\end{equation}
satisfy the estimates
\begin{equation}\label{admissible}
\begin{gathered}
\|f\|_{L^2}\lesssim \|q\|_{L^2},\quad \|f\|_{E_{2\sigma,\vk}^\sigma}\lesssim \|q\|_{E_{2\sigma,\vk}^\sigma},\quad\|f\|_{E_\sigma^\sigma}\lesssim \|q\|_{E_\sigma^\sigma},\\
\|f\|_{F^{\frac12}(h)}\lesssim \|q\|_{F^{\frac12}(h)},\qquad\|f\|_{F_\kappa^{\frac12}(h)}\lesssim \|q\|_{F_\kappa^{\frac12}(h)}
\end{gathered}
\end{equation}
uniformly for \(q\in Q\), \(|\vk|,|\kappa|\geq 1\), and \(h\in \R\).
\end{corollary}

As discussed in the introduction, multi-parameter local smoothing estimates are essential for our analysis.  As we are in the non-perturbative regime, the only reasonable approach to proving such estimates is via monotonicity identities.  In all examples that we are aware of, such monotonicity identities stem from a proper understanding of conserved densities and their corresponding currents.  This line of thinking leads inevitably to the problem of finding a microscopic representation for the conservation of $A(\vk;q)$.

It is invariably easy to find microscopic representations for conserved quantities that are polynomial in the underlying field (and its derivatives), such as the mass or energy. However, even in these simple cases there is no universal algorithm for finding such microscopic laws; indeed, this is an ill-posed problem --- the corresponding cohomology class does not have a unique representative.

When the conserved quantity in question is more complex, discovering a microscopic representation becomes truly challenging. All the more so when we need our representative to be coercive, if it is to be useful. This is the case for \(A(\vk;q)\), which is defined as the logarithm of a Fredholm determinant or as an infinite series of traces; see \eqref{A series}. Structurally, each of these traces is a paraproduct in $q$. Nevertheless, just such a microscopic representation was presented in \cite{tang2020microscopic} based on the density
\begin{equation}\label{rho}
\rho(\vk;q) := q\cdot \frac{ig_{21}(\vk;q) }{\sqrt\vk(2+\gamma(\vk;q))} + \bar q \cdot \frac{g_{12}(\vk;q)}{\sqrt\vk(2+\gamma(\vk;q))}.
\end{equation}

In finding \eqref{rho}, the authors of \cite{tang2020microscopic} were very much guided by the analogous form for the AKNS--ZS hierarchy discovered in \cite{harropgriffiths2020sharp}.  The analogue for KdV was found in \cite{MR3990604}, although this is of little assistance. Indeed, these few examples would lead one to believe that the answer will always be a rational function of matrix elements of the diagonal Green's function; this notion is refuted in \cite{MR4320535}.

Once these densities have been discovered, it is not fundamentally difficult to derive the corresponding current (though it may require considerable labour) because the time derivative of the Green's function may be deduced from the Lax pair representation of the flow.  For example, under the \eqref{DNLS} flow,
\begin{equation}\label{dt g12}\begin{aligned}
i \tfrac{d}{dt}g_{12} &= - \bigl[4\vk^2 -2i\vk|q|^2\bigr] g_{12} + \bigl[2\vk^{\frac32} q+i\vk^{\frac12}|q|^2q +\vk^{\frac12}q'\bigr] (1+\gamma)  \\
&=- g_{12}'' - i \bigl( 2|q|^2 g_{12}' + i q^2 g_{21}'\bigr).
\end{aligned}\end{equation}
We include this to illustrate the point made in the introduction that (unlike for NLS and mKdV) this change of variables alone does not allow us to treat $q\in L^2$.  Concretely, we note that for $q\in L^2$ one cannot make sense of the term $|q|^2q$ appearing in the former expression as a distribution.  This cannot be remedied by the other terms because they are distributions!

The following proposition gives the currents associated to the density \eqref{rho} and was proved in~\cite{tang2020microscopic}:

\begin{proposition}\label{P:micro laws} Let $Q\subset \Schw$ be $\delta$-good for $\delta$ sufficiently small.  Under the \eqref{DNLS} flow, we have that $\partial_t \rho(\vk) + \partial_x j_\DNLS(\vk)=0$ for all \(|\vk|\geq 1\), where
\begin{align}\label{DNLS current}
j_\DNLS&=(|q|^2-2i\vk)\rho  +\frac{q'g_{21} +i\bar q'g_{12}}{\sqrt\vk(2+\gamma)} +i |q|^2\notag\\
& = \tfrac{1}{\sqrt\vk}\tfrac{g_{21}}{2+\gamma}\cdot  (2\vk+\partial+i|q|^2)q - \tfrac{i}{\sqrt\vk}\tfrac{g_{12}}{2+\gamma}\cdot  (2\vk-\partial+i|q|^2)\bar q+i |q|^2.
\end{align}
Likewise, for $\kappa\geq1$, $\partial_t \rho(\vk) + \partial_x j_\diff(\vk,\kappa)=0$ under the $H- H_\kappa$ flow.  Here,
\begin{align}\label{j diff}
j_\diff(\vk,\kappa)
&=\tfrac{1}{\sqrt{\vk}} \bigl( \tfrac{g_{21}}{2+\gamma}\bigr)(\vk)\Bigl[ (2\vk+\p+i|q|^2)q - 2\kappa^{\frac52} \bigl( \tfrac{g_{12}(\kappa)}{\kappa-\vk} -\tfrac{ig_{12}(-\kappa)}{\kappa+\vk} \bigr)\Bigr]\\
&\quad -\tfrac{i}{\sqrt{\vk}} \bigl( \tfrac{g_{12}}{2+\gamma}\bigr)(\vk)\Bigl[ (2\vk-\p+i|q|^2)\bar q - 2\kappa^{\frac52} \bigl( \tfrac{g_{21}(-\kappa)}{\kappa+\vk} +\tfrac{ig_{21}(\kappa)}{\kappa-\vk} \bigr)\Bigr]\notag\\
&\quad +i|q|^2 -\tfrac{\kappa^2}{\kappa-\vk} \gamma(\kappa) + \tfrac{\kappa^2}{\kappa+\vk} \gamma(-\kappa).\notag
\end{align}
\end{proposition}

In our application of the microscopic conservation laws to the proof of local smoothing, we require a detailed understanding of the structure of the lower order terms (in powers of \(q\)) of the currents. To this end, we adopt the notation from~\cite{harropgriffiths2020sharp} of using square brackets to identify specific terms in power series expansions:
\begin{align*}
g_{12}\sbrack{2\ell+1}(\kappa) &= \sgn(\kappa)\kappa^{\frac12}(i\kappa)^\ell \bigl\langle \delta_x, \bigl[ (\kappa-\partial)^{-1} q(\kappa+\partial)^{-1}\bar q\bigr]^{\ell} (\kappa-\partial)^{-1}q(\kappa+\partial)^{-1}\delta_x\bigr\rangle
\end{align*}
so that
\begin{equation}\label{sbrack notation}
g_{12}\sbrack{\geq 2\ell+1}(\kappa) = \sum_{m=\ell}^\infty g_{12}\sbrack{2m+1}(\kappa)
\qtq{and}
 g_{12}(\kappa) = \sum_{\ell=0}^\infty g_{12}\sbrack{2\ell+1}(\kappa).
\end{equation}
The terms \(g_{21}\sbrack{2\ell+1}\) and \(\gamma\sbrack{2\ell}\) admit correspondingly simple definitions; however, we will also use this notation on more complicated analytic functions of $q$ such as \(\bigl(\frac{g_{12}}{2+\gamma}\bigr)\sbrack{2\ell+1}\), \(\rho\sbrack{2\ell}\), and \(j_{\DNLS}\sbrack{2\ell}\).

Using \eqref{useful IDs}, we can derive the following explicit expressions:
\begin{align}
g_{12}^{[1]}= \tfrac{\sqrt\kappa q}{2\kappa-\partial} ,\qquad  g_{21}^{[1]}= \tfrac{-i\sqrt\kappa \bar q}{2\kappa+\partial}, \qquad \gamma^{[2]} = 2i\kappa \tfrac{q}{2\kappa-\partial} \cdot \tfrac{\bar q}{2\kappa+\partial}, \label{linear}\\
g_{12}^{[3]}= \tfrac{2i\kappa^{\frac32}}{2\kappa-\partial}\bigl [ q \cdot \tfrac{q}{2\kappa-\partial} \cdot \tfrac{\bar q}{2\kappa+\partial} \bigr] \qtq{and} g_{21}^{[3]}= \tfrac{2\kappa^{\frac32}}{2\kappa+\partial}\bigl [ \bar q \cdot \tfrac{q}{2\kappa-\partial} \cdot \tfrac{\bar q}{2\kappa+\partial} \bigr], \label{cubic}
\end{align}
as well as
\begin{align}
\tfrac{1}{\sqrt\kappa}\bigl(\tfrac{g_{12}}{2+\gamma}\bigr)^{[\geq 3]} &= \tfrac{2i}{2\kappa - \p}\Bigl[\bar q \bigl(\tfrac{g_{12}}{2 + \gamma}\bigr)^2\Bigr], \label{g12 frac geq3}\\
\tfrac{1}{\sqrt\kappa}\bigl(\tfrac{g_{21}}{2+\gamma}\bigr)^{[\geq 3]} &= \tfrac{-2}{2\kappa + \p}\Bigl[q \bigl(\tfrac{g_{21}}{2 + \gamma}\bigr)^2\Bigr].\label{g21 frac geq3}
\end{align}

To represent our expansions for higher order terms in a concise form, we introduce a space of paraproducts. We begin by introducing a generating set of operators:
\begin{align}\label{G}
\mathcal G:=\Bigl\{\mathrm{Id}\Bigr\} \cup \Bigl\{\tfrac{2\vk}{2\vk\pm \p}:\, |\vk|\geq 1\Bigr\}.
\end{align}
One may regard the elements of $\mathcal G$ as letters in an alphabet.  We then define the set $\mathcal G^\star$ to be the set of (finite) words built from the alphabet $\mathcal G$.
Specifically, elements in $\mathcal G^\star$ are finite products of the form
\begin{equation}\label{spelling}
L_1L_2\dots L_n\quad\text{where \(L_i\in \mathcal G\)}.
\end{equation}
Note that each factor $L_i$ may have a different parameter $\vk_i$.

\begin{example}
For all \(|\vk|\geq 1\), the operator
\(
\frac{4\vk^2}{4\vk^2 - \p^2} = \frac{2\vk}{2\vk - \p}\tfrac{2\vk}{2\vk+\p} \in \mathcal G^\star
\)
is a word over the alphabet \(\mathcal G\).
\end{example}

We say that the paraproduct $m$ belongs to the class $S(1)$ if it admits a representation as a finite linear combination of paraproducts $m_i$ satisfying
\[
m_i[f] = T_if \quad\text{where $T_i\in \mathcal G^\star$}. 
\]

We will frequently consider families of paraproducts depending continuously on one or more parameters. For example, all the paraproducts in Proposition~\ref{p:Asymptotics} depend continuously on the parameter \(\vk\). When this is the case, we will say such a family is in \(S(1)\) only if the coefficients in this linear combination are uniformly bounded in the parameter(s).

\begin{example}\label{EXs1}
For all \(|\vk|\geq 1\), the operators
\[
\tfrac\p{2\vk+\p} = \mathrm{Id} - \tfrac{2\vk}{2\vk+\p}\qtq{and}\tfrac{2\vk - \p}{2\vk + \p} = 2\tfrac{2\vk}{2\vk + \p} - \mathrm{Id}
\]
are elements of \(S(1)\).
\end{example}

For \(n\geq 2\), we inductively define $S(n)$ as linear combinations of paraproducts that admit the representation
\begin{equation}\label{def: S(n)}
m[f_1,\dots,f_n] = T\Biggl[\,\prod_{j=1}^J m_j\bigl[f_{\sigma(n_{j-1} + 1)},\dots,f_{\sigma(n_j)}\bigr]\Biggr],
\end{equation}
where \(0=n_0<\dots<n_J=n\) are integers, \(m_j\in S(n_j-n_{j-1})\), \(T\in \mathcal G^\star\), and \(\sigma\in \mathfrak S_n\) is a permutation. The product appearing in \eqref{def: S(n)} is a pointwise product of functions (and not a composition of operators). As in the case \(n=1\), we require the coefficients in the linear combination to be uniformly bounded in any parameters.

This definition is clearly symmetric, in the sense that whenever \(m\in S(n)\) and \(\sigma\in \mathfrak S_n\) we have
\begin{equation}\label{parasym}
 m\bigl[f_{\sigma(1)},\dots,f_{\sigma(n)}\bigr]\in S(n).
\end{equation}
Further, by induction on \(n\), our definition is consistent with interior products, in the sense that if \(2\leq k\leq n\), \(m_1\in S(k)\), and \(m_2\in S(n+1-k)\) then
\begin{equation}\label{paraint}
 m_1\bigl[f_1,\dots,f_{k-1},m_2[f_{k},\dots,f_n]\bigr]\in S(n).
\end{equation}

We illustrate our paraproduct classes with an example motivated by \eqref{cubic}:
\begin{example}\label{EX0}
For all \(|\vk|\geq 1\), the paraproduct
\[
m[f_1,f_2,f_3] = \tfrac{8\vk^3}{2\vk - \p}\Bigl[f_1\tfrac{f_2}{2\vk - \p}\tfrac{f_3}{2\vk + \p}\Bigr]
\]
is an element of \(S(3)\) that can be expressed in the form \eqref{def: S(n)} with \(\sigma = \mathrm{Id}\),
\[
T = \tfrac{2\vk}{2\vk-\p},\qquad m_1[f_1] = f_1,\qquad m_2[f_2] = \tfrac{2\vk}{2\vk - \p}f_2,\qquad m_3[f_3] = \tfrac{2\vk}{2\vk + \p}f_3.
\]
\end{example}

While paraproducts are often regarded as multilinear objects, our use of them here is closer to that of a polynomial in a single variable, namely, $q$.  More accurately, our paraproducts will solely be populated by the objects appearing in the list presented in Corollary~\ref{c:g}.  Moreover, in estimating these paraproducts, we will only be employing the information \eqref{admissible} about these objects.  With these considerations in mind, we will frequently employ the expedient of writing paraproducts as \(m[f,\dots,f]\).  Similarly, if an expression involves paraproducts \(m_1,\dots,m_k\in S(n)\), we simply denote each paraproduct by \(m\) as in, e.g., \eqref{g12 3 expansion} below.

We demonstrate this notation with two examples from the proof of Proposition~\ref{p:Asymptotics} below:

\begin{example}\label{EX1} If \(|\vk|\geq 1\) the paraproduct
\[
m[f_1,f_2,f_3] = 4\vk^2\tfrac{f_1}{2\vk + \p}\tfrac{f_2}{2\vk+\p}f_3= \tfrac{2\vk}{2\vk + \p}f_1\cdot \tfrac{2\vk}{2\vk+\p}f_2\cdot f_3
\]
is an element of \(S(3)\). We may then write
\begin{align}
\tfrac{8i\vk^3}{2\vk - \p}\Bigl[\tfrac q{4\vk^2 - \p^2}\bigl|\tfrac{q'}{4\vk^2 - \p^2}\bigr|^2\Bigr] &= \tfrac{2i\vk}{2\vk-\p}m\Bigl[\tfrac q{2\vk  -\p},\tfrac{q'}{2\vk - \p},\tfrac{\bar q'}{4\vk^2 - \p^2}\Bigr]\label{9:50}\\
&= \tfrac{2i\vk}{2\vk-\p}m\Bigl[\tfrac{f}{2\vk  -\p},\tfrac{f'}{2\vk - \p},\tfrac{f'}{4\vk^2 - \p^2}\Bigr],\notag
\end{align}
where each \(f\) represents either \(q\) or \(\bar q\), which are both elements of the list \eqref{f list}. 
\end{example}

One should observe that using the definition of $S(3)$, the expression \eqref{9:50} could be further simplified to read
$$
m_1\Bigl[\tfrac{f}{2\vk  -\p},f,\tfrac{f'}{4\vk^2 - \p^2}\Bigr] \qtq{or even} (2\vk)^{-2} m_2\bigl[f,f,f\bigr]
$$ 
with $m_1,m_2\in S(3)$. However, we will need the additional structure appearing in \eqref{9:50} when we apply Proposition~\ref{p:Asymptotics}.

\begin{example}\label{EX2} If \(|\vk|\geq 1\) then, recalling Example~\ref{EXs1}, the paraproduct
\[
m[f_1,f_2,f_3] = 2i\bigl(\tfrac{2\vk-\p}{2\vk+\p}f_1\bigr)f_2f_3
\]
is again an element of \(S(3)\). From \eqref{linear}, we then have
\begin{align}\label{9:51}
\tfrac1{2\vk - \p}\bigl[\gamma\sbrack 2(\vk) \tfrac{q''}{4\vk^2 - \p^2}\bigr] &= \tfrac{2i\vk}{2\vk - \p}\bigl[\tfrac{\bar q}{2\vk+\p} \tfrac q{2\vk - \p} \tfrac{q''}{4\vk^2 - \p^2}\bigr]\\
& = \tfrac{\vk}{2\vk-\p}m\Bigl[\tfrac{\bar q}{2\vk  -\p},\tfrac{q}{2\vk - \p},\tfrac{q''}{4\vk^2 - \p^2}\Bigr] \notag\\
&=  \tfrac{\vk}{2\vk-\p}m\Bigl[\tfrac f{2\vk  -\p},\tfrac f{2\vk - \p},\tfrac{f''}{4\vk^2 - \p^2}\Bigr],\notag
\end{align}
where each \(f\) represents an element of the list \eqref{f list}.
\end{example}

Just as in the case of \eqref{9:50}, one could further simplify \eqref{9:51} to read
$$
m_1\Bigl[\tfrac{f}{2\vk  -\p},\tfrac{f}{2\vk  -\p},f\Bigr] \qtq{or} (2\vk)^{-2} m_2\bigl[f,f,f\bigr]
$$ 
with $m_1,m_2\in S(3)$.  The additional structure in \eqref{9:51} will be exploited later.

As a first application of our paraproducts we have the following:
\begin{lemma}\label{L3:11}
For \(\ell\geq 1\) and $|\vk|\geq 1$ we have the representations
\begin{align}
\tfrac1{\sqrt\vk}g_{12}\sbrack{2\ell+1}(\vk) &= \tfrac\vk{2\vk - \p}m\Bigl[\underbrace{f,\dots,f}_\ell,\underbrace{\tfrac f{2\vk - \p},\dots,\tfrac f{2\vk - \p}}_{\ell+1}\Bigr],\label{paraproduct for g}\\
\tfrac1{\sqrt\vk}g_{12}\sbrack{\geq 2\ell+1}(\vk) &= \tfrac\vk{2\vk - \p}m\Bigl[\underbrace{f,\dots,f}_\ell,\underbrace{\tfrac f{2\vk - \p},\dots,\tfrac f{2\vk - \p}}_{\ell+1}\Bigr],\label{paraproduct for g geq}\\
\tfrac1{\sqrt\vk}\bigl(\tfrac{g_{21}}{2+\gamma}\bigr)\sbrack{2\ell+1}(\vk) &= \tfrac\vk{2\vk + \p}m\Bigl[\underbrace{f,\dots,f}_\ell,\underbrace{\tfrac f{2\vk - \p},\dots,\tfrac f{2\vk - \p}}_{\ell+1}\Bigr],\label{paraproduct for frac}\\
\tfrac1{\sqrt\vk}\bigl(\tfrac{g_{21}}{2+\gamma}\bigr)\sbrack{\geq 2\ell+1}(\vk) &= \tfrac\vk{2\vk + \p}m\Bigl[\underbrace{f,\dots,f}_\ell,\underbrace{\tfrac f{2\vk - \p},\dots,\tfrac f{2\vk - \p}}_{\ell+1}\Bigr],\label{paraproduct for frac geq}
\end{align}
where \(m\in S(2\ell+1)\) and each $f$ represents an element from the list \eqref{f list}. Similar representations hold for $g_{21}$ and $\frac{g_{12}}{2+\gamma}$.
\end{lemma}

\begin{proof}
We prove all four identities simultaneously by strong induction on \(\ell\).

When \(\ell = 1\), the identity \eqref{paraproduct for g} follows from \eqref{cubic}. Using \eqref{useful IDs} and \eqref{linear}, we may write
$$
\tfrac1{\sqrt\vk}\bigl(\tfrac{g_{21}}{2+\gamma}\bigr)\sbrack{3}(\vk) = \tfrac{\vk}{2(2\vk+\p)} \Bigl[q \cdot \tfrac{ \bar q}{2\vk+\p}\cdot \tfrac{ \bar q}{2\vk+\p}\Bigr],
$$
which gives \eqref{paraproduct for frac} when $\ell=1$. Using \eqref{useful IDs} we may also write
\begin{align*}
\tfrac1{\sqrt\vk}g_{12}\sbrack{\geq 3}(\vk) = -\tfrac{4\vk}{2\vk-\p}\Bigl[ q\cdot \tfrac1{2\vk-\p}\bigl(\tfrac{2\vk-\p}{\sqrt\vk} g_{12}\bigr)\cdot \tfrac1{2\vk+\p}\bigl(\tfrac{2\vk+\p}{\sqrt\vk}\tfrac{g_{21}}{2+\gamma}\bigr)\Bigr],\\
\tfrac1{\sqrt\vk}\bigl(\tfrac{g_{21}}{2+\gamma}\bigr)\sbrack{\geq 3}(\vk)= -\tfrac{2\vk}{2\vk+\p}\Bigl[ q\cdot \tfrac1{2\vk+\p}\bigl(\tfrac{2\vk+\p}{\sqrt\vk} \tfrac{g_{21}}{2+\gamma}\bigr)\cdot \tfrac1{2\vk+\p}\bigl(\tfrac{2\vk+\p}{\sqrt\vk}\tfrac{g_{21}}{2+\gamma}\bigr)\Bigr];
\end{align*}
the identities \eqref{paraproduct for g geq} and \eqref{paraproduct for frac geq} for $\ell=1$ then follow from Corollary~\ref{c:g}.

Now assume that \eqref{paraproduct for g}--\eqref{paraproduct for frac geq} are true for all \(0\leq m\leq \ell-1\). Using \eqref{useful IDs} we see that
\begin{align*}
\tfrac1{\sqrt\vk}g_{12}\sbrack{2\ell+1}(\vk) &= -\sum_{n=0}^{\ell-1}\tfrac{4\vk}{2\vk - \p}\Bigl[q\,\tfrac1{\sqrt\vk}g_{12}\sbrack{2n+1}\,\tfrac1{\sqrt\vk}\bigl(\tfrac{g_{21}}{2+\gamma}\bigr)\sbrack{2(\ell-1 - n)+1}\Bigr],\\
\tfrac1{\sqrt\vk}g_{12}\sbrack{\geq 2\ell+1}(\vk) &= -\sum_{n=0}^{\ell-2}\tfrac{4\vk}{2\vk - \p}\Bigl[q\,\tfrac1{\sqrt\vk}g_{12}\sbrack{2n+1}\,\tfrac1{\sqrt\vk}\bigl(\tfrac{g_{21}}{2+\gamma}\bigr)\sbrack{\geq 2(\ell-1 - n)+1}\Bigr]\\
&\quad - \tfrac{4\vk}{2\vk - \p}\Bigl[q\tfrac1{\sqrt\vk}g_{12}\sbrack{\geq 2\ell-1}\tfrac1{\sqrt\vk}\bigl(\tfrac{g_{21}}{2+\gamma}\bigr)\Bigr],\\
\tfrac1{\sqrt\vk}\bigl(\tfrac{g_{21}}{2+\gamma}\bigr)\sbrack{2\ell+1}(\vk) &= -\sum_{n=0}^{\ell-1}\tfrac{2\vk}{2\vk + \p}\Bigl[q\,\tfrac1{\sqrt\vk}\bigl(\tfrac{g_{21}}{2+\gamma}\bigr)\sbrack{2n+1}\,\tfrac1{\sqrt\vk}\bigl(\tfrac{g_{21}}{2+\gamma}\bigr)\sbrack{2(\ell-1 - n)+1}\Bigr],\\
\tfrac1{\sqrt\vk}\bigl(\tfrac{g_{21}}{2+\gamma}\bigr)\sbrack{\geq 2\ell+1}(\vk) &= -\sum_{n=0}^{\ell-2}\tfrac{2\vk}{2\vk + \p}\Bigl[q\,\tfrac1{\sqrt\vk}\bigl(\tfrac{g_{12}}{2+\gamma}\bigr)\sbrack{2n+1}\,\tfrac1{\sqrt\vk}\bigl(\tfrac{g_{21}}{2+\gamma}\bigr)\sbrack{\geq 2(\ell-1 - n)+1}\Bigr]\\
&\quad - \tfrac{2\vk}{2\vk - \p}\Bigl[q\tfrac1{\sqrt\vk}\bigl(\tfrac{g_{21}}{2+\gamma}\bigr)\sbrack{\geq 2\ell-1}\tfrac1{\sqrt\vk}\bigl(\tfrac{g_{21}}{2+\gamma}\bigr)\Bigr],
\end{align*}
and hence \eqref{paraproduct for g} through \eqref{paraproduct for frac geq} follow from the inductive hypothesis, \eqref{parasym}, and \eqref{paraint}.
\end{proof}

As a second application, we have the following:

\begin{proposition}\label{p:Asymptotics}
Denote \(u = \frac q{4\vk^2 - \p^2}\). Then we have the representations
\begin{align}\label{g12 3 expansion}
\tfrac1{\sqrt\vk}g_{12}\sbrack 3(\vk) &= 16i\vk^4 |u|^2u + 24i\vk^3|u|^2u' +\tfrac\vk{2\vk - \p}\,m\Bigl [\tfrac{f}{2\vk  -\p},\tfrac{f}{2\vk - \p},\tfrac{f''}{4\vk^2 - \p^2}\Bigr] \notag\\
&\quad+ \tfrac\vk{2\vk - \p}\,m\Bigl [\tfrac{f}{2\vk  -\p},\tfrac{f'}{2\vk - \p},\tfrac{f'}{4\vk^2 - \p^2}\Bigr]\notag\\
&= 16i\vk^4 |u|^2u + 24i\vk^3|u|^2u' + 8i\vk^2|u|^2u''\\
&\quad + 8i\vk^2|u'|^2u + 12i\vk^2(u')^2\bar u + \tfrac1{2\vk - \p}\,m\Bigl [\tfrac{f}{2\vk  -\p},\tfrac{f'}{2\vk - \p},\tfrac{f''}{4\vk^2 - \p^2}\Bigr]\notag\\
&\quad + \tfrac{\p}{2\vk - \p}\,m\Bigl [\tfrac f{2\vk - \p},\tfrac{f}{2\vk - \p},\tfrac{f''}{4\vk^2 - \p^2}\Bigr] + \tfrac{\p}{2\vk - \p}\,m\Bigl [\tfrac f{2\vk - \p},\tfrac{f'}{2\vk - \p},\tfrac{f'}{4\vk^2 - \p^2}\Bigr],\notag
\end{align}
\begin{align}\label{gamma 4 expansion}
&\gamma\sbrack 4(\vk)= - 96\vk^6|u|^4 - 96\vk^5 |u|^2\bigl(u'\bar u - u\bar u'\bigr)\\
&\qquad \qquad + m\Bigl[\tfrac f{2\vk - \p},\tfrac f{2\vk - \p},\tfrac{f'}{2\vk - \p},\tfrac{f'}{2\vk - \p}\Bigr] + m\Bigl[f,\tfrac f{2\vk - \p},\tfrac{f}{2\vk - \p},\tfrac{f''}{4\vk^2 - \p^2}\Bigr],\notag
\end{align}
\begin{align}\label{g12 5 expansion}
\tfrac1{\sqrt\vk}g_{12}\sbrack 5(\vk) &= - 192\vk^7 |u|^4u - 480 \vk^6 |u|^4u' + \tfrac1{2\vk - \p}m\Bigl[f,\tfrac f{2\vk - \p},\tfrac f{2\vk - \p},\tfrac{f'}{2\vk - \p},\tfrac{f'}{2\vk - \p}\Bigr]\notag\\
&\quad + \tfrac1{2\vk - \p}m\Bigl[f,f,\tfrac f{2\vk - \p},\tfrac{f}{2\vk - \p},\tfrac{f''}{4\vk^2 - \p^2}\Bigr],
\end{align}
\begin{align}
&\gamma\sbrack 6(\vk) = -1280i \vk^9|u|^6 + m\Bigl[f,f,\tfrac f{2\vk - \p},\tfrac f{2\vk - \p},\tfrac f{2\vk - \p},\tfrac{f'}{2\vk - \p}\Bigr],\label{gamma 6 expansion}\\
&\tfrac1{\sqrt\vk}g_{12}\sbrack 7(\vk) = - 2560i\vk^{10} |u|^6 u + \tfrac1{2\vk - \p}m\Bigl[f,f,f,\tfrac f{2\vk - \p},\tfrac f{2\vk - \p},\tfrac f{2\vk - \p},\tfrac{f'}{2\vk - \p}\Bigr].\label{g12 7 expansion}
\end{align}

Further, we have the representations
\begin{align}\label{g21 frac 3 expansion}
\tfrac1{\sqrt\vk}\bigl(\tfrac{g_{21}}{2 + \gamma}\bigr)\sbrack 3(\vk)
&=4\vk^4|u|^2\bar u+ \tfrac{\vk}{2\vk + \p}\,m\Bigl [\tfrac{f}{2\vk  -\p},\tfrac{f}{2\vk  -\p},\tfrac{f'}{2\vk - \p}\Bigr]\notag\\
&=4\vk^4|u|^2\bar u - 8\vk^3|u|^2\bar u' - 2\vk^3\bar u^2 u' + \tfrac1{2\vk + \p}\,m\Bigl [f,\tfrac{f}{2\vk  -\p},\tfrac{f''}{4\vk^2 - \p^2}\Bigr]\\
&\quad + \tfrac1{2\vk + \p}\,m\Bigl [\tfrac f{2\vk -\p},\tfrac{f'}{2\vk  -\p},\tfrac{f'}{2\vk - \p}\Bigr]\notag\\
&= 4\vk^4|u|^2\bar u - 8\vk^3|u|^2\bar u' - 2\vk^3\bar u^2 u' +4\vk^2|u|^2\bar u''\notag\\
&\quad + 6\vk^2|u'|^2\bar u + 5\vk^2(\bar u')^2u + \tfrac1{2\vk + \p}\,m\Bigl [\tfrac{f}{2\vk  -\p},\tfrac{f'}{2\vk - \p},\tfrac{f''}{4\vk^2 - \p^2}\Bigr]\notag\\
&\quad + \tfrac{\p}{2\vk + \p}\,m\Bigl [\tfrac f{2\vk - \p},\tfrac{f}{2\vk - \p},\tfrac{f''}{4\vk^2 - \p^2}\Bigr] + \tfrac{\p}{2\vk + \p}\,m\Bigl [\tfrac f{2\vk - \p},\tfrac{f'}{2\vk - \p},\tfrac{f'}{4\vk^2 - \p^2}\Bigr],\notag
\\
\label{g21 frac 5 expansion}
\tfrac1{\sqrt\vk}\bigl(\tfrac{g_{21}}{2 + \gamma}\bigr)\sbrack 5(\vk) 
&= 32i\vk^7|u|^4\bar u+ \tfrac\vk{2\vk + \p}m\Bigl[f,\tfrac f{2\vk - \p},\tfrac f{2\vk - \p},\tfrac{f}{2\vk - \p},\tfrac{f'}{2\vk - \p}\Bigr]\notag\\
&= 32i\vk^7|u|^4\bar u -128i\vk^6|u|^4\bar u' - 48i\vk^6|u|^2\bar u^2u'\\
&\quad+\tfrac1{2\vk + \p}m\Bigl[f,f,\tfrac f{2\vk - \p},\tfrac{f}{2\vk - \p},\tfrac{f''}{4\vk^2 - \p^2}\Bigr]\notag\\
&\quad + \tfrac1{2\vk + \p}m\Bigl[f,\tfrac f{2\vk - \p},\tfrac f{2\vk - \p},\tfrac{f'}{2\vk - \p},\tfrac{f'}{2\vk - \p}\Bigr],\notag
\\
\label{g21 frac 7 expansion}
\tfrac1{\sqrt\vk}\bigl(\tfrac{g_{21}}{2+\gamma}\bigr)\sbrack 7(\vk)&= -320 \vk^{10}|u|^6 \bar u + \tfrac1{2\vk + \p}m\Bigl[f,f,f,\tfrac f{2\vk - \p},\tfrac f{2\vk - \p},\tfrac f{2\vk - \p},\tfrac{f'}{2\vk - \p}\Bigr].
\end{align}

Throughout, the paraproduct \(m\) lies in \(S(n)\) for an appropriate integer \(n\), each $f$ represents an element from the list \eqref{f list}, and identical expressions hold with \(\vk\) replaced by \(-\vk\).
\end{proposition}

\begin{proof}
Using \eqref{linear}, we derive the identity
\begin{equation}\label{gamma 2 expansion}
\gamma\sbrack 2 = 8i\vk^3 |u|^2 + 4i\vk^2(u'\bar u - u\bar u') - 2i\vk|u'|^2.
\end{equation}

To obtain the first expansion for $\tfrac1{\sqrt\vk}g_{12}\sbrack 3(\vk)$ in \eqref{g12 3 expansion}, we use \eqref{useful IDs} to write
\begin{align*}
\tfrac1{\sqrt\vk}g_{12}\sbrack 3&= \tfrac{4\vk^2}{2\vk - \p}\bigl[u\gamma\sbrack 2\bigr] - \tfrac 1{2\vk - \p}\bigl[u''\gamma\sbrack 2\bigr]\\
&= \tfrac{32i\vk^5}{2\vk - \p}\bigl[|u|^2u\bigr] + \tfrac{16i\vk^4}{2\vk - \p}\bigl[(u'\bar u - u\bar u')u\bigr] -\tfrac{8i\vk^3}{2\vk - \p}\bigl[|u'|^2u\bigr] - \tfrac1{2\vk - \p}\bigl[\gamma\sbrack 2 u''\bigr].
\end{align*}
Recalling Examples~\ref{EX1} and~\ref{EX2}, the last two summands are of the form 
\begin{align}\label{10:13}
\tfrac{2\vk i}{2\vk-\p}m\Bigl[\tfrac{f}{2\vk  -\p},\tfrac{f'}{2\vk - \p},\tfrac{f'}{4\vk^2 - \p^2}\Bigr]+\tfrac{\vk}{2\vk-\p}m\Bigl[\tfrac{f}{2\vk  -\p},\tfrac{f}{2\vk - \p},\tfrac{f''}{4\vk^2 - \p^2}\Bigr] 
\end{align}
and so acceptable.

For the remaining summands we first use that
\begin{equation}\label{1/2vk-p exp}
\tfrac1{2\vk - \p} = \tfrac1{2\vk} + \tfrac1{2\vk}\tfrac\p{2\vk - \p}
\end{equation}
to write
\[
\tfrac{16i\vk^4}{2\vk - \p}\bigl[(u'\bar u - u\bar u')u\bigr] = 8i\vk^3(u'\bar u - u\bar u')u + \tfrac{8i\vk^3\p}{2\vk - \p}\bigl[(u'\bar u - u\bar u')u\bigr],
\]
where we note that the second summand  is of the form \eqref{10:13} and so acceptable. Expanding further, we have
\begin{equation}\label{1/2vk-p exp 2}
\tfrac1{2\vk - \p} = \tfrac1{2\vk} + \tfrac\p{4\vk^2} + \tfrac1{4\vk^2}\tfrac{\p^2}{2\vk - \p},
\end{equation}
which we apply to get
\[
\tfrac{32i\vk^5}{2\vk - \p}\bigl[|u|^2u\bigr] = 16i\vk^4 |u|^2u +16i\vk^3|u|^2u' + 8i\vk^3u^2\bar u' + \tfrac{8i\vk^3\p^2}{2\vk - \p}\bigl[|u|^2u\bigr],
\]
where the final summand is again of the form \eqref{10:13}. This completes the proof of the first expansion recorded in \eqref{g12 3 expansion}.

We now turn to the second expression for $\tfrac1{\sqrt\vk}g_{12}\sbrack 3(\vk)$ in \eqref{g12 3 expansion}.  We again use \eqref{useful IDs} to write
\begin{align*}
\tfrac1{\sqrt\vk}g_{12}\sbrack 3
&= \tfrac{32i\vk^5}{2\vk - \p}\bigl[|u|^2u\bigr] + \tfrac{16i\vk^4}{2\vk - \p}\bigl[(u'\bar u - u\bar u')u\bigr] -\tfrac{8i\vk^3}{2\vk - \p}\bigl[|u'|^2u+|u|^2u''\bigr]\\
&\quad - \tfrac1{2\vk - \p}\bigl[(\gamma\sbrack 2 - 8i\vk^3|u|^2)u''\bigr].
\end{align*}
By \eqref{gamma 2 expansion}, the final summand is seen to be of the form $\frac1{2\vk-\p}m[\tfrac{f}{2\vk  -\p},\tfrac{f'}{2\vk - \p},\tfrac{f''}{4\vk^2 - \p^2}]$, and so acceptable. We then use \eqref{1/2vk-p exp} to write
\begin{align*}
-\tfrac{8i\vk^3}{2\vk - \p}\bigl[|u'|^2u+|u|^2u''\bigr] = -4i\vk^2\bigl[|u'|^2u + |u|^2u''\bigr] -\tfrac{4i\vk^2\p}{2\vk - \p}\bigl[|u'|^2u + |u|^2u''\bigr]
\end{align*}
and note that the final summand is of the form
\begin{align}\label{10:14}
\tfrac{\p}{2\vk-\p}m\Bigl[\tfrac{f}{2\vk  -\p},\tfrac{f}{2\vk - \p},\tfrac{f''}{4\vk^2 - \p^2}\Bigr] + \tfrac{\p}{2\vk-\p}m\Bigl[\tfrac{f}{2\vk  -\p},\tfrac{f'}{2\vk - \p},\tfrac{f'}{4\vk^2 - \p^2}\Bigr],
\end{align}
and so acceptable.  Applying \eqref{1/2vk-p exp 2}, we similarly have
\begin{align*}
\tfrac{16i\vk^4}{2\vk - \p}\bigl[(u'\bar u - u\bar u')u\bigr] &= 8i\vk^3\bigl[(u'\bar u - u\bar u')u\bigr] + 4i\vk^2\bigl[(u'\bar u - u\bar u')u\bigr]'  + \tfrac{4i\vk^2\p^2}{2\vk - \p}\bigl[(u'\bar u - u\bar u')u\bigr],
\end{align*}
where the final summand is again of the form \eqref{10:14}. Expanding even further, we have
\[
\tfrac1{2\vk - \p} = \tfrac1{2\vk} + \tfrac\p{4\vk^2} + \tfrac{\p^2}{8\vk^3} + \tfrac1{8\vk^3}\tfrac{\p^3}{2\vk - \p},
\]
which we apply to the remaining term to get
\begin{align*}
\tfrac{32i\vk^5}{2\vk - \p}\bigl[|u|^2u\bigr] &= 16i\vk^4|u|^2u + 8i\vk^3\bigl[|u|^2u\bigr]' + 4i\vk^2\bigl[|u|^2u\bigr]'' + \tfrac{4i\vk^2\p^3}{2\vk - \p}\bigl[|u|^2u\bigr].
\end{align*}
The final term is once again of the form \eqref{10:14} and so acceptable. Combining these expressions, one obtains the second expansion recorded in \eqref{g12 3 expansion}.

Turning next to \(\gamma\sbrack 4\), we use \eqref{quadratic id} to write
\begin{align*}
\gamma\sbrack 4 = -\tfrac12\bigl[\gamma\sbrack 2\big]^2 - 2g_{12}\sbrack 3g_{21}\sbrack 1 - 2g_{12}\sbrack 1g_{21}\sbrack 3.
\end{align*}
For the first summand, we use \eqref{gamma 2 expansion} to express
\[
-\tfrac12\bigl[\gamma\sbrack 2\big]^2 = 32\vk^6 |u|^4 +32 \vk^5 |u|^2(u'\bar u - u\bar u') + m\Bigl[\tfrac f{2\vk - \p},\tfrac f{2\vk - \p},\tfrac{f'}{2\vk - \p},\tfrac{f'}{2\vk - \p}\Bigr].
\]
For the second summand, we use \eqref{g12 3 expansion} and \eqref{linear} to write
\begin{align*}
-2 g_{12}\sbrack 3\,g_{21}\sbrack 1 &= -32\vk^5|u|^2u\,\tfrac{\bar q}{2\vk +\p} -48\vk^4|u|^2u'\,\tfrac{\bar q}{2\vk + \p}\\
&\quad + m\Bigl[\tfrac f{2\vk - \p},\tfrac f{2\vk - \p},\tfrac{f'}{2\vk - \p},\tfrac{f'}{2\vk - \p}\Bigr] + m\Bigl[f,\tfrac f{2\vk - \p},\tfrac{f}{2\vk - \p},\tfrac{f''}{4\vk^2 - \p^2}\Bigr].
\end{align*}
Writing
\begin{equation}\label{bar q to bar u}
\tfrac{\bar q}{2\vk + \p} = 2\vk\bar u - \bar u',
\end{equation}
in the first two terms, we obtain
\begin{align*}
-2g_{12}\sbrack 3\,g_{21}\sbrack 1 &= -64\vk^6|u|^4+32\vk^5|u|^2u\bar u'-96\vk^5|u|^2u'\bar u\\
&\quad + m\Bigl[\tfrac f{2\vk - \p},\tfrac f{2\vk - \p},\tfrac{f'}{2\vk - \p},\tfrac{f'}{2\vk - \p}\Bigr] + m\Bigl[f,\tfrac f{2\vk - \p},\tfrac{f}{2\vk - \p},\tfrac{f''}{4\vk^2 - \p^2}\Bigr].
\end{align*}
Thus, using \eqref{sym} we also have
\begin{align*}
-2 g_{12}\sbrack 1\,g_{21}\sbrack 3 &= -64\vk^6|u|^4-32\vk^5|u|^2\bar u u' + 96\vk^5|u|^2\bar u'u\\
&\quad + m\Bigl[\tfrac f{2\vk - \p},\tfrac f{2\vk - \p},\tfrac{f'}{2\vk - \p},\tfrac{f'}{2\vk - \p}\Bigr] + m\Bigl[f,\tfrac f{2\vk - \p},\tfrac{f}{2\vk - \p},\tfrac{f''}{4\vk^2 - \p^2}\Bigr].
\end{align*}
Combining these expressions gives us \eqref{gamma 4 expansion}.

Next, consider \eqref{g12 5 expansion}. Using \eqref{useful IDs} and then applying \eqref{gamma 4 expansion}, we find
\begin{align*}
\tfrac1{\sqrt\vk}g_{12}\sbrack 5 &= \tfrac1{2\vk - \p}\Bigl[\gamma\sbrack 4\,q\Bigr]\\
&= \tfrac1{2\vk - \p}\Bigl[-96\vk^6|u|^4q-96\vk^5|u|^2(u'\bar u - u\bar u')q\Bigr]\\
&\quad  + \tfrac1{2\vk - \p}m\Bigl[f,\tfrac f{2\vk - \p},\tfrac f{2\vk - \p},\tfrac{f'}{2\vk - \p},\tfrac{f'}{2\vk - \p}\Bigr]\\
&\quad + \tfrac1{2\vk - \p}m\Bigl[f,f,\tfrac f{2\vk - \p},\tfrac{f}{2\vk - \p},\tfrac{f''}{4\vk^2 - \p^2}\Bigr].
\end{align*}
For the first term, we employ \eqref{1/2vk-p exp 2} to write
\begin{align*}
- \tfrac{96\vk^6}{2\vk - \p}\Bigl[|u|^4\,q\Bigr] &= -192\vk^7|u|^4u - 96\vk^6(|u|^4u)' - \tfrac{96\vk^6\p^2}{2\vk - \p}\Bigl[|u|^4u\Bigr] + \tfrac{96\vk^6}{2\vk - \p}\Bigl[ |u|^4\,u''\Bigr],
\end{align*}
where the third and fourth terms are seen to be acceptable. Similarly, for the remaining term we apply \eqref{1/2vk-p exp} to get
\begin{align*}
-\tfrac{96\vk^5}{2\vk -  \p}\Bigl[|u|^2(u'\bar u - u\bar u')\,q\Bigr] &= -192\vk^6|u|^2(u'\bar u - u\bar u')u - \tfrac{192\vk^6\p}{2\vk - \p}\Bigl[|u|^2(u'\bar u - u\bar u')u\Bigr]\\
&\quad + \tfrac{96\vk^5}{2\vk - \p}\Bigl[|u|^2(u'\bar u - u\bar u')u''\Bigr],
\end{align*}
where the second and third summands are similarly acceptable. This completes the proof of \eqref{g12 5 expansion}.

For \(\gamma\sbrack 6\), we again use the quadratic identity \eqref{quadratic id} to write
\begin{align*}
\gamma\sbrack 6 = - \gamma\sbrack 2\gamma\sbrack 4 - 2g_{12}\sbrack 5g_{21}\sbrack 1 - 2g_{12}\sbrack 3g_{21}\sbrack 3 - 2g_{12}\sbrack 1g_{21}\sbrack 5.
\end{align*}
For the first term, we use \eqref{gamma 4 expansion} and \eqref{linear} followed by \eqref{gamma 2 expansion} to obtain
\begin{align*}
- \gamma\sbrack 2\gamma\sbrack 4 &= 96\vk^6|u|^4\,\gamma\sbrack 2 + m\Bigl[f,f,\tfrac f{2\vk - \p},\tfrac f{2\vk - \p},\tfrac f{2\vk - \p},\tfrac{f'}{2\vk - \p}\Bigr]\\
& = 768i\vk^9 |u|^6 +  m\Bigl[f,f,\tfrac f{2\vk - \p},\tfrac f{2\vk - \p},\tfrac f{2\vk - \p},\tfrac{f'}{2\vk - \p}\Bigr].
\end{align*}

Next, applying \eqref{g12 5 expansion} and \eqref{linear} followed by \eqref{bar q to bar u}, we have
\begin{align*}
-2g_{12}\sbrack 5g_{21}\sbrack 1 & = -384i\vk^8|u|^4u\,\tfrac{\bar q}{2\vk + \p} + m\Bigl[f,f,\tfrac f{2\vk - \p},\tfrac f{2\vk - \p},\tfrac f{2\vk - \p},\tfrac{f'}{2\vk - \p}\Bigr]\\
&= -768i\vk^9|u|^6 + m\Bigl[f,f,\tfrac f{2\vk - \p},\tfrac f{2\vk - \p},\tfrac f{2\vk - \p},\tfrac{f'}{2\vk - \p}\Bigr].
\end{align*}
Another application of \eqref{sym} then gives us
\begin{align*}
-2g_{12}\sbrack 1g_{21}\sbrack 5 & = -768i\vk^9 |u|^6 + m\Bigl[f,f,\tfrac f{2\vk - \p},\tfrac f{2\vk - \p},\tfrac f{2\vk - \p},\tfrac{f'}{2\vk - \p}\Bigr].
\end{align*}

For the remaining term, we use the first expansion in \eqref{g12 3 expansion}, \eqref{sym}, and \eqref{cubic} to express
\begin{align*}
- 2g_{12}\sbrack 3g_{21}\sbrack 3 &= -16i\vk^{\frac 92}|u|^2 ug_{21}\sbrack 3-16\vk^{\frac 92}|u|^2 \bar ug_{12}\sbrack 3 + m\Bigl[f,f,\tfrac f{2\vk - \p},\tfrac f{2\vk - \p},\tfrac f{2\vk - \p},\tfrac{f'}{2\vk - \p}\Bigr]\\
&= -512i\vk^9|u|^6 + m\Bigl[f,f,\tfrac f{2\vk - \p},\tfrac f{2\vk - \p},\tfrac f{2\vk - \p},\tfrac{f'}{2\vk - \p}\Bigr].
\end{align*}
Combining all of these expressions gives us \eqref{gamma 6 expansion}.

Using \eqref{useful IDs} and \eqref{gamma 6 expansion}, we may write
\begin{align*}
\tfrac1{\sqrt\vk}g_{12}\sbrack 7 = \tfrac1{2\vk - \p}\Bigl[\gamma\sbrack 6\,q\Bigr]
&= \tfrac1{2\vk-\p} [ -1280 i\vk^9|u|^6q]  + \tfrac1{2\vk - \p}m\Bigl[f,f,f,\tfrac f{2\vk - \p},\tfrac f{2\vk - \p},\tfrac f{2\vk - \p},\tfrac{f'}{2\vk - \p}\Bigr]\\
& = -2560i\vk^{10}|u|^6u+ \tfrac1{2\vk - \p}m\Bigl[f,f,f,\tfrac f{2\vk - \p},\tfrac f{2\vk - \p},\tfrac f{2\vk - \p},\tfrac{f'}{2\vk - \p}\Bigr],
\end{align*}
which settles \eqref{g12 7 expansion}.

Using \eqref{useful IDs}, \eqref{linear}, and \eqref{bar q to bar u} we may write
\begin{align*}
\tfrac1{\sqrt\vk}\bigl(\tfrac{g_{21}}{2+\gamma}\bigr)\sbrack3 &= \tfrac \vk{2(2\vk + \p)}\Bigl[q \bigl(\tfrac {\bar q}{2\vk + \p}\bigr)^2\Bigr]\\
&= \tfrac{2\vk^3}{2\vk + \p}\Bigl[u \bigl(\tfrac {\bar q}{2\vk + \p}\bigr)^2\Bigr] - \tfrac \vk{2(2\vk + \p)}\Bigl[u'' \bigl(\tfrac {\bar q}{2\vk + \p}\bigr)^2\Bigr]\\
&= \tfrac{8\vk^5}{2\vk + \p}\Bigl[|u|^2\bar u\Bigr] - \tfrac{8\vk^4}{2\vk + \p}\Bigl[|u|^2\bar u'\Bigr] + \tfrac{2\vk^3}{2\vk + \p}\Bigl[u(\bar u')^2\Bigr] - \tfrac \vk{2(2\vk + \p)}\Bigl[u'' \bigl(\tfrac {\bar q}{2\vk + \p}\bigr)^2\Bigr]\!,
\end{align*}
from which the first expansion in \eqref{g21 frac 3 expansion} follows easily using that
\[
\tfrac1{2\vk + \p} = \tfrac1{2\vk} - \tfrac\p{2\vk(2\vk + \p)}.
\]

To obtain the second expression in \eqref{g21 frac 3 expansion}, we expand even further
\begin{align*}
\tfrac1{\sqrt\vk}\bigl(\tfrac{g_{21}}{2+\gamma}\bigr)\sbrack3 &= \tfrac{8\vk^5}{2\vk + \p}\Bigl[|u|^2\bar u\Bigr] - \tfrac{8\vk^4}{2\vk + \p}\Bigl[|u|^2\bar u'\Bigr] + \tfrac{2\vk^3}{2\vk + \p}\Bigl[u(\bar u')^2\Bigr] - \tfrac{2\vk^3}{2\vk + \p}\Bigl[\bar u^2 u''\Bigr]\\
&\quad + \tfrac1{2\vk+\p}m\Bigl[\tfrac f{2\vk - \p},\tfrac{f'}{2\vk - \p},\tfrac{f''}{4\vk^2 - \p^2}\Bigr]
\end{align*}
and use that
\[
\tfrac1{2\vk + \p} = \tfrac1{2\vk} - \tfrac\p{4\vk^2} + \tfrac{\p^2}{4\vk^2(2\vk + \p)}
\]
for the first summand and that
\[
\tfrac1{2\vk + \p} = \tfrac1{2\vk} - \tfrac\p{2\vk(2\vk + \p)}
\]
for the second summand. The last expression in \eqref{g21 frac 3 expansion} is derived by expanding \(\frac1{2\vk + \p}\) one further degree for each summand.

For \eqref{g21 frac 5 expansion} we again use \eqref{useful IDs}, \eqref{linear}, and \eqref{bar q to bar u} to obtain
\begin{align*}
\tfrac1{\sqrt\vk}\bigl(\tfrac{g_{21}}{2+\gamma}\bigr)\sbrack5 &= \tfrac{2i\vk}{2\vk + \p}\Bigl[q \tfrac {\bar q}{2\vk + \p}\tfrac1{\sqrt\vk}\bigl(\tfrac{g_{21}}{2+\gamma}\bigr)\sbrack 3\Bigr]\\
&= \tfrac{8i\vk^3}{2\vk + \p}\Bigl[u \tfrac {\bar q}{2\vk + \p}\tfrac1{\sqrt\vk}\bigl(\tfrac{g_{21}}{2+\gamma}\bigr)\sbrack 3\Bigr] - \tfrac{2i\vk}{2\vk + \p}\Bigl[u'' \tfrac {\bar q}{2\vk + \p}\tfrac1{\sqrt\vk}\bigl(\tfrac{g_{21}}{2+\gamma}\bigr)\sbrack 3\Bigr]\\
&= \tfrac{16i\vk^4}{2\vk + \p}\Bigl[|u|^2\tfrac1{\sqrt\vk}\bigl(\tfrac{g_{21}}{2+\gamma}\bigr)\sbrack 3\Bigr]-\tfrac{8i\vk^3}{2\vk + \p}\Bigl[u\bar u'\tfrac1{\sqrt\vk}\bigl(\tfrac{g_{21}}{2+\gamma}\bigr)\sbrack 3\Bigr] \\
&\quad- \tfrac{2i\vk}{2\vk + \p}\Bigl[u'' \tfrac {\bar q}{2\vk + \p}\tfrac1{\sqrt\vk}\bigl(\tfrac{g_{21}}{2+\gamma}\bigr)\sbrack 3\Bigr].
\end{align*}
We then expand \(\frac1{2\vk + \p}\) and use \eqref{g21 frac 3 expansion} and to obtain \eqref{g21 frac 5 expansion}.

Similarly, for \eqref{g21 frac 7 expansion} we use \eqref{g21 frac 3 expansion} and \eqref{g21 frac 5 expansion} to write
\begin{align*}
\tfrac1{\sqrt\vk}\bigl(\tfrac{g_{21}}{2+\gamma}\bigr)\sbrack7 &= \tfrac{2i\vk}{2\vk + \p}\Bigl[q \tfrac{\bar q}{2\vk +\p}\tfrac1{\sqrt\vk}\bigl(\tfrac{g_{21}}{2+\gamma}\bigr)\sbrack 5\Bigr]-\tfrac{2\vk}{2\vk + \p}\biggl[q \Bigl[\tfrac1{\sqrt\vk}\bigl(\tfrac{g_{21}}{2+\gamma}\bigr)\sbrack 3\Bigr]^2\biggr]\\
&= -\tfrac{64\vk^8}{2\vk+\p}\Bigl[ q\tfrac{\bar q}{2\vk +\p} |u|^4\bar u\Bigr]-\tfrac{32\vk^9}{2\vk+\p}\Bigl[ q|u|^4\bar u^2\Bigr] \\
&\quad+\tfrac1{2\vk + \p}m\Bigl[f,f,f,\tfrac f{2\vk - \p},\tfrac f{2\vk - \p},\tfrac f{2\vk - \p},\tfrac{f'}{2\vk - \p}\Bigr],
\end{align*}
to which we apply \eqref{bar q to bar u}.
\end{proof}

\section{Local smoothing for the DNLS}\label{S:Local}

In this section we prove local smoothing for Schwartz solutions of \eqref{DNLS}:

\begin{proposition}[Local smoothing for the DNLS]\label{p:LS}
Let \(Q\subset \Schw\) be an \(L^2\) bounded and equicontinuous set such that
\[
Q_* = \bigl\{e^{tJ\nabla H}q:|t|\leq 1\text{ and }q\in Q\bigr\}
\]
is a $\delta$-good set for a sufficiently small $\delta$. Then the local smoothing estimate
\begin{equation}\label{local smoothing}
\|q\|_{X^{\frac12}}\lesssim \|q(0)\|_{L^2}
\end{equation}
holds uniformly for $q(0)\in Q$.

Further, equicontinuity holds in the local smoothing topology, in the sense that
\begin{equation}\label{refined local smoothing}
\lim_{\kappa\to \infty}\sup_{q(0)\in Q}\|q\|_{X_\kappa^{\frac12}}=0.
\end{equation}
\end{proposition}

We remind the reader that Corollary~\ref{C:descendants} guarantees that for any $L^2$ bounded and equicontinuous set \(Q\), there is a uniform rescaling so that the corresponding \(Q_*\) is \(\delta\)-good. For the remainder of this section, we fix \(Q_*\subset \Schw\) satisfying the hypotheses of Proposition~\ref{p:LS}.  

Our proof of Proposition~\ref{p:LS} rests on the microscopic conservation law presented in Proposition~\ref{P:micro laws}. Taking
\begin{equation}\label{phi h}
\phi_\mu(x) = \int_0^x \psi_\mu(y)^{24}\,dy,
\end{equation}
and integrating by parts, we obtain
\begin{equation}\label{IBP in dt rho}
\Im\int_{-1}^1\int j_{\DNLS}(\vk;q(t))\,\psi_\mu^{24}\,dx\,dt =\Im\int\rho(\vk;q(t))\phi_\mu\,dx\Big|_{t=-1}^{t=1}.
\end{equation}

To bound the right-hand side of this expression, we use the following:

\begin{lemma}[Estimate for \(\rho\)]\label{L:rho}
The following estimates hold uniformly for \(q\in Q_*\), \(\kappa,\vk\geq 1\), and \(\mu\in \R:\)
\begin{gather}\label{rho bound}
\left|\int \Im \rho(\varkappa;q)\,\phi_\mu\,dx\right| \lesssim \vk^{-1}\|q\|_{E_{2\sigma,\vk}^\sigma}^2\Bigl[1 + \|q\|_{L^2}^2\Bigr],\\
\label{rho 1241}
\int_\kappa^\infty\int \left|\Im\int\rho(\vk;q)\phi_\mu\,dx\right|\,e^{-\frac1{200}|h-\mu|}\,d\mu\,d\vk\lesssim\|q\|_{E_{\sigma,\kappa}^\sigma}^2\Bigl[1 + \|q\|_{L^2}^2\Bigr].
\end{gather}
\end{lemma}

\begin{proof}
A computation yields
\begin{align*}
\Im \int \rho\sbrack{2}(\varkappa;q) \, \phi_\mu\, dx = \tfrac1{2i} \int \bar q \, \phi_\mu\,  \tfrac{q'}{4\varkappa^2-\partial^2}\, dx - \tfrac1{2i} \int q \,\phi_\mu \,\tfrac{\bar q'}{4\varkappa^2-\partial^2}\, dx,
\end{align*}
so using \eqref{E loc} we get
\begin{align*}
\Biggl|\Im \int \rho\sbrack{2}(\varkappa;q) \, \phi_\mu\, dx\Biggr| &\leq \vk^{-1}\|\phi_\mu q\|_{E_{1,\vk}^{1/2}}\|q\|_{E_{1,\vk}^{1/2}} \lesssim\vk^{-1}\|q\|_{E_{2\sigma,\vk}^\sigma}^2.
\end{align*}

Turning to the higher order terms, we use \eqref{g12 frac geq3} and \eqref{g21 frac geq3} to write
\[
\rho\sbrack{\geq 4}(\varkappa;q) = q\cdot \tfrac{-2i}{2\vk + \p}\Bigl[ q \bigl(\tfrac{g_{21}}{2 + \gamma}\bigr)^2\Bigr]
	+ \bar q\cdot \tfrac{2i}{2\vk - \p}\Bigl[\bar q \bigl(\tfrac{g_{12}}{2 + \gamma}\bigr)^2\Bigr].
\]
Thus, by \eqref{g12 frac in sup norm} and the fact that \(\|\tfrac 1{2\vk \pm \p}\|_{\op}\lesssim \vk^{-1}\), we obtain
\begin{align*}
\biggl|\Im \int \rho\sbrack{\geq4}(\varkappa;q) \, \phi_\mu\, dx\biggr| &\lesssim \vk^{-1} \| \phi_\mu q\|_{L^2} \| q\|_{L^2}
	\Bigl[ \bigl\|\tfrac{g_{12}}{2 + \gamma} \bigr\|_{L^\infty}^2 + \bigl\|\tfrac{g_{21}}{2 + \gamma} \bigr\|_{L^\infty}^2 \Bigr] \\
&\lesssim \vk^{-1}\|q\|_{L^2}^2 \|q\|_{E_{2\sigma,\vk}^\sigma}^2,
\end{align*}
which completes the proof of \eqref{rho bound}.

The estimate \eqref{rho 1241} follows from \eqref{rho bound} and Lemma \ref{l:E int}.
\end{proof}

Turning to the left-hand side of \eqref{IBP in dt rho}, our main challenge will be to control the remainder terms \(j_{\DNLS}\sbrack{\geq 4}\). To do so, we need to distribute the exponential weight \(\psi_\mu^{24}\) across the arguments of paraproducts in \(S(n)\). To accomplish this, we introduce a modified space of paraproducts, \(S_{\loc}(n)\), which involve a spatial parameter \(\mu\in \R\). While this construction will aid our proof of Proposition~\ref{p:LS}, its true value will only become clear when we turn to the significantly more involved problem of obtaining local smoothing estimates for the difference flow in Section~\ref{sect: diff local smoothing}.

We define an extended set of generators
\begin{align}\label{Gloc}
\mathcal G_{\loc} =\mathcal G \cup  \Bigl\{S^{\pm 1}\psi_\mu^\ell S^{\mp 1} \psi_\mu^{-\ell}, \ \psi_\mu^\ell S^{\pm 1} \psi_\mu^{-\ell}S^{\mp 1} : \, |\ell|\leq 24\text{ is an integer and } S\in \mathcal G^\star\Bigr\},
\end{align}
which are $\mu$-dependent operators.  Recall that the set $\mathcal G$ was defined in \eqref{G} and generates the set $\mathcal G^\star$ of finite words over the alphabet $\mathcal G$, as in \eqref{spelling}. We similarly take \(\mathcal G_\loc^\star\) to be the set of words over the alphabet \(\mathcal G_\loc\). 

\begin{example}\label{EXloc0}
If \(|\vk|\geq 1\) and \(|\ell|\leq 24\) is an integer then
\[
\psi_\mu^\ell \tfrac{2\vk}{2\vk + \p}\psi_\mu^{-\ell} = \bigl(\psi_\mu^\ell \tfrac{2\vk}{2\vk + \p}\psi_\mu^{-\ell}\tfrac{2\vk + \p}{2\vk}\bigr)\tfrac{2\vk}{2\vk + \p}\in \mathcal G_{\loc}^\star.
\]
\end{example}

Paralleling the construction of $S(n)$, we say that a paraproduct $m \in S_{\loc}(1)$ if it admits a representation as a finite linear combination of elements of $\mathcal G_\loc^\star$.

For \(n\geq 2\), we inductively define $S_{\loc}(n)$ as finite linear combinations of paraproducts that admit the representation
\begin{align}\label{mu mess}
m[f_1,\dots,f_n] &= T\Biggl[\,\prod_{j=1}^J m_j[f_{\sigma(n_{j-1} + 1)},\dots,f_{\sigma(n_j)}]\Biggr],
\end{align}
where \(0=n_0<\dots<n_J=n\) are integers, \(m_j\in S_{\loc}(n_j-n_{j-1})\), \(T\in\mathcal G_{\loc}^\star\), and \(\sigma\in \mathfrak S_n\).  On both sides of \eqref{mu mess}, all paraproducts are evaluated at a common value of $\mu$. This will be the standing convention whenever we combine paraproducts in \(S_\loc(n)\).

As in the case of \(S(n)\), we require all coefficients in these linear combinations to be uniformly bounded in any parameters.

\begin{example}\label{EX0loc}
Recall from Example~\ref{EX0} that for \(|\vk|\geq1\), the paraproduct
\[
m[f_1,f_2,f_3] = \tfrac{8\vk^3}{2\vk - \p}\Bigl[f_1\tfrac{f_2}{2\vk - \p}\tfrac{f_3}{2\vk + \p}\Bigr]
\]
is in \(S(3)\). We may write
\[
\psi_\mu^{24}m[f_1,f_2,f_3] = \widetilde m[\psi_\mu^8f_1,\psi_\mu^8f_2,\psi_\mu^8f_3],
\]
where \(\widetilde m\in S_{\loc}(3)\) has representation
\[
\widetilde m[f_1,f_2,f_3] = \psi_\mu^{24}\tfrac{8\vk^3}{2\vk - \p}\Bigl[\psi_\mu^{-8}f_1\cdot \tfrac{\psi_\mu^{-8}f_2}{2\vk - \p} \cdot\tfrac{\psi_\mu^{-8}f_3}{2\vk + \p}\Bigr],
\]
which can be expressed as in \eqref{mu mess} with \(\sigma = \mathrm{Id}\), \(T = \psi_\mu^{24}\tfrac{2\vk}{2\vk-\p}\psi_\mu^{-24}\),
\[
m_1[f_1] = f_1,\qquad m_2[f_2] = \psi_\mu^8\tfrac{2\vk}{2\vk - \p}(\psi_\mu^{-8}f_2),\qquad m_3[f_3] = \psi_\mu^8\tfrac{2\vk}{2\vk + \p}(\psi_\mu^{-8}f_3).
\]
\end{example}

Example~\ref{EX0loc} demonstrates one of the key motivations for the introduction of this class of paraproducts; this is codified in property (i) of Lemma~\ref{l:paraproperties}. It is mandated by the necessity of employing local smoothing estimates on each and every argument of our paraproducts.  

We first record a result which will be used in the proof of Lemma~\ref{l:paraproperties}.

\begin{lemma}\label{L:conjugate}
If \(T\in\mathcal G_{\loc}^\star\), then the conjugated operators
\begin{align}\label{1888}
\tfrac{2+\p}2 T \tfrac2{2+\p},  \quad  \tfrac{2-\p}2 T\tfrac 2{2-\p}, \quad \tfrac 2{2+\p} T\tfrac{2+\p}2\qtq{and} \tfrac 2{2-\p}T\tfrac{2-\p}2
\end{align}
belong to $S_\loc(1)$.
\end{lemma}

\begin{proof}
By definition, any $T\in \mathcal G^\star_{\loc}$ may be written as a finite product of the generators \(\mathcal G_{\loc}\).  As conjugating a product is equivalent to conjugating each factor, it suffices to verify the claim in the case $T\in\mathcal G_{\loc}$.

The elements of $\mathcal G_{\loc}$ come in five kinds.  The easiest case to deal with is $T\in\mathcal G$ because then $T$ commutes with $2\pm\p$ and all four operators in \eqref{1888} are all equal to $T$.  In what follows, we will treat the first two operators in \eqref{1888} for each of the four remaining kinds of generators in $\mathcal G_{\loc}$.  The remaining two operators in \eqref{1888} may be treated in a parallel manner.

To unify our treatment of the first two operators in \eqref{1888}, we will show that 
\begin{equation}\label{claimsdepartment}
\tfrac{2\kappa + \p}{2\kappa} T\tfrac{2\kappa}{2\kappa + \p} \quad\text{is a linear combination of words in} \,\,\mathcal G_\loc^\star
\end{equation}
for any \(|\kappa|\geq 1\), whenever $T=S^{\pm 1}\psi_\mu^\ell S^{\mp 1} \psi_\mu^{-\ell}$ or $T=\psi_\mu^\ell S^{\pm 1} \psi_\mu^{-\ell}S^{\mp 1}$ with $S\in \mathcal G^\star\backslash\{\mathrm{Id}\}$ and $|\ell|\leq 24$.

If \(T =S^{-1}\psi_\mu^\ell S \psi_\mu^{-\ell}\), we write
\[
\tfrac{2\kappa + \p}{2\kappa} T \tfrac{2\kappa}{2\kappa + \p}  = \Bigl[\bigl( S \tfrac{2\kappa}{2\kappa + \p}\bigr)^{-1}\psi_\mu^\ell \bigl( S \tfrac{2\kappa}{2\kappa + \p}\bigr) \psi_\mu^{-\ell}\Bigr]\Bigl[ \psi_\mu^{\ell} \tfrac{2\kappa + \p}{2\kappa} \psi_\mu^{-\ell}\tfrac{2\kappa}{2\kappa + \p}\Bigr]  \in \mathcal G_{\loc}^*.
\]
Similarly, if \(T =\psi_\mu^\ell S^{-1} \psi_\mu^{-\ell}S \) we may write
\[
\tfrac{2\kappa + \p}{2\kappa} T \tfrac{2\kappa}{2\kappa + \p}  = \Bigl[\tfrac{2\kappa + \p}{2\kappa} \psi_\mu^\ell \tfrac{2\kappa}{2\kappa + \p} \psi_\mu^{-\ell}\Bigr]\Bigl[ \psi_\mu^{\ell}  \bigl(S \tfrac{2\kappa}{2\kappa + \p} \bigr)^{-1} \psi_\mu^{-\ell}\bigl(S \tfrac{2\kappa}{2\kappa + \p} \bigr)\Bigr] \in \mathcal G_{\loc}^*.
\]

Next we consider the case $T=S\psi_\mu^\ell S^{-1} \psi_\mu^{-\ell}$. By the definition of $\mathcal G^\star$, we may write $S = \frac{2\vk}{2\vk+\p}\widetilde S$ or $S = \frac{2\vk}{2\vk-\p}\widetilde S$ for some \(|\vk|\geq 1\) and \(\widetilde S\in \mathcal G^\star\).  We present the details in the case \(S = \frac{2\vk}{2\vk+\p}\widetilde S\); the remaining case can be treated analogously. 

Using the identity
\[
\tfrac{2\kappa + \p}{2\kappa}\tfrac{2\vk}{2\vk+\p} = \tfrac{2\vk}{2\vk+\p} + \tfrac\vk\kappa\tfrac{\p}{2\vk+\p}
\]
and the symmetric identity with \(\vk\leftrightarrow\kappa\), we may express
\begin{align}
\tfrac{2\kappa+\p}{2\kappa}S\psi_\mu^\ell S^{-1}\tfrac{2\kappa}{2\kappa+\p}
&= S\psi_\mu^\ell S^{-1}\tfrac{2\kappa}{2\kappa+\p} + \tfrac\vk\kappa\tfrac\p{2\vk+\p}\widetilde S\psi_\mu^\ell\widetilde S^{-1}\tfrac{2\vk+\p}{2\vk}\tfrac{2\kappa}{2\kappa+\p}\notag\\
&= S\psi_\mu^\ell S^{-1}\tfrac{2\kappa}{2\kappa+\p} + \tfrac\vk\kappa\tfrac\p{2\vk+\p}\widetilde S\psi_\mu^\ell\widetilde S^{-1} \Bigl[\tfrac{2\kappa}{2\kappa+\p} + \tfrac\kappa\vk\tfrac\p{2\kappa+\p}\Bigr]\notag\\
&=S\psi_\mu^\ell S^{-1}\tfrac{2\kappa}{2\kappa+\p}+ \Bigl[\tfrac{2\kappa+\p}{2\kappa}\tfrac{2\vk}{2\vk+\p} - \tfrac{2\vk}{2\vk+\p}\Bigr]\widetilde S\psi_\mu^\ell\widetilde S^{-1} \tfrac{2\kappa}{2\kappa+\p}\notag\\
&\quad + \tfrac\p{2\vk+\p}\widetilde S\psi_\mu^\ell\widetilde S^{-1}\tfrac\p{2\kappa+\p}\notag\\
&=S\psi_\mu^\ell S^{-1}\tfrac{2\kappa}{2\kappa+\p}+ \tfrac{2\kappa+\p}{2\kappa}\tfrac{2\vk}{2\vk+\p}\widetilde S\psi_\mu^\ell\widetilde S^{-1} \tfrac{2\kappa}{2\kappa+\p}\label{switch2}\\
&\quad  - \tfrac{2\vk}{2\vk+\p}\widetilde S\psi_\mu^\ell\widetilde S^{-1} \tfrac{2\kappa}{2\kappa+\p} + \tfrac\p{2\vk+\p}\widetilde S\psi_\mu^\ell\widetilde S^{-1}\tfrac\p{2\kappa+\p}.\notag
\end{align}
We then apply \eqref{switch2} to get
\begin{align}
\tfrac{2\kappa+\p}{2\kappa} T\tfrac{2\kappa}{2\kappa+\p} &= \Bigl[\tfrac{2\kappa+\p}{2\kappa}S\psi_\mu^\ell S^{-1}\tfrac{2\kappa}{2\kappa+\p}\Bigr]\Bigl[\tfrac{2\kappa+\p}{2\kappa}\psi_\mu^{-\ell}\tfrac{2\kappa}{2\kappa+\p}\Bigr]\notag\\
& = T\tfrac{2\kappa}{2\kappa+\p} + \tfrac{2\vk}{2\vk+\p}\Bigl[\tfrac{2\kappa+\p}{2\kappa}\widetilde T\tfrac{2\kappa}{2\kappa+\p}\Bigr] - \tfrac{2\vk}{2\vk+\p}\widetilde T\tfrac{2\kappa}{2\kappa+\p}\notag\\
&\quad + \tfrac\p{2\vk+\p}\widetilde T\Bigl[\psi_\mu^\ell\tfrac\p{2\vk+\p}\psi_\mu^{-\ell}\Bigr]\Bigl[\psi_\mu^\ell \tfrac{2\kappa+\p}{2\kappa}\psi_\mu^{-\ell}\tfrac{2\kappa}{2\kappa+\p}\Bigr],\label{switcheroo}
\end{align}
where \(\widetilde T = \widetilde S\psi_\mu^\ell \widetilde S^{-1}\psi_\mu^{-\ell}\in \mathcal G_\loc^\star\).  By using Examples~\ref{EXs1} and~\ref{EXloc0}, we see that
\begin{align}\label{1788}
\tfrac\p{2\vk+\p} = \mathrm{Id} - \tfrac{2\vk}{2\vk+\p}  \qtq{and}
	\psi_\mu^\ell\tfrac\p{2\vk+\p}\psi_\mu^{-\ell} = \mathrm{Id} - \bigl[\psi_\mu^\ell\tfrac{2\vk}{2\vk+\p}\psi_\mu^{-\ell}\tfrac{2\vk+\p}{2\vk} \bigr] \tfrac{2\vk}{2\vk+\p}
\end{align}
are linear combinations of words in \(\mathcal G_\loc^\star\). As a consequence, \eqref{switcheroo} shows that if \(\tfrac{2\kappa+\p}{2\kappa}\widetilde T\tfrac{2\kappa}{2\kappa+\p}\) is a linear combination of words in \(\mathcal G_\loc^\star\) then so is \(\tfrac{2\kappa+\p}{2\kappa}T\tfrac{2\kappa}{2\kappa+\p}\). In this case, \eqref{claimsdepartment} follows by induction on the minimal number of letters required to spell \(S\) as in the sense \eqref{spelling}.  The base case corresponds to taking \(\widetilde T = \mathrm{Id}\).

Finally, we consider the case \(T = \psi_\mu^\ell S\psi_\mu^{-\ell}S^{-1}\). Arguing as in the previous case and using \eqref{switch2} (with \(\ell\) replaced by \(-\ell\)), we get
\begin{align*}
\tfrac{2\kappa+\p}{2\kappa}T\tfrac{2\kappa}{2\kappa+\p} &= \Bigl[\tfrac{2\kappa+\p}{2\kappa}\psi_\mu^\ell\tfrac{2\kappa}{2\kappa+\p}\Bigr]\Bigl[\tfrac{2\kappa+\p}{2\kappa}S\psi_\mu^{-\ell}S^{-1}\tfrac{2\kappa}{2\kappa+\p}\Bigr]\\
&= \Bigl[\tfrac{2\kappa+\p}{2\kappa}\psi_\mu^\ell\tfrac{2\kappa}{2\kappa+\p}\psi_\mu^{-\ell}\Bigr]T\tfrac{2\kappa}{2\kappa+\p} + \Bigl[\tfrac{2\kappa+\p}{2\kappa}\psi_\mu^\ell\tfrac{2\vk}{2\vk+\p}\psi_\mu^{-\ell}\tfrac{2\kappa}{2\kappa+\p}\Bigr]\tfrac{2\kappa+\p}{2\kappa}\widetilde T \tfrac{2\kappa}{2\kappa+\p}\\
&\quad  - \Bigl[\tfrac{2\kappa+\p}{2\kappa}\psi_\mu^\ell\tfrac{2\kappa}{2\kappa+\p}\psi_\mu^{-\ell}\Bigr]
	\Bigl[\psi_\mu^\ell\tfrac{2\vk}{2\vk+\p}\psi_\mu^{-\ell} \tfrac{2\vk+\p}{2\vk}\Bigr] \tfrac{2\vk}{2\vk+\p} \widetilde T \tfrac{2\kappa}{2\kappa+\p}\\
&\quad + \Bigl[\tfrac{2\kappa+\p}{2\kappa}\psi_\mu^\ell\tfrac{2\kappa}{2\kappa+\p}\psi_\mu^{-\ell}\Bigr]\Bigl[\psi_\mu^\ell\tfrac\p{2\vk+\p}\psi_\mu^{-\ell}\Bigr]\widetilde T\tfrac\p{2\kappa+\p},
\end{align*}
where \(\widetilde T = \psi_\mu^\ell \widetilde S\psi_\mu^{-\ell}\widetilde S^{-1}\in \mathcal G_\loc^\star\). Using \eqref{1788} and observing that
\[
\tfrac{2\kappa+\p}{2\kappa}\psi_\mu^\ell\tfrac{2\vk}{2\vk+\p}\psi_\mu^{-\ell}\tfrac{2\kappa}{2\kappa+\p} = \tfrac{2\vk}{2\vk+\p}\Bigl[\tfrac{2\kappa+\p}{2\kappa}\tfrac{2\vk+\p}{2\vk}\psi_\mu^\ell\tfrac{2\vk}{2\vk+\p}\tfrac{2\kappa}{2\kappa+\p}\psi_\mu^{-\ell}\Bigr]\Bigl[\psi_\mu^\ell\tfrac{2\kappa+\p}{2\kappa}\psi_\mu^{-\ell}\tfrac{2\kappa}{2\kappa+\p}\Bigr]
\]
is a word over the alphabet \(\mathcal G_\loc\), we again see that whenever \(\tfrac{2\kappa+\p}{2\kappa}\widetilde T\tfrac{2\kappa}{2\kappa+\p}\) is a linear combination of words in \(\mathcal G_\loc^\star\), so is \(\tfrac{2\kappa+\p}{2\kappa}T\tfrac{2\kappa}{2\kappa+\p}\). The proof of \eqref{claimsdepartment} in this case is completed by inducting on the minimal number of letters required to spell \(S\).
\end{proof}

\begin{lemma}[Properties of \(S_{\loc}(n)\)]\label{l:paraproperties}~
\begin{itemize}
\item[(i)] \emph{(Distribution of exponential weights)} If \(m\in S(n)\) then for any non-negative \(\ell_0+\dots+\ell_n=24\) we have
\begin{align*}
m[f_1,\dots,f_n]\,\psi_\mu^{24} &= \psi_\mu^{\ell_0} m_1[\psi_\mu^{\ell_1} f_1,\dots, \psi_\mu^{\ell_n} f_n],\\
m[f_1,\dots,f_n\,\psi_\mu^{24}] &= \psi_\mu^{\ell_0}m_2[\psi_\mu^{\ell_1} f_1,\dots, \psi_\mu^{\ell_n} f_n],
\end{align*}
where \(m_j\in S_{\loc}(n)\).
\smallskip
\item[(ii)] \emph{(Symmetry)} If \(m\in S_{\loc}(n)\) and \(\sigma\in \mathfrak S_n\) then
\[
 m[f_{\sigma(1)},\dots,f_{\sigma(n)}]\in S_{\loc}(n).
\]
\item[(iii)] \emph{(Interior products)} If \(2\leq k\leq n\), \(m_1\in S_{\loc}(k)\), and \(m_2\in S_{\loc}(n+1-k)\) then
\[
 m_1\bigl[f_1,\dots,f_{k-1},m_2[f_{k},\dots,f_n]\bigr]\in S_{\loc}(n).
\]
\item[(iv)]\emph{(Leibniz rule)} If \(n\geq 2\) and \(m\in S_{\loc}(n)\) then
\begin{align}
m[f_1',\dots,f_{n-1},f_n] &= \p m_1[f_1,\dots,f_n] + m_2[f_1,(2-\p)f_2,f_3,\dots,f_n]\label{paraLeibniz}\\
&\quad +\dots+ m_n[f_1,\dots,f_{n-1},(2-\p)f_n],\notag
\end{align}
where \(m_1,\dots,m_n\in S_{\loc}(n)\).
\item[(v)] \emph{(H\"older's inequality)} If \(m\in S_{\loc}(n)\) and \(1\leq p,p_j\leq \infty\) are so that \(\frac1p = \frac1{p_1}+\dots+\frac1{p_n}\) then
\begin{equation}\label{paraHolder}
\|m[f_1,\dots,f_n]\|_{L^p}\lesssim \prod_{j=1}^n\|f_j\|_{L^{p_j}},
\end{equation}
uniformly in \(\mu\).
\end{itemize}
\end{lemma}

\begin{proof}
Part (ii) is an immediate consequence of the definition. Part (iii) follows from an easy induction in $k$.

We turn now to the proof of part (i), which we prove by induction on $n$. We will prove the very slightly stronger statement that for all \(m\in S(n)\), and non-negative integers \(\ell_0 + \dots + \ell_n = \ell\leq 24\) we may find \(\widetilde m\in S_{\loc}(n)\) so that
\begin{equation}\label{locIDloc1}
m[f_1,\dots,f_n]\,\psi_\mu^{\ell} = \psi_\mu^{\ell_0} \widetilde m[\psi_\mu^{\ell_1} f_1,\dots, \psi_\mu^{\ell_n} f_n].
\end{equation}

For the base case of \eqref{locIDloc1} we write \(m\in S(1)\) as a linear combination
\[
m[f] = \sum_{i=1}^n c_iT_if,
\]
where \(c_i\in \C\) and \(T_i\in \mathcal G^\star\). In this case, \eqref{locIDloc1} follows from writing
\[
m[f]\,\psi_\mu^{\ell}=  \psi_\mu^{\ell_0} \sum_{i=1}^nc_i(\psi_\mu^{\ell_1}T_i\psi_\mu^{-\ell_1})[\psi_\mu^{\ell_1}f] =  \psi_\mu^{\ell_0} \widetilde m[\psi_\mu^{\ell_1}f] 
\]
and noting that \(\psi_\mu^{\ell_1}T_i\psi_\mu^{-\ell_1}=[\psi_\mu^{\ell_1}T_i\psi_\mu^{-\ell_1}T_i^{-1}]\;\! T_i \in \mathcal G_{\loc}^\star\).

For the inductive step, we fix $N\geq 2$ and assume that \eqref{locIDloc1} is true for all \(1\leq n\leq N-1\). We recall that elements of \(S(N)\) are linear combinations of paraproducts with the representation
\[
m[f_1,\dots,f_N] = T\Biggl[\,\prod_{j=1}^J m_j[f_{\sigma(n_{j-1} + 1)}\,,\,\dots\,,\,f_{\sigma(n_j)}]\Biggr],
\]
where \(0=n_0<\dots<n_J=N\), \(m_j\in S(n_j-n_{j-1})\), \(T\in \mathcal G^\star\), and \(\sigma\in \mathfrak S_N\). Without loss of generality, we assume \(\sigma = \mathrm{Id}\). 
Applying the inductive hypothesis we write
\begin{align*}
m[f_1,\dots,f_N]\,\psi_\mu^{\ell}
& = \psi_\mu^{\ell_0}(\psi_\mu^{\ell-\ell_0}T_i\psi_\mu^{\ell_0-\ell})\Biggl[\psi_\mu^{\ell-\ell_0}\prod_{j=1}^J m_j[f_{n_{j-1} + 1}\,,\,\dots\,,\,f_{n_j}]\Biggr]\\
& = \psi_\mu^{\ell_0} \widetilde T\Biggl[\,\prod_{j=1}^J \widetilde m_j\Bigl[\psi_\mu^{\ell_{n_{j-1}+1}}f_{n_{j-1} + 1}\,,\,\dots\,,\,\psi_\mu^{\ell_{n_j}}f_{n_j}\Bigr]\Biggr]
\end{align*}
where \(\widetilde T = \psi_\mu^{\ell - \ell_0}T\psi_\mu^{\ell_0-\ell}\in \mathcal G_\loc^\star\) and \(\widetilde m_j\in S_{\loc}(n_j - n_{j-1})\).
From~\eqref{mu mess}, it is then clear that \(\widetilde m\in S_\loc(N)\). The proof of the inductive step is completed by considering linear combinations of paraproducts of this form.

We now turn to part (iv), which is also proved by induction on $n$.  All the requisite ideas can be understood most transparently from the treatment of the base case $n=2$.  

Given a word \(T\in \mathcal G_\loc^\star\), we express
\begin{align}
\tfrac2{2-\p}T\p &=  \tfrac 2{2+\p} T\tfrac{2+\p}2\tfrac{2+\p}{2-\p}-\tfrac 2{2-\p}T\tfrac{2-\p} 2,\label{loco2}\\
\p T\tfrac2{2-\p} &=  \tfrac{2+\p} 2T\tfrac 2{2+\p}\tfrac{2+\p}{2-\p}-\tfrac{2-\p} 2T\tfrac 2{2-\p}.\label{loco1}
\end{align}
By Example~\ref{EXs1} (with \(\vk=-1\)) the operator \(\frac{2+\p}{2-\p}\) is a linear combination of words in \(\mathcal G^\star\). Combining this with Lemma~\ref{L:conjugate}, we see that both \(\LHS{loco2}\) and \(\LHS{loco1}\) are linear combinations of words in $\mathcal G_\loc^\star$.

For the base step $n=2$, it suffices to consider \(m\in S_{\loc}(2)\) that can be expressed as
\[
m[f_1,f_2] = T\bigl[h_1[f_1]\cdot h_2[f_2]\bigr],
\]
where \(T \in \mathcal G_\loc^\star\) and \(h_1,h_2\in S_\loc(1)\).  We define the paraproducts \(k_1,k_2\) via
\[
k_1[f] = \tfrac1{2-\p} h_1[f'],\qquad k_2[f] = \p h_2\bigl[\tfrac{f}{2-\p}\bigr].
\]
By \eqref{loco2} and \eqref{loco1} we see that \(k_1,k_2\in S_{\loc}(1)\). We also define
\[
\ell[f] = h_2\bigl[\tfrac 2{2-\p}f\bigr]
\]
and have \(\ell\in S_{\loc}(1)\) by definition. 

We then compute
\begin{align*}
m[f_1',f_2] &= T\bigl[(2-\p)k_1[f_1]\cdot h_2[f_2]\bigr]\\
&= T(2-\p)\bigl[k_1[f_1]\cdot h_2[f_2]\bigr] + T\bigl[k_1[f_1]\cdot\p h_2[f_2]\bigr]\\
&= (2-\p) S\bigl[k_1[f_1]\cdot h_2[f_2]\bigr] + T\bigl[k_1[f_1]\cdot k_2[(2-\p)f_2]\bigr]\\
&= \!\underbrace{-\p S\bigl[k_1[f_1]\cdot h_2[f_2]\bigr]}_{=\p m_1[f_1,f_2]} \!+\! \underbrace{S\bigl[k_1[f_1]\cdot \ell[(2-\p)f_2]\bigr] \!+\! T\bigl[k_1[f_1]\cdot k_2[(2-\p)f_2]\bigr]}_{= m_2[f_1,(2-\p)f_2]},
\end{align*}
where \(S = \tfrac2{2-\p}T\tfrac{2-\p}2\in\mathcal G_{\loc}^\star\) by \eqref{claimsdepartment} and \(m_1,m_2\in S_{\loc}(2)\) by definition.

To prove part (v), we first apply Lemma~\ref{l:mult comm} to see that every element of \(\mathcal G_{\loc}^\star\) is bounded on \(L^p\) whenever \(1\leq p\leq\infty\). The claim follows from a final induction on~\(n\).
\end{proof}

To state our paraproduct estimates for \eqref{DNLS}, it will once again be convenient to employ the convention that if \(m\in S_{\loc}(n)\) and \(f_1,\dots,f_n\) satisfy the estimates \eqref{admissible}, then we denote the expression \(m[f_1,\dots,f_n]\) by \(m[f,\dots,f]\).  Moreover, if an expression involves several paraproducts \(m_1,\dots,m_k\in S_{\loc}(n)\), then we denote each paraproduct by \(m\).

With this convention in hand, we turn to our paraproduct estimates for \eqref{DNLS}:

\begin{lemma}\label{l:paraproducts I} Let \(m\in S_{\loc}(4)\) and \(f\) satisfy \eqref{admissible}.  Then, the following estimates hold uniformly for  \(h\in \R\) and \(|\vk|\geq 1:\)
\begin{align}\label{m4 type II bound}
&\int\biggl|\int m\Bigl[\psi_\mu^6f,\psi_\mu^6f,\tfrac{\psi_\mu^6f}{2\vk - \p},\tfrac{2-\p}{2\vk - \p}(\psi_\mu^6f)\Bigr]\,dx\biggr|\,e^{-\frac1{200}|h-\mu|}\,d\mu\\
&\qquad+\int\biggl|\int m\Bigl[\psi_\mu^6f,\psi_\mu^6f,\psi_\mu^6f,\tfrac{2-\p}{4\vk^2 - \p^2}(\psi_\mu^6f)\Bigr]\,dx\biggr|\,e^{-\frac1{200}|h-\mu|}\,d\mu\notag\\
&\qquad\qquad\lesssim |\vk|^{-1}\|q\|_{F^{\frac12}(h)}^2\|q\|_{E_{2\sigma, \vk}^\sigma}^2.\notag
\end{align}
Further, if \(m\in S_{\loc}(6)\) then uniformly for  \(h\in \R\) and \(|\vk|\geq 1:\)
\begin{align}\label{m4 type III bound}
&\int \biggl|\int m\Bigl[\psi_\mu^4f,\psi_\mu^4f,\psi_\mu^4f,\psi_\mu^4f,\tfrac{\psi_\mu^4f}{2\vk - \p},\tfrac{\psi_\mu^4f}{2\vk - \p}\Bigr]\,dx\biggr|\,e^{-\frac1{200}|h-\mu|}\,d\mu\\
&\qquad\lesssim|\vk|^{-1}\|q\|_{F^{\frac12}(h)}^2\|q\|_{E_{\sigma}^\sigma}^2\|q\|_{E_{2\sigma, \vk}^\sigma}^2.\notag
\end{align}
\end{lemma}

\begin{proof}
To prove these estimates, we decompose each $f$ into Littlewood--Paley pieces and estimate the two highest frequencies in $L^2$ with a view to employ \eqref{LS loc}.  To estimate the remaining lower frequency pieces, we rely on the following lemma.

\begin{lemma}\label{L:Fmu}
For any $f$ satisfying \eqref{admissible}, we have
\begin{align}
\|(\psi_\mu^\ell f)_N\|_{L^\infty}&\lesssim N^{\frac12-\sigma}(1+N)^\sigma\|q\|_{E_\sigma^\sigma},\label{Fmu-E1}\\
\|(\psi_\mu^\ell f)_N\|_{L^\infty}&\lesssim |\vk|^{-\sigma}N^{\frac12-\sigma}(|\vk|+N)^{2\sigma}\|q\|_{E_{2\sigma,\vk}^\sigma},\label{Fmu-E2}
\end{align}
uniformly for \(0\leq \ell\leq 12\) and \(N\in 2^\Z\).  In particular,
\begin{align}
\sum_{M\leq N}\tfrac{1}{|\vk|+M}\|(\psi_\mu^\ell f)_M\|_{L^\infty}&\lesssim \tfrac{N^{\frac12-\sigma}(1+N)^\sigma}{|\vk|^{\frac12} (|\vk|+N)^{\frac12}}\|q\|_{E_\sigma^\sigma},\label{Rock}\\
\sum_{M\leq N}\tfrac{1}{|\vk|+M}\|(\psi_\mu^\ell f)_M\|_{L^\infty}&\lesssim \tfrac{N^{\frac12}}{|\vk|^{\frac12} (|\vk|+N)^{\frac12}}\|q\|_{E_{2\sigma,\vk}^\sigma}.\label{Hard Place}
\end{align}
\end{lemma}

\begin{proof}
The bounds \eqref{Fmu-E1} and \eqref{Fmu-E2} follow easily from Bernstein's inequality, \eqref{E loc}, and \eqref{admissible}.  To obtain the last two bounds, one considers separately the contribution from $M\leq |\vk|$ and $M>|\vk|$.
\end{proof}

We start by considering \eqref{m4 type III bound}. Decomposing into Littlewood--Paley pieces and employing Lemma~\ref{L:Fmu}, we find
\begin{align*}
&\int m\Bigl[\psi_\mu^4f,\psi_\mu^4f,\psi_\mu^4f,\psi_\mu^4f,\tfrac{\psi_\mu^4 f}{2\vk - \p},\tfrac{\psi_\mu^4 f}{2\vk - \p}\Bigr]\,dx\\
&\qquad\lesssim\sum_{N_1\geq  \dots\geq N_6} \tfrac1{(|\vk|+N_5)(|\vk|+N_6)}\prod_{j=1}^2\|P_{N_j}(\psi_\mu^4f)\|_{L^2}\prod_{j=3}^6\|P_{N_j}(\psi_\mu^4f)\|_{L^\infty}\\
&\qquad\lesssim \sum_{N_1\geq N_2} \tfrac{|\vk|^{-1}N_2^{2-4\sigma}(1+N_2)^{2\sigma}}{(|\vk|+N_2)^{1-2\sigma}}\|P_{N_1}(\psi_\mu^4 f)\|_{L^2}\|P_{N_2}(\psi_\mu^4 f)\|_{L^2}\|q\|_{E_{\sigma}^\sigma}^2\|q\|_{E_{2\sigma, \vk}^\sigma}^2\\
&\qquad\lesssim |\vk|^{-1}\|\psi_\mu^4 f\|_{H^{\frac12}}^2\|q\|_{E_{\sigma}^\sigma}^2\|q\|_{E_{2\sigma, \vk}^\sigma}^2.
\end{align*}
In view of \eqref{LS loc} and \eqref{admissible}, this contribution is acceptable.

Next we consider the first term on $\LHS{m4 type II bound}$. Decomposing once again into Littlewood--Paley pieces, we have
\begin{align*}
&m\Bigl[\psi_\mu^6f,\psi_\mu^6f,\tfrac{\psi_\mu^6f}{2\vk - \p},\tfrac{2-\p}{2\vk - \p}\bigl(\psi_\mu^6f\bigr)\Bigr]\\
&\qquad= \sum_{M_1,M_2,M_3,M_4}m\Bigl[P_{M_1}(\psi_\mu^6f),P_{M_2}(\psi_\mu^6f),\tfrac{P_{M_3}(\psi_\mu^6f)}{2\vk - \p},\tfrac{2-\p}{2\vk - \p}P_{M_4}(\psi_\mu^6f)\Bigr].
\end{align*}
To continue, we split the sum into two pieces: the first where \(M_4 < \max\{M_j\}\) and the second where \(M_4 = \max\{M_j\}\).

For the first summand, we apply \eqref{paraHolder} to estimate the terms with the two highest frequencies in \(L^2\) and the remaining terms in \(L^\infty\). We then apply Bernstein's inequality followed by \eqref{Hard Place} to the \(N_4\)-term and \eqref{Fmu-E2} to the \(N_3\)-term to estimate
\begin{align*}
&\sum_{\substack{M_1,M_2,M_3,M_4\\M_4<\max\{M_j\}}}\left|\int m\Bigl[P_{M_1}(\psi_\mu^6f),P_{M_2}(\psi_\mu^6f),\tfrac{P_{M_3}(\psi_\mu^6f)}{2\vk - \p},\tfrac{2-\p}{2\vk - \p}P_{M_4}(\psi_\mu^6f)\Bigr]\,dx\right|\\
&\qquad\lesssim \sum_{N_1\geq N_2\geq N_3\geq N_4}\tfrac{(1+N_2)}{(|\vk| + N_2)(|\vk| + N_4)}\prod_{j=1}^2\|P_{N_j}(\psi_\mu^6 f)\|_{L^2}\prod_{j=3}^4\|P_{N_j}(\psi_\mu^6 f)\|_{L^\infty}\\
&\qquad\lesssim \sum_{N_1\geq N_2\geq N_3}\tfrac{(1+N_2)N_3^{1-\sigma}}{|\vk|^{\frac12+\sigma}(|\vk| + N_2)(|\vk| + N_3)^{\frac12-2\sigma}}\|P_{N_1}(\psi_\mu^6 f)\|_{L^2}\|P_{N_2}(\psi_\mu^6 f)\|_{L^2}\|q\|_{E_{2\sigma,\vk}^\sigma}^2\\
&\qquad\lesssim\sum_{N_1\geq N_2}\tfrac{(1+N_2)^{2-\sigma}}{|\vk|^{\frac12+ \sigma}(|\vk| + N_2)^{\frac32-2\sigma}}\|P_{N_1}(\psi_\mu^6 f)\|_{L^2}\|P_{N_2}(\psi_\mu^6 f)\|_{L^2}\|q\|_{E_{2\sigma,\vk}^\sigma}^2\\
&\qquad\lesssim\sum_{N_1\geq N_2}\tfrac{(1+N_2)^{1-\sigma}}{|\vk|^{\frac12+ \sigma}(|\vk| + N_2)^{\frac32-2\sigma}}\bigl(\tfrac{1+N_2}{1+N_1}\bigr)^{\frac12}\|P_{N_1}(\psi_\mu^6 f)\|_{H^{\frac12}}\|P_{N_2}(\psi_\mu^6 f)\|_{H^{\frac12}}\|q\|_{E_{2\sigma,\vk}^\sigma}^2\\
&\qquad\lesssim |\vk|^{-1}\|\psi_\mu^6f\|_{H^{\frac12}}^2\|q\|_{E_{2\sigma,\vk}^\sigma}^2.
\end{align*}
Note that the frequency parameters $N_j$ represent a permutation of the parameters \(M_j\) so as to account for the largest contribution. Integrating with respect to the measure \(e^{-\frac1{200}|h-\mu|}\,d\mu\) and applying \eqref{LS loc} and \eqref{admissible}, we obtain an acceptable contribution.

For the second summand, we first use \eqref{paraLeibniz} to redistribute the derivative:
\begin{align*}
&\sum_{M_4\geq M_1,M_2,M_3}\left|\int m\Bigl[P_{M_1}(\psi_\mu^6f),P_{M_2}(\psi_\mu^6f),\tfrac{P_{M_3}(\psi_\mu^6f)}{2\vk - \p},\tfrac{2-\p}{2\vk - \p}P_{M_4}(\psi_\mu^6f)\Bigr]\,dx\right|\\
&\qquad\leq \sum_{M_4\geq M_1,M_2,M_3}\left|\int m\Bigl[(2-\p)P_{M_1}(\psi_\mu^6f),P_{M_2}(\psi_\mu^6f),\tfrac{P_{M_3}(\psi_\mu^6f)}{2\vk - \p},\tfrac{P_{M_4}(\psi_\mu^6f)}{2\vk - \p}\Bigr]\,dx\right|\\
&\qquad\quad + \!\sum_{M_4\geq M_1,M_2,M_3}\left|\int m\Bigl[P_{M_1}(\psi_\mu^6f),(2-\p)P_{M_2}(\psi_\mu^6f),\tfrac{P_{M_3}(\psi_\mu^6f)}{2\vk - \p},\tfrac{P_{M_4}(\psi_\mu^6f)}{2\vk - \p}\Bigr]\,dx\right|\\
&\qquad\quad + \!\sum_{M_4\geq M_1,M_2,M_3}\left|\int m\Bigl[P_{M_1}(\psi_\mu^6f),P_{M_2}(\psi_\mu^6f),\tfrac{2-\p}{2\vk - \p}P_{M_3}(\psi_\mu^6f),\tfrac{P_{M_4}(\psi_\mu^6f)}{2\vk - \p}\Bigr]\,dx\right|,
\end{align*}
where each \(m_j\in S_{\loc}(4)\). We then proceed as in the first case to estimate each term by
\begin{align*}
&\sum_{N_1\geq N_2\geq N_3\geq N_4}\tfrac{(1+N_2)}{(|\vk| + N_1)(|\vk| + N_4)}\prod_{j=1}^2\|P_{N_j}(\psi_\mu^6 f)\|_{L^2}\prod_{j=3}^4\|P_{N_j}(\psi_\mu^6 f)\|_{L^\infty}\\
&\qquad\lesssim\sum_{N_1\geq N_2}\tfrac{(1+N_2)^{2-\sigma}}{|\vk|^{\sigma+\frac12}(|\vk| + N_1)(|\vk| + N_2)^{\frac12-2\sigma}}\|P_{N_1}(\psi_\mu^6 f)\|_{L^2}\|P_{N_2}(\psi_\mu^6 f)\|_{L^2}\|q\|_{E_{2\sigma,\vk}^\sigma}^2\\
&\qquad\lesssim |\vk|^{-1}\|\psi_\mu^6f\|_{H^{\frac12}}^2\|q\|_{E_{2\sigma,\vk}^\sigma}^2,
\end{align*}
which is seen to be acceptable after an application of \eqref{LS loc} and \eqref{admissible}.

Applying a parallel argument, the second term on $\LHS{m4 type II bound}$ can be bounded by
\[
\sum_{N_1\geq N_2\geq N_3\geq N_4}\tfrac{(1+N_2)}{(|\vk| + N_2)(|\vk| + N_4)}\prod_{j=1}^2\|P_{N_j}(\psi_\mu^6 f)\|_{L^2}\prod_{j=3}^4\|P_{N_j}(\psi_\mu^6 f)\|_{L^\infty},
\]
which is acceptable, as demonstrated above. This completes the proof of \eqref{m4 type II bound}.
\end{proof}

Combining Proposition~\ref{p:Asymptotics} and Lemma~\ref{l:paraproducts I}, we have the following:

\begin{lemma}\label{l:DNLS LS helper} The following estimate holds uniformly for \(q\in Q_*\), \(h\in \R\), and \(\vk\geq 1:\)
\begin{align}\label{jDNLS geq4}
\int\Biggl|\Im\int j_{\DNLS}\sbrack{\geq 4}(\vk)\,\psi_\mu^{24}\,dx\Biggr|\,e^{-\frac1{200}|h - \mu|}\,d\mu\lesssim |\vk|^{-1}\|q\|_{F^{\frac12}(h)}^2\|q\|_{E_{2\sigma,\vk}^\sigma}^2.
\end{align} 
In particular, in view of Lemma~\ref{l:E int}, 
\begin{align}\label{jDNLS geq4 int}
\int_\kappa^\infty\int\Biggl|\Im\int_{-1}^1 \int j_{\DNLS}\sbrack{\geq 4}(\vk)\,\psi_\mu^{24}\,dx\, dt\Biggr|\,e^{-\frac1{200}|h - \mu|}\,d\mu\, d\vk\lesssim \|q\|_{X^{\frac12}}^2\|q\|_{L^\infty_t E_{\sigma,\kappa}^\sigma}^2.
\end{align}
\end{lemma}
\begin{proof}
Using \eqref{DNLS current} and \eqref{sym}, we may write
\begin{align}\label{jDNLS rewriting}
\Im j_{\DNLS}\sbrack{\geq 4}(\vk) = \Im \Biggl\{&2\Bigl[\sqrt\vk\bigl(\tfrac{g_{21}}{2+\gamma}\bigr)\sbrack{\geq 3}(\vk) + \sqrt{-\vk}\bigl(\tfrac{g_{21}}{2+\gamma}\bigr)\sbrack{\geq 3}(-\vk)\Bigr] q\notag\\
& + \Bigl[\tfrac1{\sqrt\vk}\bigl(\tfrac{g_{21}}{2+\gamma}\bigr)\sbrack{\geq 3}(\vk) + \tfrac1{\sqrt{-\vk}}\bigl(\tfrac{g_{21}}{2+\gamma}\bigr)\sbrack{\geq 3}(-\vk)\Bigr] q'\notag\\
& + i\Bigl[\tfrac1{\sqrt\vk}\bigl(\tfrac{g_{21}}{2+\gamma}\bigr)(\vk) + \tfrac1{\sqrt{-\vk}}\bigl(\tfrac{g_{21}}{2+\gamma}\bigr)(-\vk)\Bigr]|q|^2q\Biggr\}.
\end{align}

Using \eqref{paraproduct for frac geq}, we may integrate by parts to obtain
\begin{align*}
\int \tfrac1{\sqrt\vk}\bigl(\tfrac{g_{21}}{2+\gamma}\bigr)\sbrack{\geq 3}(\vk)\,q'\,\psi_\mu^{24}\,dx &= \int m\Bigl[\psi_\mu^6f,\psi_\mu^6f,\tfrac{\psi_\mu^6f}{2\vk-\p},\tfrac{2-\p}{2\vk - \p}(\psi_\mu^6 f)\Bigr]\,dx,
\end{align*}
where \(m\in S_{\loc}(4)\). We then apply \eqref{m4 type II bound} to obtain
\begin{align}
&\int \Biggl|\int \tfrac1{\sqrt\vk}\bigl(\tfrac{g_{21}}{2+\gamma}\bigr)\sbrack{\geq 3}(\vk)\,q'\,\psi_\mu^{24}\,dx\Biggr|\,e^{-\frac1{200}|h-\mu|}\,d\mu\lesssim  |\vk|^{-1}\|q\|_{F^{\frac12}(h)}^2\|q\|_{E_{2\sigma,\vk}^\sigma}^2.\label{DNLS frac bound II}
\end{align}

Similarly, from \eqref{paraproduct for frac geq} we have
\begin{align*}
\int \tfrac1{\sqrt\vk}\bigl(\tfrac{g_{21}}{2+\gamma}\bigr)\sbrack{\geq 3}(\vk)\,|q|^2q\,\psi_\mu^{24}\,dx &= \int m\Bigl[\psi_\mu^4f,\psi_\mu^4f,\psi_\mu^4f,\psi_\mu^4f,\tfrac{\psi_\mu^4f}{2\vk - \p},\tfrac{\psi_\mu^4f}{2\vk - \p}\Bigr]\,dx,
\end{align*}
where \(m\in S_{\loc}(6)\). Recalling \eqref{linear} and \eqref{bar q to bar u} we obtain
\begin{align*}
&\int \Bigl[\tfrac1{\sqrt\vk}\bigl(\tfrac{g_{21}}{2+\gamma}\bigr)\sbrack{1}(\vk) +\tfrac1{\sqrt{-\vk}}\bigl(\tfrac{g_{21}}{2+\gamma}\bigr)\sbrack{1}(-\vk) \Bigr]\,|q|^2q\,\psi_\mu^{24}\,dx \\
&\qquad =i\int \bar u'|q|^2q\,\psi_\mu^{24}\,dx = \int m\Bigl[\psi_\mu^6f,\psi_\mu^6f,\psi_\mu^6f,\tfrac{2-\p}{4\vk^2 - \p^2}(\psi_\mu^6f)\Bigr]\,dx,
\end{align*}
where \(m\in S_{\loc}(4)\) and we recall that \(u = \frac q{4\vk^2 - \p^2}\). Applying \eqref{m4 type II bound} and \eqref{m4 type III bound} then gives us
\begin{align}
&\int \Biggl|\int \Bigl[\tfrac1{\sqrt\vk}\bigl(\tfrac{g_{21}}{2+\gamma}\bigr)(\vk) + \tfrac1{\sqrt{-\vk}}\bigl(\tfrac{g_{21}}{2+\gamma}\bigr)(-\vk)\Bigr]\,|q|^2q\,\psi_\mu^{24}\,dx\Biggr|\,e^{-\frac1{200}|h-\mu|}\,d\mu\label{DNLS frac bound III}\\
&\qquad\qquad\qquad\qquad\qquad\qquad\qquad\lesssim |\vk|^{-1}\|q\|_{F^{\frac12}(h)}^2\|q\|_{E_{2\sigma,\vk}^\sigma}^2.\notag
\end{align}

Another application of \eqref{paraproduct for frac geq} gives us
\begin{align*}
\int \sqrt\vk\bigl(\tfrac{g_{21}}{2+\gamma}\bigr)\sbrack{\geq 5}(\vk)\, q\,\psi_\mu^{24}\,dx = \int m\Bigl[\psi_\mu^4f,\psi_\mu^4f,\psi_\mu^4f,\psi_\mu^4f,\tfrac{\psi_\mu^4f}{2\vk - \p},\tfrac{\psi_\mu^4f}{2\vk - \p}\Bigr]\,dx,
\end{align*}
where \(m\in S_{\loc}(6)\). Further, from \eqref{g21 frac 3 expansion} we have
\begin{align*}
\int \Bigl[\sqrt\vk\bigl(\tfrac{g_{21}}{2+\gamma}\bigr)\sbrack{3}(\vk)- 4\vk^5|u|^2\bar u\Bigr]\,q\,\psi_\mu^{24}\,dx= \int m\Bigl[\psi_\mu^6f,\psi_\mu^6f,\tfrac{\psi_\mu^6f}{2\vk-\p},\tfrac{2-\p}{2\vk - \p}(\psi_\mu^6f)\Bigr]\,dx,
\end{align*}
where \(m\in S_{\loc}(4)\). Once again we use \eqref{m4 type II bound} and \eqref{m4 type III bound} to bound
\begin{align*}
\int \Biggl|\int \Bigl[\sqrt\vk\bigl(\tfrac{g_{21}}{2+\gamma}\bigr)\sbrack{\geq 3}(\vk) - 4\vk^5|u|^2\bar u\Bigr]\,q\,\psi_\mu^{24}\,dx\Biggr|\,e^{-\frac1{200}|h-\mu|}\,d\mu\lesssim |\vk|^{-1}\|q\|_{F^{\frac12}(h)}^2\|q\|_{E_{2\sigma,\vk}^\sigma}^2.
\end{align*}
In particular,
\begin{align}\label{DNLS frac bound I}
\int \Biggl|\int \Bigl[\sqrt\vk\bigl(\tfrac{g_{21}}{2+\gamma}\bigr)\sbrack{\geq 3}(\vk) &+ \sqrt{-\vk}\bigl(\tfrac{g_{21}}{2+\gamma}\bigr)\sbrack{\geq 3}(-\vk)\Bigr]\,q\,\psi_\mu^{24}\,dx\Biggr|\,e^{-\frac1{200}|h-\mu|}\,d\mu\notag\\
&\lesssim |\vk|^{-1}\|q\|_{F^{\frac12}(h)}^2\|q\|_{E_{2\sigma,\vk}^\sigma}^2.
\end{align}

The estimate \eqref{jDNLS geq4} then follows from combining \eqref{DNLS frac bound II} through \eqref{DNLS frac bound I}.
\end{proof}

We are now in a position to complete the:

\begin{proof}[Proof of Proposition~\ref{p:LS}]

Combining the identity \eqref{IBP in dt rho} with Lemmas~\ref{L:rho} and~\ref{l:DNLS LS helper} we obtain
\begin{align}\label{2:07}
\int_\kappa^\infty\int\Biggl|\Im\int_{-1}^1 & \int j_{\DNLS}\sbrack{2}(\vk)\,\psi_\mu^{24}\,dx\, dt\Biggr|\,e^{-\frac1{200}|h - \mu|}\,d\mu\, d\vk\notag\\
&\lesssim \|q\|_{L^\infty_t E_{\sigma,\kappa}^\sigma}^2\Bigl[1 + \|q\|_{L^\infty_tL^2}^2+\|q\|_{X^{\frac12}}^2\Bigr].
\end{align}

It remains to consider the quadratic terms in  $j_{\DNLS}$. A computation yields
\begin{align}\label{j2}
\Im j_{\DNLS}\sbrack 2 = 2\Re\Bigl\{\tfrac{q'}{4\vk^2 - \p^2} \bar q'\Bigr\} - \Re\Bigl\{\tfrac{q'}{4\vk^2 - \p^2}\bar q\Bigr\}'.
\end{align}
We may then integrate by parts to obtain
\begin{align*}
\Im \int j_{\DNLS}\sbrack 2(\vk)\,\psi_\mu^{24}\,dx &= 2\|(\psi_\mu^{12}q)'\|_{H^{-1}_\varkappa}^2 + 2\Re\int [\psi_\mu^{12},\tfrac{\p}{4\vk^2 - \p^2}]q \,(\psi_\mu^{12}\bar q)'\,dx
\end{align*}
To estimate the error, we use that
\[
[\psi_\mu^{12},\tfrac{\p}{4\vk^2 - \p^2}] = -\tfrac1{4\vk^2 - \p^2}\bigl[(\psi_\mu^{12})''+2(\psi_\mu^{12})'\p\bigr]\tfrac{\p}{4\vk^2 - \p^2} - \tfrac1{4\vk^2 - \p^2}(\psi_\mu^{12})'
\]
together with \eqref{E loc} to bound
\begin{align*}
\biggl|\int [\psi_\mu^{12},\tfrac{\p}{4\vk^2 - \p^2}]q \,(\psi_\mu^{12}\bar q)'\,dx\biggr|&\lesssim \vk^{- 1}\|q\|_{E_{2\sigma,\vk}^\sigma}^2.
\end{align*}
Thus,
\begin{align*}
\|\psi_\mu^{12}q\|_{E_{1,\vk}^1}^2\lesssim\Biggl|\Im \int j_\DNLS\sbrack{2}(\varkappa) \,\psi_\mu^{24} \, dx\Biggr| + \vk^{-1}\|q\|_{E_{2\sigma, \vk}^\sigma}^2.
\end{align*}
Integrating in $\vk$ over $[\kappa,\infty)$, using \eqref{E int} and then \eqref{2:07} we get
\begin{align*}
\kappa \int\|\psi_\mu^{12}q\|_{L^2_tE_{\frac12,\kappa}^1}^2\,e^{-\frac1{200}|h-\mu|}\,d\mu
\lesssim \|q\|_{L_t^\infty E_{\sigma,\kappa}^\sigma}^2\Bigl[ 1+ \|q\|_{L^\infty_tL_x^2}^2 + \|q\|_{X^{\frac12}}^2\Bigr].
\end{align*}
Estimating
\begin{align*}
\bigl\|\tfrac{\psi_\mu^{12}q}{\sqrt{4\kappa^2 - \p^2}}\bigr\|_{L^2_tH^{\frac32}}^2 &\lesssim \bigl\|\tfrac{P_{\leq 1}(\psi_\mu^{12}q)}{\sqrt{4\kappa^2 - \p^2}}\bigr\|_{L^2_tH^{\frac32}}^2 + \bigl\|\tfrac{P_{>1}(\psi_\mu^{12}q)}{\sqrt{4\kappa^2 - \p^2}}\bigr\|_{L^2_tH^{\frac32}}^2\\
&\lesssim \kappa^{-2}\|q\|_{L^\infty_tL^2_x}^2 + \kappa\|\psi_\mu^{12}q\|_{L^2_tE_{\frac12,\kappa}^1}^2,
\end{align*}
we then get
\begin{equation}\label{final LS}
\|q\|_{X_\kappa^{\frac12}}^2\lesssim \|q\|_{L^\infty_tE_{\sigma,\kappa}^\sigma}^2\|q\|_{X^{\frac12}}^2 + \|q\|_{L^\infty_tE_{\sigma,\kappa}^\sigma}^2\bigl[ 1+ \|q\|_{L^\infty_tL_x^2}^2\bigr] + \kappa^{-2}\|q\|_{L^\infty_tL^2_x}^2
\end{equation}

To prove \eqref{local smoothing} we first take $\kappa=1$ and use \eqref{goodness} to absorb the first term on the right-hand side of \eqref{final LS} into the left-hand side, and then invoke the conservation of the $L^2$ norm. The estimate \eqref{refined local smoothing} follows from using \eqref{local smoothing} in \eqref{final LS} and then applying \eqref{E equi prop}.
\end{proof}

\section{Tightness}\label{S:Tight}

The goal of this section is to show that the family of orbits emanating from an $L^2$-precompact set of Schwartz initial data remains tight, at least for times $t\in[-1,1]$.  We begin by constructing a suitable function $\phi_R$ for localizing to the spatial region $|x|\geq R$ for given $R\geq 100$.  For such $R$, we first define
\[
\chi_R(x) = \tfrac1R\int_{R\leq |\mu|\leq 2R}\sgn(\mu)\psi_\mu^{24}(x)\,d\mu.
\]
Note that \(\chi_R\) is odd and vanishes at \(x=0\). We then define
\[
\phi_R(x) = \int_0^x\chi_R(y)\,dy,
\]
which is even and everywhere positive. Indeed
\[
\phi_R(x)\gtrsim 1\qtq{uniformly for}|x|\geq 2R\geq 200.
\]
In view of this property of \(\phi_R\), a bounded subset $Q\subset L^2$ is tight in $L^2$ if
\begin{align*}
    \int \phi_R(x) |q(x)|^2\,dx\to 0\quad\text{as $R\to\infty$, uniformly for $q\in Q$.}
\end{align*}

\begin{proposition}[Tightness for \eqref{DNLS}]\label{P:tight}
Suppose \(Q\subset \Schw\) is an \(L^2\) bounded and equicontinuous set for which
\[
Q_* = \bigl\{e^{tJ\nabla H}q:|t|\leq 1\text{ and }q\in Q\bigr\}
\]
is $\delta$-good for some sufficiently small $\delta$.  If $Q$ is also tight, then so too is $Q_*$.
\end{proposition}

\begin{proof}
From the microscopic conservation law \eqref{micro M} we obtain
\begin{align}\label{alice}
    \frac{d}{dt}\int \phi_R|q|^2\,dx &= \tfrac1R\int_{R\leq |\mu|\leq 2R} \sgn(\mu) \int \bigl[2 \Im  (q' \bar q) +\tfrac{3}{2} |q|^4\bigr]\,\psi_\mu^{24} \,dx\,d\mu.
\end{align}
Our goal is to deduce the tightness of $Q_*$ from that of $Q$ by estimating the right-hand side above in \(L^1([-1,1];dt)\) and showing that it converges to zero as \(R\to \infty\).

For the first term on \(\RHS{alice}\) we have
\begin{align*}
\int_{-1}^1\biggl|\int q'\bar q\,\psi_\mu^{24}\,dx\biggl|\,dt \lesssim \|\psi_\mu^{12}q\|_{L^2_tH^{\frac12}}^2 \lesssim \kappa\|\psi_\mu^{12}q\|_{L_{t,x}^2}^2 + \bigl\|\tfrac{\psi_\mu^{12}q}{\sqrt{4\kappa^2 - \p^2}}\bigr\|_{L^2_tH^{\frac32}}^2,
\end{align*}
where the second inequality follows from decomposing into frequencies \(\leq \kappa\) and \(>\kappa\). For the second term on \(\RHS{alice}\), we apply the Gagliardo--Nirenberg inequality to bound
\begin{align*}
\|\psi_\mu^6 q\|_{L^4_{t,x}}^4\lesssim \|\psi_\mu^6 q\|_{L^2_tH^{\frac12}}^2\|\psi_\mu^6 q\|_{L^\infty_tL^2_x}^2\lesssim \Bigl[ \kappa\|\psi_\mu^6 q\|_{L^2_{t,x}}^2 + \bigl\|\tfrac{\psi_\mu^6q}{\sqrt{4\kappa^2 - \p^2}}\bigr\|_{L^2_tH^{\frac32}}^2\Bigr]\|\psi_\mu^6 q\|_{L^\infty_tL^2_x}^2.
\end{align*}
In this way we deduce that
\[
\|\RHS{alice}\|_{L^1_t} \lesssim\Bigl[ \tfrac\kappa R\|\psi_\mu^{12}q\|_{L^2_\mu L^2_tL^2_x}^2 + \bigl\|\tfrac{\psi_\mu^{12}q}{\sqrt{4\kappa^2 - \p^2}}\bigr\|_{L^\infty_\mu L^2_tH^{\frac32}}^2\Bigr]\bigl[1 + \|q\|_{L^\infty_tL^2_x}\bigr].
\]
By Fubini, \eqref{psi int}, and \eqref{LS alt} we then have
\[
\|\RHS{alice}\|_{L^1_t} \lesssim \Bigl[\tfrac\kappa R\|q\|_{L^\infty_tL^2_x}^2 + \|q\|_{X_\kappa^{\frac12}}^2\Bigr]\bigl[1 + \|q\|_{L^\infty_tL^2_x}\bigr],
\]
which can be made arbitrarily small by first choosing \(\kappa\) large and applying \eqref{refined local smoothing}, and then choosing \(R\) large.
\end{proof}

\section{Local smoothing for the difference flow}\label{sect: diff local smoothing}

In this section we prove local smoothing for Schwartz solutions of the difference flow.

\begin{proposition}[Local smoothing for the difference flow] \label{p: diff local smoothing}
Let \(Q\subset \Schw\) be an \(L^2\) bounded and equicontinuous set such that
\[
Q_* = \bigl\{q(t; \kappa) = e^{tJ\nabla (H - H_\kappa)}q:\, t\in \R, \, q\in Q,\text{ and } \kappa\geq 2\bigr\}
\]
is a $\delta$-good set for a sufficiently small $\delta>0$. Then  the local smoothing estimate
\begin{equation}\label{degenerate local smoothing}
\|q\|_{X_\kappa^{\frac12}}\lesssim \|q(0)\|_{L^2}.
\end{equation}
holds uniformly for \(\kappa\geq 2\) and $q(0)\in Q$.
\end{proposition}

Our proof will mirror that of Proposition~\ref{p:LS}. Once again, we take \(0<\sigma<\frac12\) and \(Q_*\) as in the hypothesis of Proposition~\ref{p: diff local smoothing}. Taking \(\phi_\mu\) to be defined as in \eqref{phi h} and using Propostion~\ref{P:micro laws}, we integrate by parts to obtain
\begin{align}\label{IBP diff}
\Im\int_{-1}^1\int j_{\diff}(q;\vk,\kappa)\,\psi_\mu^{24}\,dx\,dt = \Im\int\rho(q;\vk)\phi_\mu\,dx\Big|_{t=-1}^{t=1}.
\end{align}

Once again, our main challenge will be to estimate the remainder terms \(j_{\diff}\sbrack{\geq 4}\). To this end, we start with the following collection of paraproduct estimates:

\begin{lemma}\label{l:paraproducts II} The following paraproduct estimates hold uniformly for \(q\in Q_*\), \(h\in \R\), \(|\vk|,|\kappa|\geq 1\), and functions $f$ that are admissible in the sense of \eqref{admissible}$:$

i) Quartic paraproducts. Let \(m\in S_{\loc}(4)\). Then we have the estimates
\begin{align}
&\int \biggl|\int m\Bigl[\psi_\mu^6f,\tfrac{\psi_\mu^6 f}{2\vk - \p},\tfrac{2-\p}{2\vk - \p}(\psi_\mu^6 f),\tfrac{4-\p^2}{4\vk^2 - \p^2}(\psi_\mu^6 f)\Bigr]\,dx\biggr|\,e^{-\frac1{200}|h-\mu|}\,d\mu\label{m4 type a bound}\\
&\quad + \int\biggl|\int m\Bigl[\tfrac{\psi_\mu^6 f}{2\vk - \p},\tfrac{2-\p}{2\vk - \p}(\psi_\mu^6 f),\tfrac{2-\p}{2\vk - \p}(\psi_\mu^6 f),\tfrac{2-\p}{2\vk - \p}(\psi_\mu^6 f)\Bigr]\,dx\biggr|\,e^{-\frac1{200}|h-\mu|}\,d\mu\notag\\
&\quad + \int \biggl|\int m\Bigl[\psi_\mu^6f,\psi_\mu^6 f,\tfrac{2-\p}{4\vk^2 - \p^2}(\psi_\mu^6 f),\tfrac{4-\p^2}{4\vk^2 - \p^2}(\psi_\mu^6 f)\Bigr]\,dx\biggr|\,e^{-\frac1{200}|h-\mu|}\,d\mu\notag\\
&\quad + \int \biggl|\int m\Bigl[\psi_\mu^6f,\tfrac{\psi_\mu^6 f}{2\vk - \p},\tfrac{\psi_\mu^6 f}{2\vk - \p},\tfrac{(2 - \p)^3}{4\vk^2 - \p^2}(\psi_\mu^6 f)\Bigr]\,dx\biggr|\,e^{-\frac1{200}|h-\mu|}\,d\mu\notag\\
&\qquad\qquad\lesssim|\vk|^{-1}\tfrac{\kappa^2 + \vk^2}{\vk^2}\|q\|_{F_\kappa^{\frac12}(h)}^2\|q\|_{E_{2\sigma, \vk}^\sigma}^2,\notag
\end{align}\begin{align} 
&\int\biggl|\int m\Bigl[\psi_\mu^6 f,\psi_\mu^6 f,\tfrac{2-\p}{2\kappa - \p}(\psi_\mu^6f),\tfrac{2-\p}{2\kappa - \p}(\psi_\mu^6f)\Bigr]\,dx\biggr|\,e^{-\frac1{200}|h-\mu|}\,d\mu\label{m4 type a' bound}\\
&\quad + \int \biggl|\int m\Bigl[\psi_\mu^6 f,\psi_\mu^6 f,\psi_\mu^6 f,\tfrac{4-\p^2}{4\kappa^2 - \p^2}(\psi_\mu^6f)\Bigr]\,dx\biggr|\,e^{-\frac1{200}|h-\mu|}\,d\mu\notag\\
&\qquad\qquad\lesssim \|q\|_{F_\kappa^{\frac12}(h)}^2\|q\|_{E_\sigma^\sigma}^2,\notag
\end{align}\begin{align} 
\label{303030}
&\int\biggl|\int m\Bigl[\psi_\mu^6f,\tfrac{\psi_\mu^6 f}{2\vk-\p},\tfrac{2-\p}{2\vk-\p}(\psi_\mu^6f),\tfrac{4-\p^2}{4\kappa^2 - \p^2}(\psi_\mu^6 f)\Bigr]\,dx\biggr|\,e^{-\frac1{200}|h-\mu|}\,d\mu\\
&\quad + \int\biggl|\int m\Bigl[\psi_\mu^6f,\psi_\mu^6f,\tfrac{\psi_\mu^6 f}{2\vk-\p},\tfrac{(2-\p)^3}{(2\vk - \p)(4\kappa^2 - \p^2)}(\psi_\mu^6 f)\Bigr]\,dx\biggr|\,e^{-\frac1{200}|h-\mu|}\,d\mu\notag\\
&\quad +\int\biggl|\int m\Bigl[\psi_\mu^6f,\tfrac{2-\p}{2\kappa - \p}(\psi_\mu^6f),\tfrac{2-\p}{2\kappa - \p}(\psi_\mu^6f),\tfrac{2-\p}{4\vk^2 - \p^2}(\psi_\mu^6f)\Bigr]\,dx\biggr|\,e^{-\frac1{200}|h-\mu|}\,d\mu\notag\\
&\quad + \int\biggl|\int m\Bigl[\psi_\mu^6f,\psi_\mu^6f,\tfrac{4-\p^2}{4\kappa^2 - \p^2}(\psi_\mu^6f),\tfrac{2-\p}{4\vk^2 - \p^2}(\psi_\mu^6f)\Bigr]\,dx\biggr|\,e^{-\frac1{200}|h-\mu|}\,d\mu\notag\\
&\quad + \int \biggl|\int m\Bigl[\psi_\mu^6f,\psi_\mu^6f,\psi_\mu^6f,\tfrac{(2-\p)^3}{(4\kappa^2 - \p^2)(4\vk^2 - \p^2)}(\psi_\mu^6f)\Bigr]\,dx\biggr|\,e^{-\frac1{200}|h-\mu|}\,d\mu\notag\\
&\quad + \int \biggl|\kappa^4\int m\Bigl[\tfrac{\psi_\mu^6f}{4\kappa^2 - \p^2},\tfrac{\psi_\mu^6f}{4\kappa^2 - \p^2},\tfrac{2-\p}{4\kappa^2 - \p^2}({\psi_\mu^6f}),\tfrac{4-\p^2}{4\vk^2 - \p^2}({\psi_\mu^6f})\Bigr]\,dx\biggr|\,e^{-\frac1{200}|h-\mu|}\,d\mu\notag\\
&\qquad\qquad\lesssim |\vk|^{-1}\|q\|_{F_\kappa^{\frac12}(h)}^2\|q\|_{E_{2\sigma,\vk}^\sigma}^2.\notag
\end{align}

ii) Sextic paraproducts. If \(m\in S_{\loc}(6)\) then we have the estimates
\begin{align}\label{m6 diff a bound}
&\int \biggl|\int m\Bigl[\psi_\mu^4f,\psi_\mu^4f,\psi_\mu^4f,\tfrac{\psi_\mu^4 f}{2\vk - \p},\tfrac{\psi_\mu^4 f}{2\vk - \p},\tfrac{4-\p^2}{4\vk^2 - \p^2}(\psi_\mu^4 f)\Bigr]\,dx\biggr|\,e^{-\frac1{200}|h-\mu|}\,d\mu\\
&\quad + \int\biggl|\int m\Bigl[\psi_\mu^4f,\psi_\mu^4f,\tfrac{\psi_\mu^4 f}{2\vk - \p},\tfrac{\psi_\mu^4 f}{2\vk - \p},\tfrac{2-\p}{2\vk - \p}(\psi_\mu^4 f),\tfrac{2-\p}{2\vk - \p}(\psi_\mu^4 f)\Bigr]\,dx\biggr|\,e^{-\frac1{200}|h-\mu|}\,d\mu\notag\\
&\qquad\qquad \lesssim |\vk|^{-1}\tfrac{\kappa^2 + \vk^2}{\vk^2}\|q\|_{F_\kappa^{\frac12}(h)}^2\|q\|_{E_\sigma^\sigma}^2\|q\|_{E_{2\sigma,\vk}^\sigma}^2,\notag
\end{align}\begin{align} 
\label{m6 diff b bound}
&\int \biggl|\int m\Bigl[\psi_\mu^4 f,\psi_\mu^4 f,\psi_\mu^4 f,\tfrac{\psi_\mu^4 f}{2\kappa - \p},\tfrac{\psi_\mu^4 f}{2\kappa - \p},\tfrac{2-\p}{2\kappa - \p}(\psi_\mu^4f)\Bigr]\,dx\biggr|\,e^{-\frac1{200}|h-\mu|}\,d\mu\\
&\qquad\lesssim|\kappa|^{-1}\|q\|_{F_\kappa^{\frac12}(h)}^2\|q\|_{E_\sigma^\sigma}^4,\notag
\end{align}\begin{align} 
\label{m4 diff type II bound}
&\int \biggl|\int m\Bigl[\psi_\mu^4f,\psi_\mu^4f,\psi_\mu^4f,\tfrac{\psi_\mu^4 f}{2\vk-\p},\tfrac{\psi_\mu^4 f}{2\vk-\p},\tfrac{4-\p^2}{4\kappa^2 - \p^2}(\psi_\mu^4 f)\Bigr]\,dx\biggr|\,e^{-\frac1{200}|h-\mu|}\,d\mu\\
&\quad + \int \biggl|\int m\Bigl[\psi_\mu^4f,\psi_\mu^4f,\tfrac{\psi_\mu^4f}{2\vk - \p},\tfrac{\psi_\mu^4f}{2\vk - \p},\tfrac{2-\p}{2\kappa - \p}(\psi_\mu^4f),\tfrac{2-\p}{2\kappa - \p}(\psi_\mu^4f)\Bigr]\,dx\biggr|\,e^{-\frac1{200}|h-\mu|}\,d\mu\notag\\
&\quad + \int\biggl|\kappa^2\int m\Bigl[\psi_\mu^4f,\psi_\mu^4f,\psi_\mu^4f,\psi_\mu^4f,\tfrac{\psi_\mu^4f}{4\kappa^2 - \p^2},\tfrac{4-\p^2}{(4\vk^2 - \p^2)(4\kappa^2 - \p^2)}(\psi_\mu^4f)\Bigr]\,dx\biggr|\,e^{-\frac1{200}|h-\mu|}\,d\mu\notag\\
&\quad + \int\biggl|\vk^2\int m\Bigl[\psi_\mu^4f,\psi_\mu^4f,\psi_\mu^4f,\psi_\mu^4f,\tfrac{\psi_\mu^4f}{4\vk^2 - \p^2},\tfrac{4-\p^2}{(4\vk^2 - \p^2)(4\kappa^2 - \p^2)}(\psi_\mu^4f)\Bigr]\,dx\biggr|\,e^{-\frac1{200}|h-\mu|}\,d\mu\notag\\
&\quad + \int\biggl|\kappa^8\int m\Bigl[\tfrac{\psi_\mu^4f}{4\kappa^2 - \p^2},\tfrac{\psi_\mu^4f}{4\kappa^2 - \p^2},\tfrac{\psi_\mu^4f}{4\kappa^2 - \p^2},\tfrac{\psi_\mu^4f}{4\kappa^2 - \p^2},\tfrac{\psi_\mu^4f}{4\kappa^2 - \p^2},\tfrac{4-\p^2}{4\vk^2 - \p^2}(\psi_\mu^4f)\Bigr]\,dx\biggr|\,e^{-\frac1{200}|h-\mu|}\,d\mu\notag\\
&\quad\qquad\lesssim|\vk|^{-1}\|q\|_{F_\kappa^{\frac12}(h)}^2\|q\|_{E_{\sigma}^\sigma}^2\|q\|_{E_{2\sigma, \vk}^\sigma}^2,\notag
\end{align}\begin{align} 
\label{200x200}
&\int\biggl|\kappa\int m\Bigl[\psi_\mu^4f,\psi_\mu^4f,\tfrac{\psi_\mu^4f}{2\kappa - \p},\tfrac{\psi_\mu^4f}{2\kappa - \p},\tfrac{2-\p}{2\kappa - \p}(\psi_\mu^4f),\tfrac{2-\p}{4\vk^2 - \p^2}(\psi_\mu^4f)\Bigr]\,dx\biggr|\,e^{-\frac1{200}|h-\mu|}\,d\mu\\
&\quad + \int\Biggl|\kappa\int m\Bigl[\psi_\mu^4f,\tfrac{\psi_\mu^4 f}{2\vk - \p},\tfrac{2-\p}{2\vk - \p}(\psi_\mu^4 f),\tfrac{\psi_\mu^4 f}{2\kappa - \p},\tfrac{\psi_\mu^4 f}{2\kappa - \p},\tfrac{2-\p}{2\kappa - \p}(\psi_\mu^4 f)\Bigr]\,dx\Biggr|\,e^{-\frac1{200}|h-\mu|}\,d\mu\notag\\
&\qquad\qquad\lesssim|\vk|^{-1}\|q\|_{F_\kappa^{\frac12}(h)}^2\|q\|_{E_\sigma^\sigma}^2\|q\|_{E_{2\sigma,\vk}^\sigma}^2.\notag
\end{align}

iii) Octic paraproducts. If \(m\in S_{\loc}(8)\) then we have the estimates
\begin{align}\label{19-5}
&\int\biggl| \vk\int m\Bigl[\psi_\mu^3f,\psi_\mu^3f,\psi_\mu^3f,\tfrac{\psi_\mu^3 f}{2\vk - \p},\tfrac{\psi_\mu^3 f}{2\vk - \p},\tfrac{\psi_\mu^3 f}{2\vk - \p},\tfrac{\psi_\mu^3 f}{2\vk - \p},\tfrac{2-\p}{2\vk - \p}(\psi_\mu^3 f)\Bigr]\,dx\biggr|\,e^{-\frac1{200}|h-\mu|}\,d\mu\\
&\qquad\lesssim |\vk|^{-1}\tfrac{\kappa^2 + \vk^2}{\vk^2}\|q\|_{F_\kappa^{\frac12}(h)}^2\|q\|_{E_\sigma^\sigma}^4\|q\|_{E_{2\sigma,\vk}^\sigma}^2,\notag
\end{align}\begin{align} 
\label{15-2}
&\int\Biggl|\kappa\!\int m\Bigl[\psi_\mu^3 f,\psi_\mu^3 f,\psi_\mu^3 f,\tfrac{\psi_\mu^3 f}{2\vk - \p},\tfrac{\psi_\mu^3 f}{2\vk - \p},\tfrac{\psi_\mu^3 f}{2\kappa - \p},\tfrac{\psi_\mu^3 f}{2\kappa - \p},\tfrac{2-\p}{2\kappa - \p}(\psi_\mu^3 f)\Bigr]\,dx\Biggr|\,e^{-\frac1{200}|h-\mu|}\,d\mu\\
&\quad+\int\Biggl|\kappa^2\!\int m\Bigl[\psi_\mu^3f,\psi_\mu^3f,\tfrac{\psi_\mu^3 f}{2\vk - \p},\tfrac{2-\p}{2\vk - \p}(\psi_\mu^3 f),\tfrac{\psi_\mu^3 f}{2\kappa - \p},\tfrac{\psi_\mu^3 f}{2\kappa - \p},\tfrac{\psi_\mu^3 f}{2\kappa - \p},\tfrac{\psi_\mu^3 f}{2\kappa - \p}\Bigr]\,dx\Biggr|\,e^{-\frac1{200}|h-\mu|}\,d\mu\notag\\
&\quad + \int \biggl|\kappa^2\!\int m\Bigl[\psi_\mu^3f,\psi_\mu^3f,\psi_\mu^3f,\tfrac{\psi_\mu^3f}{2\kappa - \p},\tfrac{\psi_\mu^3f}{2\kappa - \p},\tfrac{\psi_\mu^3f}{2\kappa - \p},\tfrac{\psi_\mu^3f}{2\kappa - \p},\tfrac{2-\p}{4\vk^2 - \p^2}(\psi_\mu^3f)\Bigr]\,dx\biggr|\,e^{-\frac1{200}|h-\mu|}\,d\mu\notag\\
&\qquad\qquad\lesssim|\vk|^{-1}\|q\|_{F_\kappa^{\frac12}(h)}^2\|q\|_{E_\sigma^\sigma}^4\|q\|_{E_{2\sigma,\vk}^\sigma}^2.\notag
\end{align}

iv) Decic paraproducts. If \(m\in S_{\loc}(10)\) then we have the estimates
\begin{align}\label{19-1}
&\int\biggl|\vk^2\!\int m\Bigl[\underbrace{\psi_\mu^2f,\dots,\psi_\mu^2 f}_4,\underbrace{\tfrac{\psi_\mu^2 f}{2\vk - \p},\dots,\tfrac{\psi_\mu^2 f}{2\vk - \p}}_6\Bigr]\,\psi_\mu^4\,dx\biggr|\,e^{-\frac1{200}|h-\mu|}\,d\mu\\
&\qquad\lesssim |\vk|^{-1}\tfrac{\kappa^2 + \vk^2}{\vk^2}\|q\|_{F_\kappa^{\frac12}(h)}^2\|q\|_{E_\sigma^\sigma}^6\|q\|_{E_{2\sigma,\vk}^\sigma}^2,\notag
\end{align}\begin{align} 
\label{S10}
&\int\Biggl|\kappa^2\!\!\int\! m\Bigl[\underbrace{\psi_\mu^2 f,\dots,\psi_\mu^2 f}_4,\tfrac{\psi_\mu^2 f}{2\vk - \p},\tfrac{\psi_\mu^2 f}{2\vk - \p},\underbrace{\tfrac{\psi_\mu^2 f}{2\kappa - \p},\dots,\tfrac{\psi_\mu^2 f}{2\kappa - \p}}_4\Bigr]\,\psi_\mu^4\,dx\Biggr|\,e^{-\frac1{200}|h-\mu|}\,d\mu\\
&\qquad\lesssim |\vk|^{-1}\|q\|_{F_\kappa^{\frac12}(h)}^2\|q\|_{E_\sigma^\sigma}^6\|q\|_{E_{2\sigma,\vk}^\sigma}^2.\notag
\end{align}

\end{lemma}
\begin{proof}

For all of the ensuing estimates, we follow the argument of Lemma~\ref{l:paraproducts I}:
\begin{enumerate}
\item Decompose into Littlewood--Paley pieces.
\item If derivatives fall at high frequency, integrate by parts using \eqref{paraLeibniz}.
\item Estimate the two highest frequency terms in \(L^2\) and the remaining terms in \(L^\infty\) using \eqref{paraHolder}.
\item Estimate the low frequency terms using Lemma~\ref{L:Fmu}.
\item Bound the highest two frequency terms using \eqref{LS loc}.
\end{enumerate}

We illustrate this in detail with the first term on \(\LHS{m4 type a bound}\). Decomposing into Littlewood--Paley pieces, we consider
\[
\sum_{M_1,M_2,M_3,M_4}\int m\Bigl[P_{M_1}(\psi_\mu^6f),\tfrac{P_{M_2}(\psi_\mu^6 f)}{2\vk - \p},\tfrac{2-\p}{2\vk - \p}P_{M_3}(\psi_\mu^6 f),\tfrac{4-\p^2}{4\vk^2 - \p^2}P_{M_4}(\psi_\mu^6 f)\Bigr]\,dx.
\]
We now decompose the sum into three parts.

The first summand is where \(M_3,M_4<\max\{M_j\}\). Here, we apply \eqref{paraHolder}, Bernstein's inequality, and Lemma~\ref{L:Fmu} to estimate
\begin{align*}
&\sum_{\substack{M_1,M_2,M_3,M_4\\M_3,M_4<\max\{M_j\}}}\left|\int m\Bigl[P_{M_1}(\psi_\mu^6f),\tfrac{P_{M_2}(\psi_\mu^6 f)}{2\vk - \p},\tfrac{2-\p}{2\vk - \p}P_{M_3}(\psi_\mu^6 f),\tfrac{4-\p^2}{4\vk^2 - \p^2}P_{M_4}(\psi_\mu^6 f)\Bigr]\,dx\right|\\
&\quad\lesssim \sum_{N_1\geq \dots \geq N_4}\tfrac{(1+N_2)^2(1+N_3)}{(|\vk| + N_2)^2(|\vk| + N_3)(|\vk|+N_4)}\prod_{j=1}^2\|P_{N_j}(\psi_\mu^6 f)\|_{L^2}\prod_{j=3}^4\|P_{N_j}(\psi_\mu^6 f)\|_{L^\infty}\\
&\quad\lesssim \sum_{N_1\geq N_2} \tfrac{|\vk|^{-1}N_2^{1-2\sigma}(1+N_2)^{3}}{(|\vk| + N_2)^{3-2\sigma}}\|P_{N_1}(\psi_\mu^6 f)\|_{L^2}\|P_{N_2}(\psi_\mu^6 f)\|_{L^2}\|q\|_{E_{2\sigma,\vk}^\sigma}^2\\
&\quad\lesssim |\vk|^{-1}\Bigl\|\tfrac{\psi_\mu^6 f}{\sqrt{4\vk^2 - \p^2}}\Bigr\|_{H^{\frac32}}^2\|q\|_{E_{2\sigma,\vk}^\sigma}^2\lesssim |\vk|^{-1}\tfrac{\kappa^2 + \vk^2}{\vk^2}\Bigl\|\tfrac{\psi_\mu^6 f}{\sqrt{4\kappa^2 - \p^2}}\Bigr\|_{H^{\frac32}}^2\|q\|_{E_{2\sigma,\vk}^\sigma}^2,
\end{align*}
where we note that the \(M_j\) have been permuted in the first inequality to account fo the largest contribution. This is acceptable after integrating with respect to \(e^{-\frac1{200}|h-\mu|}\,d\mu\) and applying \eqref{LS loc}.

The second summand is where \(M_3\leq M_4 = \max\{M_j\}\). Here, we use \eqref{paraLeibniz} to integrate by parts and then proceed as for the first summand to obtain
\begin{align*}
&\sum_{\substack{M_1,M_2,M_3,M_4\\M_3\leq M_4=\max\{M_j\}}}\!\left|\int m\Bigl[P_{M_1}(\psi_\mu^6f),\tfrac{P_{M_2}(\psi_\mu^6 f)}{2\vk - \p},\tfrac{2-\p}{2\vk - \p}P_{M_3}(\psi_\mu^6 f),\tfrac{4-\p^2}{4\vk^2 - \p^2}P_{M_4}(\psi_\mu^6 f)\Bigr]\,dx\right|\\
&\quad\lesssim \sum_{N_1\geq \dots \geq N_4}\tfrac{(1+N_2)^3}{(|\vk| + N_1)^2(|\vk| + N_2)(|\vk| + N_4)}\prod_{j=1}^2\|P_{N_j}(\psi_\mu^6 f)\|_{L^2}\prod_{j=3}^4\|P_{N_j}(\psi_\mu^6 f)\|_{L^\infty}\\
&\quad\lesssim \sum_{N_1\geq N_2} \tfrac{N_2^{1-2\sigma}(1+N_2)^{3}}{|\vk|^{\frac12+\sigma}(|\vk| + N_2)^{\frac 72-3\sigma}}\|P_{N_1}(\psi_\mu^6 f)\|_{L^2}\|P_{N_2}(\psi_\mu^6 f)\|_{L^2}\|q\|_{E_{2\sigma,\vk}^\sigma}^2\\
&\quad\lesssim |\vk|^{-1}\Bigl\|\tfrac{\psi_\mu^6 f}{\sqrt{4\vk^2 - \p^2}}\Bigr\|_{H^{\frac32}}^2\|q\|_{E_{2\sigma,\vk}^\sigma}^2\lesssim |\vk|^{-1}\tfrac{\kappa^2 + \vk^2}{\vk^2}\Bigl\|\tfrac{\psi_\mu^6 f}{\sqrt{4\kappa^2 - \p^2}}\Bigr\|_{H^{\frac32}}^2\|q\|_{E_{2\sigma,\vk}^\sigma}^2,
\end{align*}
which is again acceptable.

The final summand, where \(M_4<M_3=\max\{M_j\}\), is estimated in a similar way, using \eqref{paraLeibniz} to obtain a contribution of
\begin{align*}
&\sum_{\substack{M_1,M_2,M_3,M_4\\M_4< M_3=\max\{M_j\}}}\!\left|\int m\Bigl[P_{M_1}(\psi_\mu^6f),\tfrac{P_{M_2}(\psi_\mu^6 f)}{2\vk - \p},\tfrac{2-\p}{2\vk - \p}P_{M_3}(\psi_\mu^6 f),\tfrac{4-\p^2}{4\vk^2 - \p^2}P_{M_4}(\psi_\mu^6 f)\Bigr]\,dx\right|\\
&\qquad\lesssim \sum_{N_1\geq \dots \geq N_4}\tfrac{(1+N_2)^3}{(|\vk| + N_1)(|\vk| + N_2)^2(|\vk| + N_4)}\prod_{j=1}^2\|P_{N_j}(\psi_\mu^6 f)\|_{L^2}\prod_{j=3}^4\|P_{N_j}(\psi_\mu^6 f)\|_{L^\infty},
\end{align*}
which is again acceptable, as before.

The remaining terms on \(\LHS{m4 type a bound}\) are estimated similarly. In each case, their contribution is bounded by
\begin{align*}
\sum_{N_1\geq \dots\geq N_4}\tfrac{(1+N_2)^3}{(|\vk| + N_2)^2(|\vk| + N_3)(|\vk| + N_4)}\prod_{j=1}^2\|P_{N_j}(\psi_\mu^6 f)\|_{L^2}\prod_{j=3}^4\|P_{N_j}(\psi_\mu^6 f)\|_{L^\infty},
\end{align*}
which is acceptable after summation.

For \eqref{m4 type a' bound}, we proceed similarly, to obtain a bound of
\begin{align*}
&\sum_{N_1\geq \dots \geq N_4}\tfrac{(1+N_2)^2}{(|\kappa| + N_2)^2}\prod_{j=1}^2\|P_{N_j}(\psi_\mu^6 f)\|_{L^2}\prod_{j=3}^4\|P_{N_j}(\psi_\mu^6 f)\|_{L^\infty}\\
&\qquad\lesssim \sum_{N_1\geq N_2} \tfrac{N_2^{1-2\sigma}(1+N_2)^{2+2\sigma}}{(|\kappa| + N_2)^2}\|P_{N_1}(\psi_\mu^6 f)\|_{L^2}\|P_{N_2}(\psi_\mu^6 f)\|_{L^2}\|q\|_{E_\sigma^\sigma}^2\\
&\qquad\lesssim \Bigl\|\tfrac{\psi_\mu^6 f}{\sqrt{4\kappa^2 - \p^2}}\Bigr\|_{H^{\frac32}}^2\|q\|_{E_\sigma^\sigma}^2,
\end{align*}
which is acceptable.

For the estimate \eqref{303030} we argue as before to obtain a bound of
\begin{align*}
&\sum_{N_1\geq \dots\geq N_4}\tfrac{(1+N_2)^3}{(|\kappa|+N_2)^2(|\vk| + N_3)(|\vk| + N_4)}\prod_{j=1}^2\|P_{N_j}(\psi_\mu^6 f)\|_{L^2}\prod_{j=3}^4\|P_{N_j}(\psi_\mu^6 f)\|_{L^\infty}\\
&\qquad\lesssim \sum_{N_1\geq N_2} \tfrac{|\vk|^{-1}N_2^{1-2\sigma}(1+N_2)^3}{(|\kappa| + N_2)^2(|\vk| + N_2)^{1-2\sigma}}\|P_{N_1}(\psi_\mu^6 f)\|_{L^2}\|P_{N_2}(\psi_\mu^6 f)\|_{L^2}\|q\|_{E_{2\sigma,\vk}^\sigma}^2\\
&\qquad\lesssim |\vk|^{-1}\Bigl\|\tfrac{\psi_\mu^6 f}{\sqrt{4\kappa^2 - \p^2}}\Bigr\|_{H^{\frac32}}^2\|q\|_{E_{2\sigma,\vk}^\sigma}^2,
\end{align*}
which is once again acceptable.

The sextic terms are estimated similarly. For \eqref{m6 diff a bound}, after integrating by parts we obtain a bound of
\begin{align*}
&\sum_{N_1\geq\dots\geq N_6}\tfrac{(1+N_2)^2}{(|\vk|+N_2)^2(|\vk|+N_5)(|\vk|+N_6)}\prod_{j=1}^2\|P_{N_j}(\psi_\mu^4 f)\|_{L^2}\prod_{j=3}^6\|P_{N_j}(\psi_\mu^4 f)\|_{L^\infty}\\
&\qquad\lesssim \sum_{N_1\geq N_2} \tfrac{|\vk|^{-1}N_2^{2-4\sigma}(1+N_2)^{2+2\sigma}}{(|\vk| + N_2)^{3-2\sigma}}\|P_{N_1}(\psi_\mu^4 f)\|_{L^2}\|P_{N_2}(\psi_\mu^4 f)\|_{L^2}\|q\|_{E_\sigma^\sigma}^2\|q\|_{E_{2\sigma,\vk}^\sigma}^2\\
&\qquad\lesssim |\vk|^{-1}\tfrac{\kappa^2 + \vk^2}{\vk^2}\Bigl\|\tfrac{\psi_\mu^4 f}{\sqrt{4\kappa^2 - \p^2}}\Bigr\|_{H^{\frac32}}^2\|q\|_{E_\sigma^\sigma}^2\|q\|_{E_{2\sigma,\vk}^\sigma}^2,
\end{align*}
which is acceptable.

Similarly, we may bound $\LHS{m6 diff b bound}$ by
\begin{align*}
&\sum_{N_1\geq\dots\geq N_6}\tfrac{(1+N_2)}{(|\kappa|+N_2)(|\kappa|+N_5)(|\kappa|+N_6)}\prod_{j=1}^2\|P_{N_j}(\psi_\mu^4 f)\|_{L^2}\prod_{j=3}^6\|P_{N_j}(\psi_\mu^4 f)\|_{L^\infty}\\
&\qquad\lesssim \sum_{N_1\geq N_2} \tfrac{N_2^{2-4\sigma}(1+N_2)^{1+4\sigma}}{|\kappa|(|\kappa| + N_2)^2}\|P_{N_1}(\psi_\mu^4 f)\|_{L^2}\|P_{N_2}(\psi_\mu^4 f)\|_{L^2}\|q\|_{E_\sigma^\sigma}^4\\
&\qquad\lesssim |\kappa|^{-1}\Bigl\|\tfrac{\psi_\mu^4 f}{\sqrt{4\kappa^2 - \p^2}}\Bigr\|_{H^{\frac32}}^2\|q\|_{E_\sigma^\sigma}^4,
\end{align*}
which is acceptable.

For \eqref{m4 diff type II bound}, after integrating by parts we may bound each term by
\begin{align*}
&\sum_{N_1\geq \dots \geq N_6}\tfrac{(1+N_2)^2}{(|\kappa| + N_2)^2(|\vk|+N_5)(|\vk| + N_6)}\prod_{j=1}^2\|P_{N_j}(\psi_\mu^4 f)\|_{L^2}\prod_{j=3}^6\|P_{N_j}(\psi_\mu^4 f)\|_{L^\infty}\\
&\qquad\lesssim \sum_{N_1\geq N_2}\tfrac{|\vk|^{-1}N_2^{2-4\sigma}(1+N_2)^{2+2\sigma}}{(|\kappa|+N_2)^2(|\vk| + N_2)^{1-2\sigma}} \|P_{N_1}(\psi_\mu^4f)\|_{L^2}\|P_{N_2}(\psi_\mu^4 f)\|_{L^2}\|q\|_{E_\sigma^\sigma}^2\|q\|_{E_{2\sigma,\vk}^\sigma}^2\\
&\qquad\lesssim |\vk|^{-1}\Bigl\|\tfrac{\psi_\mu^6 f}{\sqrt{4\kappa^2 - \p^2}}\Bigr\|_{H^{\frac32}}^2\|q\|_{E_\sigma^\sigma}^2\|q\|_{E_{2\sigma,\vk}^\sigma}^2,
\end{align*}
which is again acceptable.

Turning to \eqref{200x200}, our basic technique gives us a bound of
\begin{align*}
\sum_{\tau\in \mathfrak S}\sum_{N_1\geq \dots \geq N_6}&\tfrac{|\kappa|(1+N_2)^2}{(|\kappa| + N_2)(|\kappa| + N_{\tau(3)})(|\kappa| + N_{\tau(4)})(|\vk| + N_{\tau(5)})(|\vk| + N_{\tau(6)})}\\
&\qquad\qquad\qquad\prod_{j=1}^2\|P_{N_j}(\psi_\mu^4 f)\|_{L^2}\prod_{j=3}^6\|P_{N_j}(\psi_\mu^4 f)\|_{L^\infty},
\end{align*}
where \(\mathfrak S\) is the set of permutations of \(\{3,4,5,6\}\). Estimating the \(N_{\tau(3)},N_{\tau(4)}\) terms in \(E_\sigma^\sigma\) using \eqref{Fmu-E1} and the \(N_{\tau(5)},N_{\tau(6)}\) terms in \(E_{2\sigma,\vk}^\sigma\) using \eqref{Fmu-E2}, we obtain a bound of
\begin{align*}
\sum_{N_1\geq N_2}\tfrac{|\vk|^{-1}N_2^{2-4\sigma}(1+N_2)^{2+2\sigma}}{(|\kappa| + N_2)^2(|\vk|+N_2)^{1-2\sigma}}\prod_{j=1}^2\|P_{N_j}(\psi_\mu^4 f)\|_{L^2}\|q\|_{E_\sigma^\sigma}^2\|q\|_{E_{2\sigma,\vk}^\sigma}^2,
\end{align*}
which is acceptable.

Turning next to the octic estimates, we may bound \(\LHS{19-5}\) by
\begin{align*}
&\sum_{N_1\geq \dots\geq N_8} \tfrac{|\vk|(1+N_2)}{(|\vk|+N_2)}\prod_{j=5}^8\tfrac1{|\vk| + N_j}\prod_{j=1}^2\|P_{N_j}(\psi_\mu^3 f)\|_{L^2}\prod_{j=3}^8\|P_{N_j}(\psi_\mu^3 f)\|_{L^\infty}\\
&\qquad\lesssim \sum_{N_1\geq N_2} \tfrac{|\vk|^{-1}N_2^{3-6\sigma}(1+N_2)^{1+4\sigma}}{(|\vk| + N_2)^{3-2\sigma}}\prod_{j=1}^2\|P_{N_j}(\psi_\mu^3 f)\|_{L^2}\|q\|_{E_\sigma^\sigma}^4\|q\|_{E_{2\sigma,\vk}^\sigma}^2\\
&\qquad\lesssim |\vk|^{-1}\tfrac{\kappa^2 + \vk^2}{\vk^2}\Bigl\|\tfrac{\psi_\mu^3 f}{\sqrt{4\kappa^2 - \p^2}}\Bigr\|_{H^{\frac32}}^2\|q\|_{E_\sigma^\sigma}^4\|q\|_{E_{2\sigma,\vk}^\sigma}^2,
\end{align*}
which is acceptable.

For \eqref{15-2}, we argue as in \eqref{200x200} so that after estimating the low frequency terms in \(E_\sigma^\sigma\) or \(E_{2\sigma,\vk}^\sigma\) (depending on the associated denominator) we obtain a bound of
\begin{align*}
\sum_{N_1\geq N_2}\!\! \Bigl\{\tfrac{|\vk|^{-1}N_2^{3-6\sigma}(1+N_2)^{1+4\sigma}}{(|\kappa| + N_2)^2(|\vk| + N_2)^{1-2\sigma}} +\tfrac{|\vk|^{-\frac12-\sigma}N_2^{3-2\sigma}(1+N_2)^{1+\sigma}}{(|\kappa| + N_2)^2(|\vk| + N_2)^{\frac32-2\sigma}}\Bigr\}\prod_{j=1}^2\|P_{N_j}(\psi_\mu^3 f)\|_{L^2}\|q\|_{E_\sigma^\sigma}^4\|q\|_{E_{2\sigma,\vk}^\sigma}^2
\end{align*}
which is acceptable.

Again applying our basic technique to \eqref{19-1}, we obtain a bound of
\begin{align*}
&\sum_{N_1\geq \dots\geq N_{10}} \!\!|\vk|^2\,\prod_{j=5}^{10}\tfrac{1}{|\vk| + N_j}\prod_{j=1}^2\|P_{N_j}(\psi_\mu^2 f)\|_{L^2}\prod_{j=3}^{10}\|P_{N_j}(\psi_\mu^2 f)\|_{L^\infty}\\
&\qquad\lesssim \sum_{N_1\geq N_2} \tfrac{|\vk|^{-1}N_2^{4-8\sigma}(1+N_2)^{6\sigma}}{(|\vk| + N_2)^{3-2\sigma}}\prod_{j=1}^2\|P_{N_j}(\psi_\mu^2 f)\|_{L^2}\|q\|_{E_\sigma^\sigma}^6\|q\|_{E_{2\sigma,\vk}^\sigma}^2\\
&\qquad\lesssim |\vk|^{-1}\tfrac{\kappa^2 + \vk^2}{\vk^2}\Bigl\|\tfrac{\psi_\mu^2 f}{\sqrt{4\kappa^2 - \p^2}}\Bigr\|_{H^{\frac32}}^2\|q\|_{E_\sigma^\sigma}^6\|q\|_{E_{2\sigma,\vk}^\sigma}^2,
\end{align*}
which is acceptable.

Finally, \eqref{S10} again follows the argument of \eqref{200x200}, \eqref{15-2}, estimating the low frequency terms in \(E_\sigma^\sigma\) or \(E_{2\sigma,\vk}^\sigma\) to obtain an acceptable bound of
\begin{align*}
\sum_{N_1\geq N_2} \tfrac{|\vk|^{-1}N_2^{4-8\sigma}(1+N_2)^{6\sigma}}{(|\kappa| + N_2)^2(|\vk| + N_2)^{1-2\sigma}}\prod_{j=1}^2\|P_{N_j}(\psi_\mu^2 f)\|_{L^2}\|q\|_{E_\sigma^\sigma}^6\|q\|_{E_{2\sigma,\vk}^\sigma}^2.
\end{align*}

This completes the proof of the lemma.
\end{proof}

Combining Propositions~\ref{p:Asymptotics} and Lemma~\ref{l:paraproducts II}, we obtain the following:

\begin{lemma}\label{L:rewriting}
Let \(q\in Q_*\), \(h\in \R\), \(\kappa\geq 1\), and \(\vk \in I_\kappa = [1,\frac\kappa 2]\cup[2\kappa,\infty)\). Then we have the estimate
\begin{align}
\int \biggl|\Im \int j_{\diff}\sbrack{\geq 4}\,\psi_\mu^{24}\,dx\biggr|\,&e^{-\frac1{200}|h-\mu|}\,d\mu\label{j diff geq 4 bound}\\
&\qquad\lesssim \|q\|_{F_\kappa^{\frac12}(h)}^2\Bigl[\vk^{- 1}\|q\|_{E_{2\sigma,\vk}^\sigma}^2 + \tfrac\kappa{\kappa^2 + \vk^2}\|q\|_{E_\sigma^\sigma}^2\Bigr].\notag
\end{align}
\end{lemma}

\begin{proof}
From \eqref{linear}, we have
\begin{align*}
\tfrac1\kappa q - \Bigl[\tfrac1{\sqrt\kappa} g_{12}\sbrack 1(\kappa) - \tfrac1{\sqrt{-\kappa}}g_{12}\sbrack 1(-\kappa) \Bigr] &= - \tfrac1\kappa u(\kappa)'',\\
\tfrac1{2\kappa}q' - \Bigl[ \sqrt\kappa g_{12}\sbrack 1(\kappa) - \sqrt{-\kappa}g_{12}\sbrack 1(-\kappa) \Bigr] &=  - \tfrac1{2\kappa}u(\kappa)'''.
\end{align*}
We may then use \eqref{j diff} and \eqref{sym} to decompose
\begin{align}
\Im j_{\diff}\sbrack{\geq 4} = \Im \Biggl\{&\notag\\
& -\tfrac{2\vk^2}{\kappa^2 - \vk^2}\Bigl[\sqrt\vk\bigl(\tfrac{g_{21}}{2+\gamma}\bigr)\sbrack{\geq 3}(\vk) + \sqrt{-\vk}\bigl(\tfrac{g_{21}}{2+\gamma}\bigr)\sbrack{\geq 3}(-\vk)\Bigr] q\label{l I}\\
& -\tfrac{\vk^2}{\kappa^2 - \vk^2} \Bigl[\tfrac1{\sqrt\vk}\bigl(\tfrac{g_{21}}{2+\gamma}\bigr)\sbrack{\geq 3}(\vk) + \tfrac1{\sqrt{-\vk}}\bigl(\tfrac{g_{21}}{2+\gamma}\bigr)\sbrack{\geq 3}(-\vk)\Bigr] q'\label{l II}\\
& - \tfrac{i\vk^2}{\kappa^2 - \vk^2}\Bigl[\tfrac1{\sqrt\vk}\bigl(\tfrac{g_{21}}{2+\gamma}\bigr)(\vk) + \tfrac1{\sqrt{-\vk}}\bigl(\tfrac{g_{21}}{2+\gamma}\bigr)(-\vk)\Bigr]|q|^2q\label{l III}\\
&- \tfrac{2\kappa^2}{\kappa^2 - \vk^2}\Bigl[\sqrt\vk\bigl(\tfrac{g_{21}}{2+\gamma}\bigr)\sbrack{\geq 3}(\vk) + \sqrt{-\vk}\bigl(\tfrac{g_{21}}{2+\gamma}\bigr)\sbrack{\geq 3}(-\vk)\Bigr]u(\kappa)''\label{l IV}\\
& -\tfrac{\kappa^2}{\kappa^2 - \vk^2} \Bigl[\tfrac1{\sqrt\vk}\bigl(\tfrac{g_{21}}{2+\gamma}\bigr)\sbrack{\geq 3}(\vk) + \tfrac1{\sqrt{-\vk}}\bigl(\tfrac{g_{21}}{2+\gamma}\bigr)\sbrack{\geq 3}(-\vk)\Bigr]u(\kappa)'''\label{l V}\\
&-\tfrac{2\kappa^3}{\kappa^2 - \vk^2}\Bigl[\sqrt\vk\bigl(\tfrac{g_{21}}{2+\gamma}\bigr)(\vk) + \sqrt{-\vk}\bigl(\tfrac{g_{21}}{2+\gamma}\bigr)(-\vk)\Bigr]\label{l VI}\\
&\qquad\qquad\qquad\qquad\quad\qquad\qquad\qquad\times\Bigl[\tfrac1{\sqrt\kappa}g_{12}\sbrack{\geq 3}(\kappa) - \tfrac1{\sqrt{-\kappa}}g_{12}\sbrack{\geq 3}(-\kappa)\Bigr]\notag\\
& -\tfrac{2\kappa^3}{\kappa^2 - \vk^2} \Bigl[\tfrac1{\sqrt\vk}\bigl(\tfrac{g_{21}}{2+\gamma}\bigr)(\vk) + \tfrac1{\sqrt{-\vk}}\bigl(\tfrac{g_{21}}{2+\gamma}\bigr)(-\vk)\Bigr]\label{l VII}\\
&\qquad\qquad\qquad\qquad\quad\qquad\times\Bigl[\sqrt\kappa g_{12}\sbrack{\geq 3}(\kappa) - \sqrt{-\kappa}g_{12}\sbrack{\geq 3}(-\kappa)-\tfrac i{2\kappa}|q|^2q\Bigr]\notag\\
&-\tfrac{\kappa^3}{\kappa^2 - \vk^2}\Bigl[\gamma\sbrack{\geq 4}(\kappa) - \gamma\sbrack{\geq 4}(-\kappa)\Bigr]\label{l VIII}\\
& \Biggr\}.\notag
\end{align}

We now proceed to use Proposition~\ref{p:Asymptotics} and Lemma~\ref{l:paraproducts II} to remove the leading order terms from each line as follows:

For \eqref{l I}, we first use \eqref{paraproduct for frac geq} to write
\[
\int \sqrt\vk\bigl(\tfrac{g_{21}}{2+\gamma}\bigr)\sbrack{\geq 9}(\vk)\,q\,\psi_\mu^{24}\,dx = \vk^2\int m\Bigl[\underbrace{\psi_\mu^2f,\dots,\psi_\mu^2 f}_4,\underbrace{\tfrac{\psi_\mu^2 f}{2\vk - \p},\dots,\tfrac{\psi_\mu^2 f}{2\vk - \p}}_6\Bigr]\,\psi_\mu^4\,dx,
\]
where \(m\in S_{\loc}(10)\), to which we apply \eqref{19-1}. Next, we use \eqref{g21 frac 7 expansion} to write
\begin{align*}
&\int \Bigl[\sqrt\vk\bigl(\tfrac{g_{21}}{2+\gamma}\bigr)\sbrack{7}(\vk) + \sqrt{-\vk}\bigl(\tfrac{g_{21}}{2+\gamma}\bigr)\sbrack{7}(-\vk)\Bigr]\,q\,\psi_\mu^{24}\,dx\\
&\qquad = \vk\int m\Bigl[\psi_\mu^3f,\psi_\mu^3f,\psi_\mu^3f,\tfrac{\psi_\mu^3 f}{2\vk - \p},\tfrac{\psi_\mu^3 f}{2\vk - \p},\tfrac{\psi_\mu^3 f}{2\vk - \p},\tfrac{\psi_\mu^3 f}{2\vk - \p},\tfrac{2-\p}{2\vk - \p}(\psi_\mu^3 f)\Bigr]\,dx,
\end{align*}
where \(m\in S_{\loc}(8)\), which we bound using \eqref{19-5}. Similarly, we use \eqref{g21 frac 5 expansion} to write
\begin{align*}
&\int \Bigl[\sqrt\vk\bigl(\tfrac{g_{21}}{2+\gamma}\bigr)\sbrack{5}(\vk) + \sqrt{-\vk}\bigl(\tfrac{g_{21}}{2+\gamma}\bigr)\sbrack{5}(-\vk) - 64i\vk^8|u(\vk)|^4\bar u(\vk)\Bigr]\,q\,\psi_\mu^{24}\,dx\\
&\qquad = \int m\Bigl[\psi_\mu^4f,\psi_\mu^4f,\psi_\mu^4f,\tfrac{\psi_\mu^4 f}{2\vk - \p},\tfrac{\psi_\mu^4 f}{2\vk - \p},\tfrac{4-\p^2}{4\vk^2 - \p^2}(\psi_\mu^4 f)\Bigr]\,dx\\
&\qquad\quad + \int m\Bigl[\psi_\mu^4f,\psi_\mu^4f,\tfrac{\psi_\mu^4 f}{2\vk - \p},\tfrac{\psi_\mu^4 f}{2\vk - \p},\tfrac{2-\p}{2\vk - \p}(\psi_\mu^4 f),\tfrac{2-\p}{2\vk - \p}(\psi_\mu^4 f)\Bigr]\,dx,
\end{align*}
where each \(m\in S_{\loc}(6)\), which can be estimated using \eqref{m6 diff a bound}. Finally, we use the second expression on $\RHS{g21 frac 3 expansion}$ to obtain
\begin{align*}
&\int \Bigl[\sqrt\vk\bigl(\tfrac{g_{21}}{2+\gamma}\bigr)\sbrack{3}(\vk) + \sqrt{-\vk}\bigl(\tfrac{g_{21}}{2+\gamma}\bigr)\sbrack{3}(-\vk)\\
&\qquad\qquad\qquad\qquad\qquad\qquad+16\vk^4|u(\vk)|^2\bar u(\vk)'+4\vk^4\bar u(\vk)^2u(\vk)'\Bigr]\,q\,\psi_\mu^{24}\,dx\\
&\qquad = \int m\Bigl[\psi_\mu^6f,\tfrac{\psi_\mu^6 f}{2\vk - \p},\tfrac{2-\p}{2\vk - \p}(\psi_\mu^6 f),\tfrac{4-\p^2}{4\vk^2 - \p^2}(\psi_\mu^6 f)\Bigr]\,dx\\
&\qquad\quad + \int m\Bigl[\tfrac{\psi_\mu^6 f}{2\vk - \p},\tfrac{2-\p}{2\vk - \p}(\psi_\mu^6 f),\tfrac{2-\p}{2\vk - \p}(\psi_\mu^6 f),\tfrac{2-\p}{2\vk - \p}(\psi_\mu^6 f)\Bigr]\,dx,
\end{align*}
where each \(m\in S_{\loc}(4)\) and then apply \eqref{m4 type a bound}. Combining these estimates and using \eqref{goodness} gives us
\begin{align}
\int \Biggl|\int \biggl[\eqref{l I} & - \tfrac{2\vk^2}{\kappa^2 - \vk^2}\Bigl[16\vk^4|u(\vk)|^2\bar u(\vk)'+ 4\vk^4\bar u(\vk)^2u(\vk)'\notag\\
& - 64i\vk^8|u(\vk)|^4\bar u(\vk)\Bigr]q\biggr]\,\psi_\mu^{24}\,dx\Biggr|\,e^{-\frac1{200}|h-\mu|}\,d\mu\lesssim |\vk|^{-1}\|q\|_{F_\kappa^{\frac12}(h)}^2\|q\|_{E_{2\sigma,\vk}^\sigma}^2.\label{X-A}
\end{align}

We estimate the contribution of  \eqref{l II} similarly. From \eqref{paraproduct for frac geq}, we have
\begin{align*}
&\int \tfrac1{\sqrt\vk}\bigl(\tfrac{g_{21}}{2+\gamma}\bigr)\sbrack{\geq 7}(\vk)\,q'\,\psi_\mu^{24}\,dx\\
&\qquad = \vk\int m\Bigl[\psi_\mu^3f,\psi_\mu^3f,\psi_\mu^3f,\tfrac{\psi_\mu^3 f}{2\vk - \p},\tfrac{\psi_\mu^3 f}{2\vk - \p},\tfrac{\psi_\mu^3 f}{2\vk - \p},\tfrac{\psi_\mu^3 f}{2\vk - \p},\tfrac{2-\p}{2\vk - \p}(\psi_\mu^3 f)\Bigr]\,dx,
\end{align*}
where \(m\in S_{\loc}(8)\), whereas from \eqref{g21 frac 5 expansion} we have
\begin{align*}
&\int \Bigl[\tfrac1{\sqrt\vk}\bigl(\tfrac{g_{21}}{2+\gamma}\bigr)\sbrack{5}(\vk) + \tfrac1{\sqrt{-\vk}}\bigl(\tfrac{g_{21}}{2+\gamma}\bigr)\sbrack{5}(-\vk)\Bigr]\,q'\,\psi_\mu^{24}\,dx\\
&\qquad =  \int m\Bigl[\psi_\mu^4f,\psi_\mu^4f,\tfrac{\psi_\mu^4 f}{2\vk - \p},\tfrac{\psi_\mu^4 f}{2\vk - \p},\tfrac{2-\p}{2\vk - \p}(\psi_\mu^4 f),\tfrac{2-\p}{2\vk - \p}(\psi_\mu^4 f)\Bigr]\,dx,
\end{align*}
where each \(m\in S_{\loc}(6)\).  Using the second expression on $\RHS{g21 frac 3 expansion}$ we get
\begin{align*}
&\int \Bigl[\tfrac1{\sqrt\vk}\bigl(\tfrac{g_{21}}{2+\gamma}\bigr)\sbrack{3}(\vk) + \tfrac1{\sqrt{-\vk}}\bigl(\tfrac{g_{21}}{2+\gamma}\bigr)\sbrack{3}(-\vk)-8\vk^4|u(\vk)|^2\bar u(\vk)\Bigr]\,q'\,\psi_\mu^{24}\,dx\\
&\qquad = \int m\Bigl[\psi_\mu^6f,\tfrac{\psi_\mu^6 f}{2\vk - \p},\tfrac{2-\p}{2\vk - \p}(\psi_\mu^6 f),\tfrac{4-\p^2}{4\vk^2 - \p^2}(\psi_\mu^6 f)\Bigr]\,dx\\
&\qquad\quad + \int m\Bigl[\tfrac{\psi_\mu^6 f}{2\vk - \p},\tfrac{2-\p}{2\vk - \p}(\psi_\mu^6 f),\tfrac{2-\p}{2\vk - \p}(\psi_\mu^6 f),\tfrac{2-\p}{2\vk - \p}(\psi_\mu^6 f)\Bigr]\,dx,
\end{align*}
where each \(m\in S_{\loc}(4)\). Again applying \eqref{19-5}, \eqref{m6 diff a bound}, \eqref{m4 type a bound}, respectively, and using \eqref{goodness} we have the estimate
\begin{align}
\int \biggl|\int \Bigl[\eqref{l II} - \tfrac{\vk^2}{\kappa^2 - \vk^2}\Bigl[-8\vk^4|u(\vk)|^2\bar u(\vk)\Bigr]q'\Bigr]\,&\psi_\mu^{24}\,dx\biggr|\,e^{-\frac1{200}|h-\mu|}\,d\mu\label{X-B}\\
&\lesssim |\vk|^{-1}\|q\|_{F_\kappa^{\frac12}(h)}^2\|q\|_{E_{2\sigma,\vk}^\sigma}^2.\notag
\end{align}

For \eqref{l III} we first write
\[
|q|^2q = 16\vk^4 |u(\vk)|^2 q + m\Bigl[f,f,\tfrac{f''}{4\vk^2 - \p^2}\Bigr],
\]
where \(m\in S(3)\). From \eqref{paraproduct for frac geq}, we then have
\begin{align*}
&\int \tfrac1{\sqrt\vk}\bigl(\tfrac{g_{21}}{2+\gamma}\bigr)\sbrack{\geq 3}(\vk) \Bigl[|q|^2 - 16\vk^4|u(\vk)|^2\Bigr]q\,\psi_\mu^{24}\,dx\\
&\qquad= \int m\Bigl[\psi_\mu^4f,\psi_\mu^4f,\psi_\mu^4f,\tfrac{\psi_\mu^4 f}{2\vk - \p},\tfrac{\psi_\mu^4 f}{2\vk - \p},\tfrac{4-\p^2}{4\vk^2 - \p^2}(\psi_\mu^4 f)\Bigr]\,dx,
\end{align*}
where \(m\in S_{\loc}(6)\), to which we can apply \eqref{m6 diff a bound}. Next, we use \eqref{paraproduct for frac geq} to write
\begin{align*}
&\int \tfrac1{\sqrt\vk}\bigl(\tfrac{g_{21}}{2+\gamma}\bigr)\sbrack{\geq 7}(\vk)\,16\vk^4|u(\vk)|^2q\,\psi_\mu^{24}\,dx\\
&\qquad= \vk^2\int m\Bigl[\underbrace{\psi_\mu^2f,\dots,\psi_\mu^2 f}_4,\underbrace{\tfrac{\psi_\mu^2 f}{2\vk - \p},\dots,\tfrac{\psi_\mu^2 f}{2\vk - \p}}_6\Bigr]\,\psi_\mu^4\,dx,
\end{align*}
where \(m\in S_{\loc}(10)\), which can be estimated using \eqref{19-1}. Similarly, using \eqref{g21 frac 5 expansion} we have
\begin{align*}
&\int \Bigl[\tfrac1{\sqrt\vk}\bigl(\tfrac{g_{21}}{2+\gamma}\bigr)\sbrack{5}(\vk) + \tfrac1{\sqrt{-\vk}}\bigl(\tfrac{g_{21}}{2+\gamma}\bigr)\sbrack{5}(-\vk)\Bigr]\,16\vk^4 |u(\vk)|^2 q\,\psi_\mu^{24}\,dx\\
&\qquad = \vk\int m\Bigl[\psi_\mu^3f,\psi_\mu^3f,\psi_\mu^3f,\tfrac{\psi_\mu^3 f}{2\vk - \p},\tfrac{\psi_\mu^3 f}{2\vk - \p},\tfrac{\psi_\mu^3 f}{2\vk - \p},\tfrac{\psi_\mu^3 f}{2\vk - \p},\tfrac{2-\p}{2\vk - \p}(\psi_\mu^3 f)\Bigr]\,dx,
\end{align*}
where \(m\in S_{\loc}(6)\), which we estimate using \eqref{m6 diff a bound}. From \eqref{g21 frac 3 expansion} we have
\begin{align*}
&\int \Bigl[\tfrac1{\sqrt\vk}\bigl(\tfrac{g_{21}}{2+\gamma}\bigr)\sbrack{3}(\vk) + \tfrac1{\sqrt{-\vk}}\bigl(\tfrac{g_{21}}{2+\gamma}\bigr)\sbrack{3}(-\vk) - 8\vk^4|u(\vk)|^2\bar u(\vk)\Bigr]\,16\vk^4|u(\vk)|^2q\,\psi_\mu^{24}\,dx\\
&\qquad = \int m\Bigl[\psi_\mu^4f,\psi_\mu^4f,\psi_\mu^4f,\tfrac{\psi_\mu^4 f}{2\vk - \p},\tfrac{\psi_\mu^4 f}{2\vk - \p},\tfrac{4-\p^2}{4\vk^2 - \p^2}(\psi_\mu^4 f)\Bigr]\,dx\\
&\qquad\quad + \int m\Bigl[\psi_\mu^4f,\psi_\mu^4f,\tfrac{\psi_\mu^4 f}{2\vk - \p},\tfrac{\psi_\mu^4 f}{2\vk - \p},\tfrac{2-\p}{2\vk - \p}(\psi_\mu^4 f),\tfrac{2-\p}{2\vk - \p}(\psi_\mu^4 f)\Bigr]\,dx,
\end{align*}
where \(m\in S_{\loc}(4)\), and can then apply \eqref{m6 diff a bound}. Finally, using that
\[
\tfrac1{\sqrt\vk}\bigl(\tfrac{g_{21}}{2+\gamma}\bigr)\sbrack 1(\vk) + \tfrac1{\sqrt{-\vk}}\bigl(\tfrac{g_{21}}{2+\gamma}\bigr)\sbrack 1(-\vk) = i\bar u(\vk)'
\]
and \eqref{goodness}, we obtain the estimate
\begin{align}
\int \Bigl|\int \Bigl[\eqref{l III} &- \tfrac{i\vk^2}{\kappa^2 - \vk^2}\Bigl[-i\bar u(\vk)' - 8\vk^4|u(\vk)|^2\bar u(\vk)\Bigr]16\vk^4|u(\vk)|^2q\Bigr]\,\psi_\mu^{24}\,dx\Bigr|\,e^{-\frac1{200}|h-\mu|}\,d\mu\notag\\
&\lesssim |\vk|^{-1}\|q\|_{F_\kappa^{\frac12}(h)}^2\|q\|_{E_{2\sigma,\vk}^\sigma}^2.\label{X-C}
\end{align}

For \eqref{l IV}, we use \eqref{paraproduct for frac geq} to write
\begin{align*}
\int \sqrt\vk \bigl(\tfrac{g_{21}}{2+\gamma}\bigr)\sbrack{\geq 5}(\vk)\,u(\kappa)''\,\psi_\mu^{24}\,dx = \int m\Bigl[\psi_\mu^4f,\psi_\mu^4f,\psi_\mu^4f,\tfrac{\psi_\mu^4 f}{2\vk-\p},\tfrac{\psi_\mu^4 f}{2\vk-\p},\tfrac{4-\p^2}{4\kappa^2 - \p^2}(\psi_\mu^4 f)\Bigr]\,dx,
\end{align*}
where \(m\in S_{\loc}(6)\), and \eqref{g21 frac 3 expansion} to write
\begin{align*}
&\int\Bigl[\sqrt\vk \bigl(\tfrac{g_{21}}{2+\gamma}\bigr)\sbrack{3}(\vk) + \sqrt{-\vk} \bigl(\tfrac{g_{21}}{2+\gamma}\bigr)\sbrack{3}(-\vk)\Bigr]u(\kappa)''\,\psi_\mu^{24}\,dx\\
&\qquad = \int m\Bigl[\psi_\mu^6f,\tfrac{\psi_\mu^6 f}{2\vk-\p},\tfrac{2-\p}{2\vk-\p}(\psi_\mu^6f),\tfrac{4-\p^2}{4\kappa^2 - \p^2}(\psi_\mu^6 f)\Bigr]\,dx,
\end{align*}
where \(m\in S_{\loc}(4)\). Similarly, for \eqref{l V} we use \eqref{paraproduct for frac geq} to write
\begin{align*}
\int \tfrac1{\sqrt\vk} \bigl(\tfrac{g_{21}}{2+\gamma}\bigr)\sbrack{\geq 3}(\vk)\,u(\kappa)'''\,\psi_\mu^{24}\,dx = \int m\Bigl[\psi_\mu^6f,\psi_\mu^6f,\tfrac{\psi_\mu^6 f}{2\vk-\p},\tfrac{(2-\p)^3}{(2\vk - \p)(4\kappa^2 - \p^2)}(\psi_\mu^6 f)\Bigr]\,dx.
\end{align*}
Applying \eqref{303030} and \eqref{m4 diff type II bound} gives us
\begin{align}
\int \left|\int \Bigl[\eqref{l IV} + \eqref{l V}\Bigr]\,\psi_\mu^{24}\,dx\right|\,e^{-\frac1{200}|h-\mu|}\,d\mu\lesssim |\vk|^{-1}\|q\|_{F_\kappa^{\frac12}(h)}^2\|q\|_{E_{2\sigma,\vk}^\sigma}^2.\label{X-D}
\end{align}

Turning to \eqref{l VIII}, we first take a test function \(w\in L^\infty\) and use \eqref{mult comm}, \eqref{I8}, \eqref{op norm bound}, and \eqref{goodness} to bound
\begin{align*}
\left|\int w\,\gamma\sbrack{\geq 8}(\pm\kappa)\,\psi_\mu^{24}\,dx\right| &\lesssim \sum_{\ell=0}^\infty\kappa^{4} \|\Lambda(w)\|_{\op}\|\Lambda(q \psi_\mu^3)\|_{\I_8}^8\bigl(C\sqrt \kappa \|\Lambda(q)\|_{\op}\bigr)^{2\ell}\\
&\lesssim \kappa^{-2}\Bigl\|\tfrac{\psi_\mu^3q}{\sqrt{4\kappa^2 - \p^2}}\Bigr\|_{H^{\frac32}}^2\|q\|_{E_\sigma^\sigma}^6\|w\|_{L^\infty}. 
\end{align*}
By duality, this yields
\[
\|\gamma\sbrack{\geq 8}(\pm\kappa)\,\psi_\mu^{24}\|_{L^1}\lesssim \kappa^{-2}\Bigl\|\tfrac{\psi_\mu^3q}{\sqrt{4\kappa^2 - \p^2}}\Bigr\|_{H^{\frac32}}^2\|q\|_{E_\sigma^\sigma}^6.
\]
For the remaining terms, first use \eqref{gamma 6 expansion} to write
\begin{align*}
&\int \Bigl[\gamma\sbrack 6(\kappa) - \gamma\sbrack 6(-\kappa) + 2560i\kappa^9|u(\kappa)|^6\Bigr]\psi_\mu^{24}\,dx\\
&\qquad = \kappa^{-1}\int m\Bigl[\psi_\mu^4 f,\psi_\mu^4 f,\psi_\mu^4 f,\tfrac{\psi_\mu^4 f}{2\kappa - \p},\tfrac{\psi_\mu^4 f}{2\kappa - \p},\tfrac{2-\p}{2\kappa - \p}(\psi_\mu^4f)\Bigr]\,dx,
\end{align*}
which we bound using \eqref{m6 diff b bound}. Second, apply \eqref{gamma 4 expansion} to obtain
\begin{align*}
&\int \Bigl[\gamma\sbrack 4(\kappa) - \gamma\sbrack 4(-\kappa) + 192\kappa^5|u(\kappa)|^2\bigl[u(\kappa)'\bar u(\kappa) - u(\kappa)\bar u(\kappa)'\bigr]\Bigr]\psi_\mu^{24}\,dx\\
&\qquad = \kappa^{-2}\int m\Bigl[\psi_\mu^6 f,\psi_\mu^6 f,\tfrac{2-\p}{2\kappa - \p}(\psi_\mu^6f),\tfrac{2-\p}{2\kappa - \p}(\psi_\mu^6f)\Bigr]\,dx\\
&\qquad\quad + \kappa^{-2}\int m\Bigl[\psi_\mu^6 f,\psi_\mu^6 f,\psi_\mu^6 f,\tfrac{4-\p^2}{4\kappa^2 - \p^2}(\psi_\mu^6f)\Bigr]\,dx,
\end{align*}
which we estimate by \eqref{m4 type a' bound}. Together, these give us
\begin{align}
&\int \Biggl|\int \biggl[\eqref{l VIII}-\tfrac{\kappa^3}{\kappa^2 - \vk^2}\Bigl[192\kappa^5|u(\kappa)|^2\bigl[u(\kappa)'\bar u(\kappa) - u(\kappa)\bar u(\kappa)'\bigr]\label{X-E}\\
&\qquad\qquad\qquad\qquad\qquad\qquad+2560i\kappa^9|u(\kappa)|^6\Bigr]\biggr]\,\psi_\mu^{24}\,dx\Biggr|\,e^{-\frac1{200}|h-\mu|}\,d\mu\notag\\
&\qquad\lesssim \tfrac \kappa{\kappa^2 + \vk^2}\|q\|_{F_\kappa^{\frac12}(h)}^2\|q\|_{E_\sigma^\sigma}^2.\notag
\end{align}

For \eqref{l VI}, we take \(h = \sqrt{\pm \vk}\bigl(\tfrac{g_{21}}{2+\gamma}\bigr)(\pm \vk)\) and again apply \eqref{mult comm}, \eqref{I8}, \eqref{op norm bound}, \eqref{goodness}, and Corollary~\ref{c:g} to estimate
\begin{align*}
\left|\int h\,\tfrac1{\sqrt{\pm\kappa}}g_{12}\sbrack{\geq 7}(\pm\kappa)\,\psi_\mu^{24}\,dx\right|&\lesssim \sum_{\ell=0}^\infty\kappa^{3} \|\Lambda(h\psi_\mu^3)\|_{\I_8}\|\Lambda(q \psi_\mu^3)\|_{\I_8}^7\bigl(C\sqrt\kappa\|\Lambda(q)\|_{\op}\bigr)^{2\ell}\\
&\lesssim \kappa^{-2}\Bigl\|\tfrac{\psi_\mu^3q}{\sqrt{4\kappa^2 - \p^2}}\Bigr\|_{H^{\frac32}}^2\|q\|_{E_\sigma^\sigma}^6,
\end{align*}
which is acceptable. For the lower order terms, we use \eqref{g12 3 expansion} and \eqref{g12 5 expansion} to write
\begin{align*}
&\int h\,\Bigl[\tfrac1{\sqrt\kappa}g_{12}\sbrack 3(\kappa) - \tfrac1{\sqrt{-\kappa}}g_{12}\sbrack 3(-\kappa)-48i\kappa^3|u(\kappa)|^2u(\kappa)'\Bigr]\,\psi_\mu^{24} \,dx\\
&\qquad =  \kappa^{-2}\int m\Bigl[\psi_\mu^6 f,\psi_\mu^6 f,\tfrac{2-\p}{2\kappa - \p}(\psi_\mu^6f),\tfrac{2-\p}{2\kappa - \p}(\psi_\mu^6f)\Bigr]\,dx\\
&\qquad\quad + \kappa^{-2}\int m\Bigl[\psi_\mu^6 f,\psi_\mu^6 f,\psi_\mu^6 f,\tfrac{4-\p^2}{4\kappa^2 - \p^2}(\psi_\mu^6f)\Bigr]\,dx,
\end{align*}
where each \(m\in S_{\loc}(4)\), and 
\begin{align*}
&\int h\,\Bigl[\tfrac1{\sqrt\kappa}g_{12}\sbrack 5(\kappa) - \tfrac1{\sqrt{-\kappa}}g_{12}\sbrack 5(-\kappa)+384\kappa^7|u(\kappa)|^4u(\kappa)\Bigr]\,\psi_\mu^{24} \,dx\\
&\qquad =  \kappa^{-1}\int m\Bigl[\psi_\mu^4 f,\psi_\mu^4 f,\psi_\mu^4 f,\tfrac{\psi_\mu^4 f}{2\kappa - \p},\tfrac{\psi_\mu^4 f}{2\kappa - \p},\tfrac{2-\p}{2\kappa - \p}(\psi_\mu^4f)\Bigr]\,dx,
\end{align*}
where each \(m\in S_{\loc}(6)\). As a consequence, if we introduce
\begin{align}
=\;\tfrac{2\kappa^3}{\kappa^2 - \vk^2}\Bigl[\sqrt\vk\bigl(\tfrac{g_{21}}{2+\gamma}\bigr)(\vk) + \sqrt{-\vk}&\bigl(\tfrac{g_{21}}{2+\gamma}\bigr)(-\vk)\Bigr]\label{l VI correction}\\
&\times\Bigl[-48i\kappa^3|u(\kappa)|^2u(\kappa)' + 384\kappa^7|u(\kappa)|^4u(\kappa)\Bigr],\notag
\end{align}
we may apply \eqref{m4 type a' bound} and \eqref{m6 diff b bound} to obtain
\begin{align}
&\int \left|\int \Bigl[\eqref{l VI} - \eqref{l VI correction}\Bigr]\,\psi_\mu^{24}\,dx\right|\,e^{-\frac1{200}|h-\mu|}\,d\mu\lesssim \tfrac \kappa{\kappa^2 + \vk^2}\|q\|_{F_\kappa^{\frac12}(h)}^2\|q\|_{E_\sigma^\sigma}^2.\label{X-F}
\end{align}

To estimate \eqref{l VI correction}, we first use \eqref{paraproduct for frac geq} to write
\begin{align*}
&\int \sqrt\vk\bigl(\tfrac{g_{21}}{2+\gamma}\bigr)\sbrack{\geq 5}(\vk)\,\kappa^3|u(\kappa)|^2u(\kappa)'\,\psi_\mu^{24}\,dx& \\
&\qquad= \int m\Bigl[\psi_\mu^3 f,\psi_\mu^3 f,\psi_\mu^3 f,\tfrac{\psi_\mu^3 f}{2\vk - \p},\tfrac{\psi_\mu^3 f}{2\vk - \p},\tfrac{\psi_\mu^3 f}{2\kappa - \p},\tfrac{\psi_\mu^3 f}{2\kappa - \p},\tfrac{2-\p}{2\kappa - \p}(\psi_\mu^3 f)\Bigr]\,dx,
\end{align*}
where \(m\in S_{\loc}(8)\), to which we apply \eqref{15-2}. Similarly,
\begin{align*}
&\int \sqrt\vk\bigl(\tfrac{g_{21}}{2+\gamma}\bigr)\sbrack{\geq 5}(\vk)\,\kappa^7|u(\kappa)|^4u(\kappa)\,\psi_\mu^{24}\,dx& \\
&\qquad= \kappa\int m\Bigl[\underbrace{\psi_\mu^2 f,\dots,\psi_\mu^2 f}_4,\tfrac{\psi_\mu^2 f}{2\vk - \p},\tfrac{\psi_\mu^2 f}{2\vk - \p},\underbrace{\tfrac{\psi_\mu^2 f}{2\kappa - \p},\dots,\tfrac{\psi_\mu^2 f}{2\kappa - \p}}_4\Bigr]\,\psi_\mu^4\,dx,
\end{align*}
where \(m\in S_{\loc}(10)\), which can be bounded using \eqref{S10}. Further, from \eqref{g21 frac 3 expansion} we have
\begin{align*}
&\int \Bigl[\sqrt\vk\bigl(\tfrac{g_{21}}{2+\gamma}\bigr)\sbrack{3}(\vk) + \sqrt{-\vk}\bigl(\tfrac{g_{21}}{2+\gamma}\bigr)\sbrack{ 3}(-\vk)\Bigr]\,\kappa^3|u(\kappa)|^2u(\kappa)'\,\psi_\mu^{24}\,dx& \\
&\qquad=\int m\Bigl[\psi_\mu^4f,\tfrac{\psi_\mu^4 f}{2\vk - \p},\tfrac{2-\p}{2\vk - \p}(\psi_\mu^4 f),\tfrac{\psi_\mu^4 f}{2\kappa - \p},\tfrac{\psi_\mu^4 f}{2\kappa - \p},\tfrac{2-\p}{2\kappa - \p}(\psi_\mu^4 f)\Bigr]\,dx,
\end{align*}
where \(m\in S_{\loc}(6)\) to which we apply \eqref{200x200}, and
\begin{align*}
&\int \Bigl[\sqrt\vk\bigl(\tfrac{g_{21}}{2+\gamma}\bigr)\sbrack{3}(\vk) + \sqrt{-\vk}\bigl(\tfrac{g_{21}}{2+\gamma}\bigr)\sbrack{ 3}(-\vk)\Bigr]\,\kappa^7|u(\kappa)|^4u(\kappa)\,\psi_\mu^{24}\,dx& \\
&\qquad=\kappa\int m\Bigl[\psi_\mu^3f,\psi_\mu^3f,\tfrac{\psi_\mu^3 f}{2\vk - \p},\tfrac{2-\p}{2\vk - \p}(\psi_\mu^3 f),\tfrac{\psi_\mu^3 f}{2\kappa - \p},\tfrac{\psi_\mu^3 f}{2\kappa - \p},\tfrac{\psi_\mu^3 f}{2\kappa - \p},\tfrac{\psi_\mu^3 f}{2\kappa - \p}\Bigr]\,dx,
\end{align*}
which we estimate with \eqref{15-2}. Observing that
\[
\sqrt\vk\bigl(\tfrac{g_{21}}{2+\gamma}\bigr)\sbrack 1(\vk) + \sqrt{-\vk}\bigl(\tfrac{g_{21}}{2+\gamma}\bigr)\sbrack 1(-\vk) = -2i\vk^2\bar u(\vk),
\]
and once again using \eqref{goodness}, these bounds combine to give us the estimate
\begin{align}
\int \biggl|\int \biggl[\eqref{l VI correction}+&\tfrac{2\kappa^3}{\kappa^2 - \vk^2}\Bigl[-2i\vk^2\bar u(\vk)\Bigr]\notag\\
&\times\Bigl[48i\kappa^3|u(\kappa)|^2u(\kappa)' - 384\kappa^7|u(\kappa)|^4u(\kappa)\Bigr]
\biggr]\,\psi_\mu^{24}\,dx\biggr|\,e^{-\frac1{200}|h-\mu|}\,d\mu\label{X-G}\\
&\lesssim |\vk|^{-1}\|q\|_{F_\kappa^{\frac12}(h)}^2\|q\|_{E_\sigma^\sigma}^2\|q\|_{E_{2\sigma,\vk}^\sigma}^2.\notag
\end{align}

It remains to extract the leading order terms from \eqref{l VII}. We first use \eqref{useful IDs} and \eqref{paraproduct for frac geq} to write
\begin{align}
\tfrac1{\sqrt\vk}\bigl(\tfrac{g_{21}}{2+\gamma}\bigr)\sbrack 1(\vk) + \tfrac1{\sqrt{-\vk}}\bigl(\tfrac{g_{21}}{2+\gamma}\bigr)\sbrack 1(-\vk) &= i\bar u(\vk)'\label{spigot}\\
\tfrac1{\sqrt\vk}\bigl(\tfrac{g_{21}}{2+\gamma}\bigr)\sbrack 3(\vk) + \tfrac1{\sqrt{-\vk}}\bigl(\tfrac{g_{21}}{2+\gamma}\bigr)\sbrack 3(-\vk)&=\tfrac \vk{2\vk + \p}m\Bigl[f,\tfrac f{2\vk - \p},\tfrac f{2\vk - \p}\Bigr],\label{faucet}
\end{align}
where \(m\in S(3)\). From \eqref{paraproduct for g geq}, we then have
\begin{align*}
&\int i\bar u(\vk)'\sqrt{\pm\kappa}g_{12}\sbrack{\geq 7}(\pm \kappa)\,\psi_\mu^{24}\,dx\\
&\qquad = \kappa\int m\Bigl[\psi_\mu^3f,\psi_\mu^3f,\psi_\mu^3f,\tfrac{\psi_\mu^3f}{2\kappa - \p},\tfrac{\psi_\mu^3f}{2\kappa - \p},\tfrac{\psi_\mu^3f}{2\kappa - \p},\tfrac{\psi_\mu^3f}{2\kappa - \p},\tfrac{2-\p}{4\vk^2 - \p^2}(\psi_\mu^3f)\Bigr]\,dx,
\end{align*}
where \(m\in S_{\loc}(8)\) and
\begin{align*}
&\int \Bigl[\tfrac1{\sqrt\vk}\bigl(\tfrac{g_{21}}{2+\gamma}\bigr)\sbrack{\geq 3}(\vk) + \tfrac1{\sqrt{-\vk}}\bigl(\tfrac{g_{21}}{2+\gamma}\bigr)\sbrack{\geq 3}(\vk)\Bigr]\sqrt{\pm\kappa}g_{12}\sbrack{\geq 7}(\pm \kappa)\,\psi_\mu^{24}\,dx\\
&\qquad = \kappa\int m\Bigl[\underbrace{\psi_\mu^2f,\dots,\psi_\mu^2f}_4,\tfrac{\psi_\mu^2f}{2\vk - \p},\tfrac{\psi_\mu^2f}{2\vk - \p},\underbrace{\tfrac{\psi_\mu^2f}{2\kappa - \p},\dots,\tfrac{\psi_\mu^2f}{2\kappa - \p}}_4\Bigr]\,\psi_\mu^4\,dx,
\end{align*}
where \(m\in S_{\loc}(10)\). These terms can be estimated using \eqref{15-2} and  \eqref{S10}, respectively. Applying \eqref{g12 5 expansion} we have
\begin{align*}
&\int i\bar u(\vk)'\Bigl[\sqrt{\kappa}g_{12}\sbrack{5}(\kappa) - \sqrt{-\kappa}g_{12}\sbrack 5(-\kappa)\Bigr]\,\psi_\mu^{24}\,dx\\
&\qquad = \int m\Bigl[\psi_\mu^4f,\psi_\mu^4f,\tfrac{\psi_\mu^4f}{2\kappa - \p},\tfrac{\psi_\mu^4f}{2\kappa - \p},\tfrac{2-\p}{2\kappa - \p}(\psi_\mu^4f),\tfrac{2-\p}{4\vk^2 - \p^2}(\psi_\mu^4f)\Bigr]\,dx,
\end{align*}
where \(m\in S_{\loc}(6)\), which we bound using \eqref{200x200}. Similarly, we have
\begin{align*}
&\int \Bigl[\tfrac1{\sqrt\vk}\bigl(\tfrac{g_{21}}{2+\gamma}\bigr)\sbrack{\geq 3}(\vk) + \tfrac1{\sqrt{-\vk}}\bigl(\tfrac{g_{21}}{2+\gamma}\bigr)\sbrack{\geq 3}(-\vk)\Bigr]\Bigl[\sqrt{\kappa}g_{12}\sbrack{5}(\kappa) - \sqrt{-\kappa}g_{12}\sbrack 5(-\kappa)\Bigr]\,\psi_\mu^{24}\,dx\\
&\qquad = \int m\Bigl[\psi_\mu^3f,\psi_\mu^3f,\psi_\mu^3f,\tfrac{\psi_\mu^3f}{2\vk - \p},\tfrac{\psi_\mu^3f}{2\vk - \p},\tfrac{\psi_\mu^3f}{2\kappa - \p},\tfrac{\psi_\mu^3f}{2\kappa - \p},\tfrac{2-\p}{2\kappa - \p}(\psi_\mu^3f)\Bigr]\,dx,
\end{align*}
where \(m\in S_{\loc}(8)\), which is bounded using \eqref{15-2}. Using \eqref{g12 3 expansion} we have
\begin{align*}
&\int i\bar u(\vk)'\Bigl[\sqrt{\kappa}g_{12}\sbrack{3}(\kappa) - \sqrt{-\kappa}g_{12}\sbrack 3(-\kappa) - 32i\kappa^5|u(\kappa)|^2u(\kappa)\Bigr]\,\psi_\mu^{24}\,dx\\
&\qquad = \kappa^{-1}\int m\Bigl[\psi_\mu^6f,\tfrac{2-\p}{2\kappa - \p}(\psi_\mu^6f),\tfrac{2-\p}{2\kappa - \p}(\psi_\mu^6f),\tfrac{2-\p}{4\vk^2 - \p^2}(\psi_\mu^6f)\Bigr]\,dx\\
&\qquad\quad + \kappa^{-1}\int m\Bigl[\psi_\mu^6f,\psi_\mu^6f,\tfrac{4-\p^2}{4\kappa^2 - \p^2}(\psi_\mu^6f),\tfrac{2-\p}{4\vk^2 - \p^2}(\psi_\mu^6f)\Bigr]\,dx
\end{align*}
where \(m\in S_{\loc}(4)\). This is estimated using \eqref{303030}. Also using \eqref{g12 3 expansion} we have
\begin{align*}
&\int \Bigl[\tfrac1{\sqrt\vk}\bigl(\tfrac{g_{21}}{2+\gamma}\bigr)\sbrack{\geq 3}(\vk) + \tfrac1{\sqrt{-\vk}}\bigl(\tfrac{g_{21}}{2+\gamma}\bigr)\sbrack{\geq 3}(-\vk)\Bigr]\\
&\qquad\qquad\qquad\qquad\qquad\qquad\Bigl[\sqrt{\kappa}g_{12}\sbrack{3}(\kappa) - \sqrt{-\kappa}g_{12}\sbrack 3(-\kappa) - 32i\kappa^5|u(\kappa)|^2u(\kappa)\Bigr]\,\psi_\mu^{24}\,dx\\
&\qquad = \kappa^{-1}\int m\Bigl[\psi_\mu^4f,\psi_\mu^4f,\tfrac{\psi_\mu^4f}{2\vk - \p},\tfrac{\psi_\mu^4f}{2\vk - \p},\tfrac{2-\p}{2\kappa - \p}(\psi_\mu^4f),\tfrac{2-\p}{2\kappa - \p}(\psi_\mu^4f)\Bigr]\,dx\\
&\qquad\quad + \kappa^{-1}\int m\Bigl[\psi_\mu^4f,\psi_\mu^4f,\psi_\mu^4f,\tfrac{\psi_\mu^4f}{2\vk - \p},\tfrac{\psi_\mu^4f}{2\vk - \p},\tfrac{4-\p^2}{4\kappa^2 - \p^2}(\psi_\mu^4f)\Bigr]\,dx,
\end{align*}
where \(m\in m_\mu(6)\). A final application of \eqref{m4 diff type II bound}  and \eqref{goodness} gives us the estimate
\begin{align}
&\int \left|\int \Bigl[\eqref{l VII} - \eqref{l VII correction}\Bigr]\,\psi_\mu^{24}\,dx\right|\,e^{-\frac1{200}|h-\mu|}\,d\mu\lesssim |\vk|^{-1}\|q\|_{F_\kappa^{\frac12}(h)}^2\|q\|_{E_{2\sigma,\vk}^\sigma}^2,\label{X-H}
\end{align}
where we define
\begin{align}
& =\;\tfrac{2\kappa^3}{\kappa^2 - \vk^2} \Bigl[\tfrac1{\sqrt\vk}\bigl(\tfrac{g_{21}}{2+\gamma}\bigr)(\vk) + \tfrac1{\sqrt{-\vk}}\bigl(\tfrac{g_{21}}{2+\gamma}\bigr)(-\vk)\Bigr]\Bigl[\tfrac i{2\kappa}|q|^2q-32i\kappa^5|u(\kappa)|^2u(\kappa)\Bigr].\label{l VII correction}
\end{align}
However, by writing
\begin{align*}
\tfrac i{2\kappa}|q|^2q-32i\kappa^5|u(\kappa)|^2u(\kappa) = -\tfrac i{2\kappa}|q|^2u(\kappa)''- 2i\kappa u(\kappa)\bar u(\kappa)''q - 8i\kappa^3|u(\kappa)|^2u(\kappa)'',
\end{align*}
and using \eqref{spigot}, \eqref{faucet}, we obtain
\begin{align*}
&\Bigl[\tfrac1{\sqrt\vk}\bigl(\tfrac{g_{21}}{2+\gamma}\bigr)\sbrack 1(\vk) + \tfrac1{\sqrt{-\vk}}\bigl(\tfrac{g_{21}}{2+\gamma}\bigr)\sbrack 1(-\vk)\Bigr]\Bigl[\tfrac i{2\kappa}|q|^2q-32i\kappa^5|u(\kappa)|^2u(\kappa)\Bigr]\\
&\qquad= \kappa^{-1}m\Bigl[f,f,\tfrac{f'}{4\vk^2 - \p^2},\tfrac{f''}{4\kappa^2 - \p^2}\Bigr],\\
&\Bigl[\tfrac1{\sqrt\vk}\bigl(\tfrac{g_{21}}{2+\gamma}\bigr)\sbrack{\geq 3}(\vk) + \tfrac1{\sqrt{-\vk}}\bigl(\tfrac{g_{21}}{2+\gamma}\bigr)\sbrack{\geq 3}(-\vk)\Bigr]\Bigl[\tfrac i{2\kappa}|q|^2q-32i\kappa^5|u(\kappa)|^2u(\kappa)\Bigr]\\
&\qquad= \kappa^{-1}m\Bigl[f,f,f,\tfrac{f}{2\vk - \p},\tfrac{f}{2\vk - \p},\tfrac{f''}{4\kappa^2 - \p^2}\Bigr],
\end{align*}
for \(m\in S(4)\) and \(S(6)\), respectively. As a consequence, we may apply \eqref{303030} and  \eqref{m4 diff type II bound} to obtain
\begin{align}
\int \left|\int \eqref{l VII correction}\,\psi_\mu^{24}\,dx\right|\,e^{-\frac1{200}|h-\mu|}\,d\mu\lesssim |\vk|^{-1}\|q\|_{F_\kappa^{\frac12}(h)}^2\|q\|_{E_{2\sigma,\vk}^\sigma}^2,\label{X-I}
\end{align}

We now collect the leading order terms from our above estimates into quartic and sextic contributions, as folllows
\begin{align*}
J_1 = \tfrac1{\kappa^2 - \vk^2}\Bigl\{&-48\vk^6|u(\vk)|^2\bar q u(\vk)' + 8\vk^6q\bar u(\vk)^2u(\vk)'- 8\vk^6|u(\vk)|^2\bar u(\vk) q'\\
&-192\kappa^6\vk^2|u(\kappa)|^2\bar u(\vk)u(\kappa)'+384\kappa^8|u(\kappa)|^2\bar u(\kappa)u(\kappa)'\Bigr\},\\
J_2 = \tfrac1{\kappa^2 - \vk^2}\Bigl\{& -256i\vk^{10}|u(\vk)|^4\bar u(\vk) q-1536i\kappa^{10}\vk^2|u(\kappa)|^4u(\kappa)\bar u(\vk)\\
& +2560i\kappa^{12}|u(\kappa)|^6\Bigr\}.
\end{align*}
Combining the estimates \eqref{X-A}--\eqref{X-E}, \eqref{X-F}, \eqref{X-G}, \eqref{X-H}, \eqref{X-I} gives us
\begin{align}
\int \biggl|\Im \int \Bigl[j_{\diff}\sbrack{\geq 4} - J_1 - J_2\Bigr]&\,\psi_\mu^{24}\,dx\biggr|\,e^{-\frac1{200}|h-\mu|}\,d\mu\label{j diff geq 4 perturbative}\\
&\lesssim \|q\|_{F_\kappa^{\frac12}(h)}^2\Bigl[|\vk|^{ - 1}\|q\|_{E_{2\sigma,\vk}^\sigma}^2 + \tfrac\kappa{\kappa^2 + \vk^2}\|q\|_{E_\sigma^\sigma}^2\Bigr].\notag
\end{align}

To estimate \(J_1\), we write
\begin{align*}
-48\vk^6|u(\vk)|^2\bar q u(\vk)' &= -3\vk^2|q|^2\bar q u(\vk)' - 3\vk^2 |q|^2\bar u(\vk)''u(\vk)' - 12\vk^4 u(\vk)''\bar q \bar u(\vk)u(\vk)'\\
&= -3\vk^2|q|^2\bar q u(\vk)' + \vk^2 m\Bigl[f,f,\tfrac{f'}{4\vk^2 - \p^2},\tfrac{f''}{4\vk^2 - \p^2}\Bigr],
\end{align*}
where \(m\in S(4)\). Similarly, we have
\begin{align*}
8\vk^6q\bar u(\vk)^2u(\vk)' &= \tfrac12\vk^2|q|^2\bar qu(\vk)' + \vk^2m\Bigl[f,f,\tfrac{f'}{4\vk^2 - \p^2},\tfrac{f''}{4\vk^2 - \p^2}\Bigr],\\
- 8\vk^6|u(\vk)|^2\bar u(\vk) q' &= -\tfrac12\vk^2|q|^2\bar qu(\vk)'+ \vk^2m\Bigl[f,\tfrac f{2\vk - \p},\tfrac{f}{2\vk - \p},\tfrac{f'''}{4\vk^2 - \p^2}\Bigr]\\
&\quad + \vk^2m\Bigl[f,f,\tfrac{f'}{4\vk^2 - \p^2},\tfrac{f''}{4\vk^2 - \p^2}\Bigr],\\
-192\kappa^6\vk^2|u(\kappa)|^2\bar u(\vk)u(\kappa)' &= -48\kappa^6|u(\kappa)|^2 \bar q u(\kappa)' - 48\kappa^6|u(\kappa)|^2 \bar u(\vk)'' u(\kappa)'\\
&= -3\kappa^2 |q|^2\bar q u(\kappa)' + \kappa m\Bigl[f,f,f,\tfrac{f''}{4\kappa^2 - \p^2}\Bigr]\\
&\quad + \kappa^6 m\Bigl[\tfrac f{4\kappa^2 - \p^2},\tfrac f{4\kappa^2 - \p^2},\tfrac {f'}{4\kappa^2 - \p^2},\tfrac{f''}{4\vk^2 - \p^2}\Bigr],\\
384 \kappa^8 |u(\kappa)|^2\bar u(\kappa)u(\kappa)' &= 6\kappa^2|q|^2\bar qu(\kappa)' + \kappa m\Bigl[f,f,f,\tfrac{f''}{4\kappa^2 - \p^2}\Bigr],
\end{align*}
where each \(m \in S(4)\). Applying \eqref{m4 type a bound}, \eqref{m4 type a' bound}, and \eqref{303030} we then have
\begin{align*}
&\int \left|\int \Bigl[J_1-\tfrac{3[\kappa^2u(\kappa)' - \vk^2u(\vk)']}{\kappa^2 - \vk^2}|q|^2\bar q\Bigr]\,\psi_\mu^{24}\,dx\right|\,e^{-\frac1{200}|h-\mu|}\,d\mu\\
&\qquad\lesssim \|q\|_{F_\kappa^{\frac12}(h)}^2\Bigl[|\vk|^{ - 1}\|q\|_{E_{2\sigma,\vk}^\sigma}^2 + \tfrac\kappa{\kappa^2 + \vk^2}\|q\|_{E_\sigma^\sigma}^2\Bigr].
\end{align*}
For the remaining term in \(J_1\), we compute that
\[
\tfrac{\kappa^2u(\kappa)' - \vk^2u(\vk)'}{\kappa^2 - \vk^2} = -\tfrac{q'''}{(4\kappa^2-\p^2)(4\vk^2-\p^2)},
\]
and applying \eqref{303030} gives us
\begin{align*}
&\int \left|\int \tfrac{\kappa^2u(\kappa)' - \vk^2u(\vk)'}{\kappa^2 - \vk^2}|q|^2\bar q\,\psi_\mu^{24}\,dx\right|\,e^{-\frac1{200}|h-\mu|}\,d\mu\lesssim |\vk|^{ - 1}\|q\|_{F_\kappa^{\frac12}(h)}^2\|q\|_{E_{2\sigma,\vk}^\sigma}^2.
\end{align*}
Collecting these bounds gives us
\begin{align}
\int \left|\int J_1\,\psi_\mu^{24}\,dx\right|\,e^{-\frac1{200}|h-\mu|}\,d\mu\lesssim \|q\|_{F_\kappa^{\frac12}(h)}^2\Bigl[|\vk|^{ - 1}\|q\|_{E_{2\sigma,\vk}^\sigma}^2 + \tfrac\kappa{\kappa^2 + \vk^2}\|q\|_{E_\sigma^\sigma}^2\Bigr].\label{J1}
\end{align}

Finally, we consider \(J_2\). Arguing as for \(J_1\) we have
\begin{align*}
-256i\vk^{10}|u(\vk)|^4\bar u(\vk) q&= -4i\vk^4|q|^4|u(\vk)|^2 + \vk^2 m\Bigl[f,f,f,\tfrac f{2\vk - \p},\tfrac f{2\vk - \p},\tfrac{f''}{4\vk^2 - \p^2}\Bigr],\\
-1536i\kappa^{10}\vk^2|u(\kappa)|^4u(\kappa)\bar u(\vk)&= -6i\kappa^4|q|^4|u(\kappa)|^2 + \kappa^2m\Bigl[f,f,f,\tfrac f{2\kappa - \p},\tfrac f{2\kappa - \p},\tfrac{f'}{2\kappa - \p}\Bigr]\\
&\quad + \kappa^{10} m\Bigl[\tfrac f{4\kappa^2 - \p^2},\tfrac f{4\kappa^2 - \p^2},\tfrac f{4\kappa^2 - \p^2},\tfrac f{4\kappa^2 - \p^2},\tfrac f{4\kappa^2 - \p^2},\tfrac{f''}{4\vk^2 - \p^2}\Bigr],\\
2560i\kappa^{12}|u(\kappa)|^6 &= 10i\kappa^4|q|^4|u(\kappa)|^2 + \kappa^2m\Bigl[f,f,f,\tfrac f{2\kappa - \p},\tfrac f{2\kappa - \p},\tfrac{f'}{2\kappa - \p}\Bigr],
\end{align*}
where each \(m\in S(6)\). Applying \eqref{m6 diff a bound}, \eqref{m6 diff b bound}, and \eqref{m4 diff type II bound} gives us
\begin{align*}
&\int \left|\int \Bigl[J_2-\tfrac{4i[\kappa^4|u(\kappa)|^2 - \vk^4|u(\vk)|^2]}{\kappa^2 - \vk^2}|q|^4\Bigr]\,\psi_\mu^{24}\,dx\right|\,e^{-\frac1{200}|h-\mu|}\,d\mu\\
&\qquad\lesssim \|q\|_{E_\sigma^\sigma}^2\|q\|_{F_\kappa^{\frac12}(h)}^2\Bigl[|\vk|^{- 1}\|q\|_{E_{2\sigma,\vk}^\sigma}^2 + \tfrac\kappa{\kappa^2 + \vk^2}\|q\|_{E_\sigma^\sigma}^2\Bigr].
\end{align*}
For the remaining term we compute that
\[
\tfrac{\kappa^4|u(\kappa)|^2 - \vk^4|u(\vk)|^4}{\kappa^2 - \vk^2} = - \kappa^2\bar u(\kappa)\tfrac{q''}{(4\kappa^2 - \p^2)(4\vk^2 - \p^2)} - \vk^2 u(\vk)\tfrac{\bar q''}{(4\kappa^2 - \p^2)(4\vk^2 - \p^2)},
\]
and can then apply \eqref{m4 diff type II bound} to conclude that
\begin{align*}
\int \left|\int \tfrac{4i[\kappa^4|u(\kappa)|^2 - \vk^4|u(\vk)|^2]}{\kappa^2 - \vk^2}|q|^4\,\psi_\mu^{24}\,dx\right|\,e^{-\frac1{200}|h-\mu|}\,d\mu\lesssim \tfrac{1}{|\vk|}\|q\|_{F_\kappa^{\frac12}(h)}^2\|q\|_{E_\sigma^\sigma}^2\|q\|_{E_{2\sigma,\vk}^\sigma}^2.
\end{align*}
Together these, yield
\begin{align}
\int \left|\int J_2\,\psi_\mu^{24}\,dx\right|\,e^{-\frac1{200}|h-\mu|}\,d\mu\lesssim \|q\|_{E_\sigma^\sigma}^2\|q\|_{F_\kappa^{\frac12}(h)}^2\Bigl[|\vk|^{ - 1}\|q\|_{E_{2\sigma,\vk}^\sigma}^2 + \tfrac\kappa{\kappa^2 + \vk^2}\|q\|_{E_\sigma^\sigma}^2\Bigr].\label{J2}
\end{align}

The estimate \eqref{j diff geq 4 bound} now follows from combining \eqref{j diff geq 4 perturbative}, \eqref{J1}, and \eqref{J2}.
\end{proof}

\begin{proof}[Proof of Proposition~\ref{p: diff local smoothing}]
Recall \eqref{IBP diff}. Using Lemma~\ref{L:rho} and \eqref{E int}, we estimate
\begin{equation}\label{rho 1228}
\left|\int_{I_\kappa}\int \Im\int\rho(\vk)\phi_\mu\,dx\Big|_{t=-1}^{t=1}\,e^{-\frac1{200}|h-\mu|}\,d\mu\,d\vk\right|\lesssim\|q\|_{L^\infty_tE_{\sigma}^\sigma}^2\Bigl[1 + \|q\|_{L^\infty_tL^2}^2\Bigr].
\end{equation}

Turning to $\LHS{IBP diff}$ and recalling that \(u(\kappa) = \tfrac q{4\kappa^2 - \p^2}\), we use \eqref{j diff} and \eqref{linear} to write
\begin{align}\label{j2diff}
\Im j_\diff\sbrack{2}(\varkappa, \kappa) &= \Re\Big\{16\kappa^2 \tfrac{u(\kappa)''}{4\vk^2 - \p^2}\bar u(\kappa)''+2\tfrac{u(\kappa)'''}{4\vk^2 - \p^2} \bar u(\kappa)'''\Big\}\notag\\
&\quad - \Re\Bigl\{8\kappa^2\tfrac{u(\kappa)''}{4\vk^2 - \p^2}\bar u(\kappa)'+\tfrac{u(\kappa)'''}{4\vk^2 - \p^2} \bar u(\kappa)''\Bigr\}'\notag\\
&\quad - 8\kappa^2\Re\Bigl\{\tfrac{u(\kappa)'}{4\vk^2 - \p^2}\bar u(\kappa)'\Bigr\}'' + 4\kappa^2\Re\Bigl\{\tfrac{u(\kappa)'}{4\vk^2 - \p^2}\bar u(\kappa)\Bigr\}'''.
\end{align}
Integrating by parts and then applying \eqref{E loc} we may bound the contribution of the last three summands as follows:
\begin{align*}
&\left|\int \Bigl\{8\kappa^2\tfrac{u(\kappa)''}{4\vk^2 - \p^2}\bar u(\kappa)'+\tfrac{u(\kappa)'''}{4\vk^2 - \p^2} \bar u(\kappa)''\Bigr\}'\,\psi_\mu^{24}\,dx\right|\\
&\qquad\lesssim \kappa^2\vk^{-1} \|u(\kappa)'\|_{E^{\frac12}_{1,\vk}}\|(\psi_\mu^{24})'u(\kappa)'\|_{E^{\frac12}_{1,\vk}} +\vk^{-1} \|u(\kappa)''\|_{E^{\frac12}_{1,\vk}}\|(\psi_\mu^{24})'u(\kappa)''\|_{E^{\frac12}_{1,\vk}}\\
&\qquad\lesssim \vk^{-1}\|q\|_{E_{2\sigma,\vk}^\sigma}^2,\\
&\left|\int \Bigl\{8\kappa^2\tfrac{u(\kappa)'}{4\vk^2 - \p^2}\bar u(\kappa)'\Bigr\}''\,\psi_\mu^{24}\,dx\right| \\
&\qquad\lesssim \kappa^2\vk^{-1} \|u(\kappa)\|_{E_{1,\vk}^{\frac12}}\|(\psi_\mu^{24})'' u(\kappa)'\|_{E_{1,\vk}^{\frac12}}\lesssim \kappa^{-1}\vk^{-1}\|q\|_{E_{2\sigma,\vk}^\sigma}^2,\\
&\left|\int \Bigl\{4\kappa^2\tfrac{u(\kappa)'}{4\vk^2 - \p^2}\bar u(\kappa)\Bigr\}'''\,\psi_\mu^{24}\,dx\right|\\
&\qquad\lesssim \kappa^2\vk^{-1}\|u(\kappa)\|_{E_{1,\vk}^{\frac12}}\|(\psi_\mu^{24})'''u(\kappa)\|_{E_{1,\vk}^{\frac12}}\lesssim \kappa^{-2}\vk^{-1}\|q\|_{E_{2\sigma,\vk}^\sigma}^2.
\end{align*}
For the remaining term, we write
\begin{align*}
&\int \Bigl\{16\kappa^2 \tfrac{u(\kappa)''}{4\vk^2 - \p^2}\bar u(\kappa)''+2\tfrac{u(\kappa)'''}{4\vk^2 - \p^2} \bar u(\kappa)'''\Bigr\}\,\psi_\mu^{24}\,dx\\
&\qquad= \int \Bigl\{16\kappa^2 \tfrac{(\psi_\mu^{12}\,q)''}{(4\vk^2 - \p^2)(4\kappa^2 - \p^2)}\tfrac{(\psi_\mu^{24}\,\bar q)''}{4\kappa^2 - \p^2} + 2\tfrac{(\psi_\mu^{12}\,q)'''}{(4\vk^2 - \p^2)(4\kappa^2 - \p^2)}\tfrac{(\psi_\mu^{12}\,\bar q)'''}{4\kappa^2 - \p^2}\Bigr\}\,dx\\
&\qquad\quad + \int \Bigl\{16\kappa^2 \bigl[\psi_\mu^{12},\tfrac{\p^2}{(4\vk^2 - \p^2)(4\kappa^2 - \p^2)}\bigr]q\cdot\tfrac{(\psi_\mu^{24}\,\bar q)''}{4\kappa^2 - \p^2}\Bigr\}\,dx\\
&\qquad\quad + \int \Bigl\{2\bigl[\psi_\mu^{12},\tfrac{\p^3}{(4\vk^2 - \p^2)(4\kappa^2 - \p^2)}\bigr]\cdot\tfrac{(\psi_\mu^{12}\,\bar q)'''}{4\kappa^2 - \p^2}\Bigr\}\,dx\\
&\qquad\quad + \int \Bigl\{16\kappa^2 \tfrac{u(\kappa)''}{4\vk^2 - \p^2}q\cdot\bigl[\psi_\mu^{12},\tfrac{\p^2}{4\kappa^2 - \p^2}\bigr]\bar q + 2\tfrac{u(\kappa)'''}{4\vk^2 - \p^2}\cdot\bigl[\psi_\mu^{12},\tfrac{\p^3}{4\kappa^2 - \p^2}\bigr]\bar q\Bigr\}\,\psi_\mu^{12}\,dx.
\end{align*}
Using that
\begin{align*}
\bigl[\psi_\mu^{12},\tfrac1{4\vk^2 - \p^2}\bigr] &= -\tfrac1{4\vk^2 - \p^2}\bigl\{2(\psi_\mu^{12})'\p + (\psi_\mu^{12})''\bigr\}\tfrac1{4\vk^2 - \p^2},\\
\bigl[\psi_\mu^{12},\tfrac{\p^2}{4\kappa^2 - \p^2}\bigr] &= - \tfrac{4\kappa^2}{4\kappa^2 - \p^2}\bigl\{2(\psi_\mu^{12})'\p + (\psi_\mu^{12})''\bigr\}\tfrac1{4\kappa^2 - \p^2},\\
\bigl[\psi_\mu^{12},\tfrac{\p^3}{4\kappa^2 - \p^2}\bigr] &= - \tfrac{4\kappa^2}{4\kappa^2 - \p^2}\bigl\{2(\psi_\mu^{12})'\p + (\psi_\mu^{12})''\bigr\}\tfrac\p{4\kappa^2 - \p^2}- \tfrac{\p^2}{4\kappa^2 - \p^2}(\psi_\mu^{12})',
\end{align*}
we may again use \eqref{E loc} to bound
\begin{align*}
&\left|\int \Bigl\{16\kappa^2 \bigl[\psi_\mu^{12},\tfrac{\p^2}{(4\vk^2 - \p^2)(4\kappa^2 - \p^2)}\bigr]q\cdot\tfrac{(\psi_\mu^{24}\,\bar q)''}{4\kappa^2 - \p^2}\Bigr\}\,dx\right|\\
&\quad\lesssim \kappa^2\Biggl\{\vk\Bigl\|\bigl[\psi_\mu^{12},\tfrac1{4\vk^2 - \p^2}\bigr]u(\kappa)''\Bigr\|_{E_{-1,\vk}^{\frac12}} + \vk^{-1}\Bigl\|\bigl[\psi_\mu^{12},\tfrac{\p^2}{4\kappa^2 - \p^2}]q\Bigr\|_{E_{1,\vk}^{\frac12}}\Biggr\}\Bigl\|\tfrac{(\psi_\mu^{24}\,\bar q)'}{4\kappa^2 - \p^2}\Bigr\|_{E_{1,\vk}^{\frac12}}\\
&\quad\lesssim \vk^{-1}\|q\|_{E_{2\sigma,\vk}^\sigma}^2,\\
&\left|\int \Bigl\{2\bigl[\psi_\mu^{12},\tfrac{\p^3}{(4\vk^2 - \p^2)(4\kappa^2 - \p^2)}\bigr]q\cdot\tfrac{(\psi_\mu^{12}\,\bar q)'''}{4\kappa^2 - \p^2}\Bigr\}\,dx\right|\\
&\quad\lesssim \Biggl\{\vk\Bigl\|\bigl[\psi_\mu^{12},\tfrac1{4\vk^2 - \p^2}\bigr]u(\kappa)'''\Bigr\|_{E_{-1,\vk}^{\frac12}} + \vk^{-1}\Bigl\|\bigl[\psi_\mu^{12},\tfrac{\p^3}{4\kappa^2 - \p^2}]q\Bigr\|_{E_{1,\vk}^{\frac12}}\Biggr\}\Bigl\|\tfrac{(\psi_\mu^{24}\,\bar q)''}{4\kappa^2 - \p^2}\Bigr\|_{E_{1,\vk}^{\frac12}}\\
&\quad\lesssim \vk^{-1}\|q\|_{E_{2\sigma,\vk}^\sigma}^2,\\
&\left|\int \Bigl\{16\kappa^2 \tfrac{u(\kappa)''}{4\vk^2 - \p^2}q\cdot\bigl[\psi_\mu^{12},\tfrac{\p^2}{4\kappa^2 - \p^2}\bigr]\bar q\Bigr\}\,\psi_\mu^{12}\,dx\right|\\
&\quad\lesssim \kappa^2\vk^{-1}\|u(\kappa)'\|_{E_{1,\vk}^{\frac12}}\Bigl\|\psi_\mu^{12}\bigl[\psi_\mu^{12},\tfrac{\p^2}{4\kappa^2 - \p^2}\bigr]q\Bigr\|_{E_{1,\vk}^{\frac12}}\lesssim \vk^{-1}\|q\|_{E_{2\sigma,\vk}^\sigma}^2,\\
&\left|\int \Bigl\{2\tfrac{u(\kappa)'''}{4\vk^2 - \p^2}\cdot\bigl[\psi_\mu^{12},\tfrac{\p^3}{4\kappa^2 - \p^2}\bigr]\bar q\Bigr\}\,\psi_\mu^{12}\,dx\right|\\
&\quad\lesssim \vk^{-1}\|u(\kappa)''\|_{E_{1,\vk}^{\frac12}}\Bigl\|\psi_\mu^{12}\bigl[\psi_\mu^{12},\tfrac{\p^3}{4\kappa^2 - \p^2}\bigr]q\Bigr\|_{E_{1,\vk}^{\frac12}}\lesssim \vk^{-1}\|q\|_{E_{2\sigma,\vk}^\sigma}^2.
\end{align*}
Combining these estimates gives us
\begin{equation}\label{j diff quadratic}
\vk^2\Bigl\|\tfrac{\psi_\mu^{12} q}{\sqrt{4\kappa^2 - \p^2}}\Bigr\|_{E_{1,\vk}^2}^2 \lesssim \left|\Im \int j_{\diff}(\vk,\kappa)\sbrack 2\,\psi_\mu^{24}\,dx\right| + \vk^{-1}\|q\|_{E_{2\sigma,\vk}^\sigma}^2.
\end{equation}

Collecting \eqref{rho 1228}, \eqref{j diff quadratic}, \eqref{j diff geq 4 bound}, and using \eqref{E int 2}, we have
\begin{align*}
&\int\Bigl\|\tfrac{\psi_\mu^{12}q}{\sqrt{4\kappa^2 - \p^2}}\Bigr\|_{L^2_tE_{\frac12}^2}^2\,e^{-\frac1{200}|h-\mu|}\,d\mu\\
&\qquad\lesssim \Biggl|\int_{-1}^1\int_{I_\kappa}\int\Im j_{\diff}\sbrack{2}(\vk,\kappa)\,\psi_\mu^{12}\,dx\,e^{-\frac1{200}|h-\mu|}\,d\mu\,d\vk\,dt\Biggr|+ \|q\|_{L^\infty_tE_{\sigma}^\sigma}^2\\
&\qquad\lesssim \|q\|_{L_t^\infty E_\sigma^\sigma}^2\Bigl[ 1+ \|q\|_{L^\infty_tL_x^2}^2 + \|q\|_{X^{\frac12}_\kappa}^2\Bigr].
\end{align*}
Estimating
\begin{align*}
\bigl\|\tfrac{\psi_\mu^{12}q}{\sqrt{4\kappa^2 - \p^2}}\bigr\|_{L^2_tH^{\frac32}}^2 &\lesssim \bigl\|\tfrac{P_{\leq 1}(\psi_\mu^{12}q)}{\sqrt{4\kappa^2 - \p^2}}\bigr\|_{L^2_tH^{\frac32}}^2 + \bigl\|\tfrac{P_{>1}(\psi_\mu^{12}q)}{\sqrt{4\kappa^2 - \p^2}}\bigr\|_{L^2_tH^{\frac32}}^2\\
&\lesssim \kappa^{-2}\|q\|_{L^\infty_tL^2_x}^2 + \Bigl\|\tfrac{\psi_\mu^{12}q}{\sqrt{4\kappa^2 - \p^2}}\Bigr\|_{L^2_tE_{\frac12}^2}^2,
\end{align*}
we then get
\[
\|q\|_{X_\kappa^{\frac12}}^2\lesssim \|q\|_{L^\infty_tE_\sigma^\sigma}^2\|q\|_{X_\kappa^{\frac12}}^2 + \bigl(\|q\|_{L^\infty_tE_\sigma^\sigma}^2+\kappa^{-2}\bigr)\bigl(1+\|q\|_{L^\infty_tL^2_x}^2\bigr) .
\]
Using \eqref{goodness} to absorb the first term on the right-hand side into the left-hand side, \eqref{degenerate local smoothing} then follows from the conservation of the $L^2$ norm.
\end{proof}

\section{Convergence of the difference flow}\label{S:Convergence}

Our main goal in this section is to prove that as $\kappa\to \infty$, the flow determined by the difference of the Hamiltonians $H_\kappa^{\rm{diff}} = H-H_\kappa$ converges to the identity, locally in spacetime, uniformly over $L^2$-bounded and equicontinuous sets.

\begin{theorem}[Difference flow converges to the identity] \label{T: diff flow convergence}
Let $Q\subset \Schw(\R)$ be an $L^2$-bounded and equicontinuous set such that
\begin{align*}
        Q_{*}&=\bigl\{e^{tJ\nabla H_\kappa^{\rm{diff}}} q: q\in Q,\  |t|\leq 1, \text{ and } \kappa\geq2 \bigr\}
\end{align*}
is a $\delta$-good set for a sufficiently small $\delta>0$. Then 
\begin{align*}
    \lim_{\kappa\to\infty}\sup_{q\in Q}\,\sup_{\mu\in\mathbb R}\,\sup_{|t|\leq 1}\,\bigl\|\psi_\mu^{12} e^{tJ\nabla H_\kappa^{\rm{diff}}} q-\psi_\mu^{12} q\bigr\|_{L^2_x}=0.
\end{align*}
\end{theorem}

\begin{proof}
The $L^2$-boundedness and equicontinuity of the set $Q_*$ extends readily to the set
$$
\{\psi_\mu^{j} q: q\in Q_{*},\ 1\leq|j|\leq 12,\ \mu\in\mathbb R\}.
$$
In view of this equicontinuity property and employing the fundamental theorem of calculus, the proof of the theorem reduces to showing that
\begin{align}\label{diff flow conv}
    \lim_{\kappa\to\infty}\sup_{q\in Q}\,\sup_{\mu\in\mathbb R}\, \bigl\|\tfrac{d}{dt}\bigl(\psi_\mu^{12} e^{tJ\nabla H_\kappa^{\rm{diff}}} q\bigr)\bigr\|_{L^1_t([-1,1];H^{-5}_x)}=0.
\end{align}
A quick computation reveals that
\begin{align*}
i\tfrac{d}{dt}\bigl(\psi_\mu^{12} e^{tJ\nabla H_\kappa^{\rm{diff}}} q\bigr) = \psi_\mu^{12}\bigl[F_\kappa(e^{tJ\nabla H_\kappa^{\rm{diff}}} q\bigr)\bigr]'
\end{align*}
where 
\begin{align*}
    F_\kappa(q)=-q'-i|q|^2q+ 2\kappa \bigl[\sqrt\kappa g_{12}(\kappa)- \sqrt{-\kappa} g_{12}(-\kappa)\bigr].
\end{align*}
Thus, \eqref{diff flow conv} will follow from
\begin{align}\label{diff flow conv'}
    \lim_{\kappa\to\infty}\sup_{q\in Q_*}\,\sup_{\mu\in\mathbb R}\,\bigl\|\psi_\mu^{12} F_\kappa(q)\bigr\|_{L^1_t([-1,1];H^{-4}_x)}=0.
\end{align}

Employing \eqref{linear} and \eqref{g12 3 expansion}, we decompose
\begin{align*}
F_\kappa(q) &= \tfrac{\partial^3}{4\kappa^2- \partial^2}q + m\bigl[f,f,\tfrac{f''}{4\kappa^2-\p^2} \bigr] + m\bigl[f,\tfrac {f'}{2\kappa-\p},\tfrac{f'}{2\kappa-\p} \bigr] \\
&\quad + 2\kappa \bigl[\sqrt\kappa g_{12}\sbrack{\geq 5} (\kappa)- \sqrt{-\kappa} g_{12}\sbrack{\geq 5}(-\kappa)\bigr],
\end{align*}
where $f$ satisfies \eqref{admissible} and each paraproduct $m$ lies in $S(3)$.

The contribution of the linear term is easily estimated via
\begin{align*}
    \bigl\|\psi_\mu^{12} \tfrac{\partial^3}{4\kappa^2- \partial^2}q\bigr\|_{L_t^1 H^{-4}_x}\lesssim \kappa^{-2}\|\psi_\mu^{12}\|_{H^{4}_x} \|q \|_{L_t^\infty L^2_x}
\end{align*}
which converges to $0$ as $\kappa\to\infty$, uniformly for all $q\in Q_*$ and all $\mu\in \mathbb R$, in view of the conservation of the $L^2$ norm.

To estimate the contribution of the cubic and higher order terms in $F_\kappa(q)$, we rely on the consequences of the local smoothing estimates proved in Proposition~\ref{p: diff local smoothing}.  To simplify our bounds, we introduce
\begin{align*}
    \|q\|_{\DNLS_\kappa}:=\|q\|_{L_t^\infty L^2_x}+ \|q\|_{X_\kappa^{\frac 12}}
\end{align*}
and note that $L^2$-conservation and Proposition~\ref{p: diff local smoothing} yield
\begin{align}\label{2:09}
\| e^{tJ\nabla H_\kappa^{\rm{diff}}} q\|_{\DNLS_\kappa}\lesssim \|q\|_{L^2}
\end{align}
uniformly for $\kappa\geq 2$ and $q\in Q$.

\begin{lemma}[Local smoothing estimates]\label{L:LS conseq} For $f$ satisfying \eqref{admissible}, we have 
\begin{align}
\sup_{\mu\in\mathbb R}\left\| \tfrac{\partial (\psi_\mu^4 f) }{2\kappa\pm\partial} \right\|_{L_t^2 L^4_x}&\lesssim \kappa^{-\frac 16} \|q\|_{\DNLS_\kappa}, \label{L24 deriv}\\
\sup_{\mu\in\mathbb R}\left\| \psi_\mu^4 f\right\|_{L_{t,x}^4}&\lesssim \kappa^{\frac 16} \|q\|_{\DNLS_\kappa}. \label{L4 est}
\end{align}
Moreover, the following estimate holds uniformly for $\kappa\geq 2$, $q\in Q_*$, and~$\mu\in \R:$ 
\begin{align}\label{I5}
\int_{-1}^1\|\Lambda (\psi_\mu q)\|_{\mathfrak I_5}^5\, dt &\lesssim  \kappa^{-3}\left[\kappa^{-\frac{1}{4}}\|q\|_{L^2} +\bigl\|(\psi_\mu q)_{>\kappa^{\frac 13}}\bigr\|_{L^2} \right]\|q\|_{\DNLS_\kappa}^4.
\end{align}
\end{lemma}

\begin{proof}
Decomposing into low and high frequencies and using Bernstein's inequality for frequencies $\leq \kappa^{\frac{2}{3}}$, and Proposition~\ref{p: diff local smoothing} for frequencies $>\kappa^{\frac{2}{3}}$, we may bound
\begin{align}\label{L2H1/2 est}
\left\| \psi_\mu^4 f\right\|_{L_t^2 H^{\frac12}} \lesssim \kappa^{\frac13} \|f\|_{L_t^\infty L_x^2} + \kappa^{\frac13} \bigl\| \tfrac{\psi_\mu f}{\sqrt{4\kappa^2-\p^2}} \bigr\|_{L_t^2 H^{\frac32}_x}
\lesssim \kappa^{\frac 13} \|q\|_{\DNLS_\kappa}.
\end{align}
A parallel argument yields
\begin{align*}
\left\|\tfrac{(\psi_\mu^4 f)' }{2\kappa\pm\partial} \right\|_{L_{t,x}^2}\lesssim \kappa^{-\frac 13} \|q\|_{\DNLS_\kappa}.
\end{align*}
Using this latter bound, Sobolev embedding, and interpolation, we obtain
\begin{align*}
\left\| \tfrac{\partial (\psi_\mu^4 f) }{2\kappa\pm\partial} \right\|_{L_t^2 L^4_x}^2
&\lesssim \left\| |\p|^{\frac14}\tfrac{\partial (\psi_\mu^4 f) }{2\kappa\pm\partial} \right\|_{L_{t,x}^2}^2
\lesssim \left\|\tfrac{\psi_\mu^4 f }{2\kappa\pm\partial} \right\|_{L_t^2 H^{\frac32}_x}\left\|\tfrac{(\psi_\mu^4 f)' }{2\kappa\pm\partial} \right\|_{L_{t,x}^2}
\lesssim \kappa^{-\frac13}\|q\|_{\DNLS_\kappa}^2,
\end{align*}
which settles \eqref{L24 deriv}.

Arguing similarly and using \eqref{L2H1/2 est}, we may bound
\begin{align*}
\int_{-1}^1\left\| \psi_\mu^4 f\right\|_{L_x^4}^4\, dt &\lesssim \int_{-1}^1 \left\| |\p|^{\frac14}(\psi_\mu^4 f)\right\|_{L_x^2}^4\, dt
\lesssim\left\|\psi_\mu^4 f\right\|_{L^2_tH^{\frac12}_x}^2\left\|\psi_\mu^4 f\right\|_{L_{t,x}^2}^2\lesssim \kappa^{\frac 23} \|q\|_{\DNLS_\kappa}^4.
\end{align*}

We now turn to the proof of \eqref{I5}. By the Bernstein inequality and Lemma~\ref{L:op est}, 
\begin{align*}
    &\sum_{N\leq N_0}\|\Lambda_N(q)\|_{\op} \lesssim 
    \begin{cases} 
    \tfrac{\sqrt{N_0}}{\kappa}\|q\|_{L^2}\qquad&\text{for}\,\,N_0\leq\kappa^{\frac 13}\\
    \kappa^{-\frac 5 6} \|q\|_{L^2} +\tfrac{\sqrt{N_0}}{\kappa}\|q_{>\kappa^{\frac 13}}\|_{L^2}\qquad&\text{for}\,\,\kappa^{\frac 13}<N_0\leq \kappa\\
    \tfrac{1}{\sqrt{\kappa}}\|q\|_{L^2}\qquad&\text{for}\,\,N_0>\kappa.
    \end{cases}
\end{align*}
Employing this estimate and Lemma~\ref{L:HS}, we may bound
\begin{align*}
&\|\Lambda (\psi_\mu q)\|_{\mathfrak I_6}^6= \sum_{N_1\sim N_2} \bigl\|\Lambda\bigl[(\psi_\mu q)_{N_1}\bigr]\bigr\|_{\hs}  \bigl\|\Lambda\bigl[(\psi_\mu q)_{N_2}\bigr]\bigr\|_{\hs} \Bigl[\sum_{N_3\leq N_2} \bigl\|\Lambda\bigl[(\psi_\mu q)_{N_3}\bigr]\bigr\|_{\op}\Bigr]^4\\
& \lesssim \sum_{\kappa^{\frac 1 3} \geq N_1\sim N_2} \tfrac{N^2_2}{\kappa^5} \|(\psi_\mu q)_{N_1}\|_{L^2} \|(\psi_\mu q)_{N_2}\|_{L^2} \|q\|_{L^2}^4\\
&\quad + \sum_{\kappa^{\frac 1 3} <N_1\sim N_2\leq  \kappa^{\frac 2 3}} \tfrac{1}{\kappa} \|(\psi_\mu q)_{N_1}\|_{L^2} \|(\psi_\mu q)_{N_2}\|_{L^2} \left[\kappa^{-\frac {10} {3}} \|q\|_{L^2}^4 + \tfrac{N_2^2}{\kappa^4}\bigl\|(\psi_\mu q)_{>\kappa^{\frac 13}}\bigr\|_{L^2}^4\right]\\
&\quad + \sum_{\kappa^{\frac 2 3} <N_1\sim N_2\leq  \kappa}  \! \! \tfrac{\kappa} {N_2^{3}}\left\|\tfrac{(\psi_\mu q)_{N_1}}{\sqrt{4\kappa^2-\partial^2}}\right\|_{H^{\frac 32}} \! \left\|\tfrac{(\psi_\mu q)_{N_2}}{\sqrt{4\kappa^2-\partial^2}}\right\|_{H^{\frac 32}} \! \left[\kappa^{-\frac {10} {3}} \|q\|_{L^2}^4 + \tfrac{N_2^2}{\kappa^4}\bigl\|(\psi_\mu q)_{>\kappa^{\frac 13}}\bigr\|_{L^2}^4\right]\\
&\quad + \sum_{N_1\sim N_2>  \kappa} \tfrac{1}{\kappa^2N_2^{2}} \log\left(4+ \tfrac{N_2^2}{\kappa^2}\right) \left\|\tfrac{(\psi_\mu q)_{N_1}}{\sqrt{4\kappa^2-\partial^2}}\right\|_{H^{\frac 32} } \left\|\tfrac{(\psi_\mu q)_{N_2}}{\sqrt{4\kappa^2-\partial^2}}\right\|_{H^{\frac 32} } \|q\|_{L^2}^4\\
&\lesssim \left[\kappa^{-\frac{12}{3}}\|q\|_{L^2}^2 + \kappa^{-\frac{11}{3}} \bigl\|(\psi_\mu q)_{>\kappa^{\frac 13}}\bigr\|_{L^2}^2 \right] \|q\|_{L^2}^2 \left[\|q\|_{L^2}^2+ \left\|\tfrac{\psi_\mu q}{\sqrt{4\kappa^2-\partial^2}}\right\|_{H^{\frac 32} }^2 \right].
\end{align*}
The estimate \eqref{I5} now follows by interpolation between this bound and
$$
\|\Lambda (\psi_\mu q)\|_{\hs}\lesssim \kappa^{-\frac12}\|q\|_{L^2}. 
$$
This completes the proof of the lemma.
\end{proof}

Returning to the contribution of the cubic terms in $F_\kappa(q)$ to LHS\eqref{diff flow conv'}, we employ Lemmas~\ref{l:paraproperties} and~\ref{L:LS conseq} as well as \eqref{LS alt} to estimate
\begin{align*}
&\bigl\|\psi_\mu^{12} m\bigl[f,f,\tfrac{f''}{4\kappa^2-\p^2} \bigr]\bigr\|_{L^1_t([-1,1];H^{-4}_x)} \\
&\qquad\lesssim \bigl\|\psi_\mu^{12} m\bigl[f,f,\tfrac{f''}{4\kappa^2-\p^2} \bigr]\bigr\|_{L^1_{t,x}}\lesssim \|\psi_\mu^4f\|_{L_{t,x}^4}^2\bigl\| \tfrac{(\psi_\mu f)''}{4\kappa^2-\p^2} \bigr\|_{L_{t,x}^2}\lesssim \kappa^{\frac13}\kappa^{-\frac12}\|q\|_{\DNLS_\kappa}^3,\\
&\bigl\|\psi_\mu^{12} m\bigl[f,\tfrac {f'}{2\kappa-\p},\tfrac{f'}{2\kappa-\p} \bigr]\bigr\|_{L^1_t([-1,1];H^{-4}_x)} \\
&\qquad\lesssim \bigl\|\psi_\mu^{12} m\bigl[f,\tfrac {f'}{2\kappa-\p},\tfrac{f'}{2\kappa-\p} \bigr]\bigr\|_{L^1_{t,x}}\lesssim \|\psi_\mu^4f\|_{L_t^\infty L_x^2}\bigl\| \tfrac{(\psi_\mu f)'}{2\kappa-\p} \bigr\|_{L_t^2L_x^4}\lesssim \kappa^{-\frac13}\|q\|_{\DNLS_\kappa}^3.
\end{align*}
In view of \eqref{2:09}, the contribution of these cubic terms is acceptable.

To estimate the contribution of the quintic and higher order terms in $F_\kappa(q)$, we argue by duality.  To this end, fix $w\in H^4_x$.
Using Lemma~\ref{l:mult comm} and \eqref{I5}, we estimate
\begin{align*}
\int_{-1}^1\biggl|  \int  \psi_\mu^{12} \kappa^{\frac32} g_{12}\sbrack{\geq 5} (\kappa) \,w \,dx \biggr|\,dt
&\lesssim \sum_{\ell\geq 2}\kappa^4\|\Lambda(\psi_\mu q)\|_{\mathfrak I_5}^5 \bigl(\sqrt\kappa\|\Lambda(q)\|_{\op}\bigr)^{2\ell-4}  \|\Lambda(w\psi_\mu^7)\|_{\op}\\
&\lesssim \|w\|_{L^\infty} \left[\kappa^{-\frac{1}{4}}\|q\|_{L^2} + \bigl\|(\psi_\mu q)_{>\kappa^{\frac 13}}\bigr\|_{L^2} \right] \|q\|_{\DNLS_\kappa}^4,
\end{align*}
where we used Corollary~\ref{C:large kappa} combined with the fact that $Q_*$ is $\delta$-good in order to sum in $\ell\geq 2$.  By \eqref{2:09} and equicontinuity, the contribution of these terms to LHS\eqref{diff flow conv'} is also acceptable.
\end{proof}

\section{Proofs of the main theorems} \label{S:end}

All the main difficulties have already been addressed in previous sections, albeit under the assumption that the solutions remain in a $\delta$-good set for some universally small $\delta$ and only for the time interval $[-1,1]$.  In this section, we put all the pieces together and show how to circumvent these illusory restrictions.    

\begin{proof}[Proof of Theorem~\ref{t:main}]
The fundamental question to settle is this: Given an $L^2$-Cauchy sequence of initial data $q_n(0)\in\Schw$ and a Cauchy sequence of times $t_n\in\R$, show that their evolutions $q_n(t_n)$ under \eqref{DNLS} form an $L^2$-Cauchy sequence.

Evidently, the set $\{q_n(0)\}$ is $L^2$-precompact.  Thus by Corollary~\ref{C:descendants}, there is a uniform rescaling parameter $\lambda$ so that not only are the rescaled initial data $\delta$-good, but so are their evolutions under \eqref{DNLS} as well as any other dynamics preserving $A(\vk;q)$.  This rescaling does not meaningfully alter our original ambition --- we just replace the original sequences of solutions and times by their rescaled values (for which we reuse the original names).

It suffices to treat the case where $|t_n|\leq 1$, because larger values can be treated by iterating the argument.  For example, if $t_n\to\frac32$, then we first run the argument with $t_n\equiv 1$ and then use $q_n(1)$, which we now know to be convergent, as initial data to extend up to the chosen $t_n\to\frac32$.

Assuming now that $|t_n|\leq 1$,  Theorem~\ref{T:HGKV:equi} guarantees that $\{q_n(t_n)\}$ is equicontinuous and Proposition~\ref{P:tight} guarantees that it is tight.  Thus every subsequence has an $L^2$-convergent subsequence; we just need to verify that all such subsequential limits agree.  For this purpose, it suffices to test against some fixed $w\in C^\infty_c(\R)$.

It is at this moment that we employ the commutativity of \eqref{DNLS} and the $H_\kappa$ flows.  Using \eqref{E:comm and diff}, our task is reduced to verifying the following two claims:
\begin{gather}
\sup_{\kappa\geq 1} \ \limsup_{m,n\to\infty} \ \bigl| \big\langle w,  e^{ t_n J \nabla H_\kappa} q_n(0) - e^{ t_m J \nabla H_\kappa} q_m(0) \bigr\rangle\bigr| = 0 \\
\limsup_{\kappa\to\infty} \ \sup_{q_0 \in Q_*} \ \sup_{|t|\leq 1} \ \bigl| \big\langle w,  \bigl[ e^{ t J \nabla (H-H_\kappa)} -\Id \bigr] (q_0) \bigr\rangle\bigr| = 0
\end{gather}
where $Q_*$ is defined via \eqref{QsI} with the choice $Q=\{q_n(0):n\in\N\}$.

The first of these two claims follows from the $L^2$-wellposedness of the $H_\kappa$ flow shown already in \cite[Cor.~5.4]{KNV}.  The second was addressed by Theorem~\ref{T: diff flow convergence}.
\end{proof}

\begin{proof}[Proof of Corollary~\ref{C:main}] Recall that local well-posedness for $s\geq \frac12$ was proved already by Takaoka in \cite{MR1693278}.  This result is rendered global by the a priori bounds shown in \cite{BLP,BP}.

Consider now $0\leq s <\frac12$.  Evidently, the existence of solutions follows immediately from Theorem~\ref{t:main}, as does continuous dependence in the $L^2$ metric.  Continuous dependence in the $H^s$ metric follows from this and $H^s$-equicontinuity, which was shown in \cite[Th.~5.6]{KNV} contingent on the equicontinuity conjecture that was subsequently resolved in \cite{HGKV:equi}.
\end{proof}

\begin{proof}[Proof of Theorem~\ref{t: local smoothing}] By Corollary~\ref{C:descendants}, there is a uniform rescaling of the set of initial data $Q$ so that not only are the rescaled initial data $\delta$-good, but so are their entire \eqref{DNLS} evolutions.  This allows us to invoke \eqref{local smoothing} from Proposition~\ref{p:LS} to obtain the local smoothing estimate (over any unit time interval) for the rescaled solutions.  The estimate \eqref{E:naiveLS} for the unrescaled solutions follows by a simple covering argument.  Naturally the resulting constant depends on the rescaling parameter; however, this is dictated solely by $Q$. 
\end{proof}

\begin{proof}[Proof of Corollary~\ref{C:distrib}] Given initial data $q(0)\in L^2$, we choose Schwartz-class initial data $q_n(0)$ that converge to it in $L^2$.  By Theorem~\ref{t:main}, the solutions $q_n(t)$ converge to $q(t)$ in $C([-T,T];L^2(\R))$.  To deduce that $q(t)$ is a distributional solution, we need another form of convergence to handle the nonlinearity.  This can be obtained with the aid of Theorem~\ref{t: local smoothing} and a Gagliardo--Nirenberg inequality:
$$
\| \psi^{8} [q_n -q] \|_{L^3_{t,x}}^3
\lesssim \| q_n -q \|_{L^\infty_t L^2_x}
	\Bigl\{\|\psi^{12}q_n\|_{L^{2\vphantom/}_t H^{1/2}_{x\vphantom t}}^2 +  \|\psi^{12}q\|_{L^{2\vphantom/}_t H^{1/2}_{x\vphantom t}}^2 \Bigr\}
\to 0\ \ \text{as $n\to\infty$}.
$$
Here all norms are taken over the spacetime slab $[-T,T]\times\R$.
\end{proof}

\bibliographystyle{habbrv}
\bibliography{refs}

\end{document}